\documentclass[a4paper]{article}

\usepackage[utf8]{inputenc}
\usepackage[width=15cm,height=22cm]{geometry}
\usepackage{amsmath,amsfonts,amssymb,amsthm}

\usepackage[hidelinks,pdfusetitle]{hyperref}
\hypersetup{bookmarksdepth=subsection}
\usepackage[nameinlink,capitalise,noabbrev]{cleveref}
\hypersetup{
  colorlinks=true,
  linkcolor=blue!70!black,
  citecolor=red!70!black,
  urlcolor=blue!50!black,
}

\usepackage{tocloft}
\usepackage[usenames,dvipsnames]{pstricks}
\usepackage{bbm}
\usepackage{stmaryrd}
\usepackage{enumitem}
\usepackage{etoolbox}
\usepackage{array}
\usepackage{booktabs}

\crefname{equation}{}{}

\newcounter{stepcounter}
\newcommand{\step}[1]{%
  \ifnum\value{stepcounter}>0 \medskip\fi
  \refstepcounter{stepcounter}\noindent\emph{Step \thestepcounter: #1.}}
\AtBeginEnvironment{proof}{\setcounter{stepcounter}{0}}

\newtheorem{proposition}{Proposition}[section]
\newtheorem{theorem}[proposition]{Theorem}
\newtheorem{open}[proposition]{Open problem}
\newtheorem{definition}[proposition]{Definition}
\newtheorem{corollary}[proposition]{Corollary}
\newtheorem{lemma}[proposition]{Lemma}
\newtheorem{remark}[proposition]{Remark}
\numberwithin{equation}{section}
\newtheorem{example}[proposition]{Example}

\title{Some quartic control results for scalar-input systems}

\author{
Karine Beauchard\texorpdfstring{\thanks{Univ Rennes, CNRS, IRMAR - UMR 6625, F-35000 Rennes, France}}{},
Fr\'ed\'eric Marbach\texorpdfstring{\thanks{DMA, École normale supérieure, Université PSL, CNRS, 75005 Paris, France}}{}
}

\newcommand{\N}{\mathbb{N}}

\newcommand{\R}{\mathbb{R}}
\newcommand{\Z}{\mathbb{Z}}
\newcommand{\dd}{\,\mathrm{d}}
\newcommand{\q}{\mathfrak{q}}
\newcommand{\CC}{\mathcal{C}}
\newcommand{\bb}{\mathfrak{b}}
\newcommand{\Bs}{\mathcal{B}^\star}
\newcommand{\Bbad}{\mathcal{B}^\star_{4,\textrm{bad}}}
\newcommand{\Bgood}{\mathcal{B}^\star_{4,\textrm{good}}}
\newcommand{\lone}{L^1(0,t)}
\newcommand{\lloc}{L^1_{\mathrm{loc}}(\R_+)}
\newcommand{\eval}{\textnormal{\textsc{e}}}
\DeclareMathOperator{\Br}{Br}
\DeclareMathOperator{\ad}{ad}
\DeclareMathOperator{\sign}{sign}
\DeclareMathOperator{\vect}{span}
\newcommand{\GL}{\operatorname{GL}}

\newcommand{\Lie}{\operatorname{Lie}}
\newcommand{\card}{\operatorname{card}}
\newcommand{\supp}{\operatorname{supp}}

\newcommand{\intset}[1]{\llbracket #1 \rrbracket}
\newcommand{\hyp}[1]{{\textbf{\ref{#1}}}}

\newcommand{\bad}{\mathrm{bad}}
\newcommand{\good}{\mathrm{good}}

\begin{document}

\maketitle

\begin{abstract}
    We investigate the role of quartic terms in the small-time local controllability of scalar-input systems. 
    First, we prove a new sufficient condition for controllability, which exploits simultaneously more good quartic Lie brackets than previous results.
    Second, we identify a family of quartic obstructions to controllability, relying on Lie brackets whose coordinates of the second kind are positive-definite functionals of the control.
    Third, we show that the complementarity of these results can be seen as an answer to Kawski's 1987 open problem concerning the construction of a Hall basis which somehow separates good and bad quartic brackets.
    We give examples and remarks illustrating some of the intricacies of these notions.
\end{abstract}

\setcounter{tocdepth}{1}
\tableofcontents

\newpage

\part{Smooth-STLC}
\label{part:smooth}

\section{Introduction}

\subsection{Scalar-input control-affine systems}

In this article, we consider a scalar-input affine control system
\begin{equation} \label{syst}
\dot{x}(t)=f_0(x(t))+u(t) f_1(x(t)) 
\end{equation}
where the state $x(t)$ belongs to $\R^d$ ($d \geq 1$), the control is a scalar input $u(t) \in \R$, $f_0$ and $f_1$ are vector fields on $\R^d$, real-analytic on a neighborhood of $0$, such that $f_0(0)=0$.
These are standing assumptions throughout the paper.
Nevertheless, we expect that the analyticity assumption can be removed using the arguments of \cite[Section 11]{BeauchardMarbach2026}.

For each $T > 0$ and $u \in L^1((0,T);\R)$, the Cauchy problem \eqref{syst} with initial condition $x(0)=0$ admits a unique maximal absolutely continuous solution, denoted by $x(\cdot;u)$.
We will consider small enough controls and small enough times so that this solution is defined up to time $T$.

\subsection{Small-time local controllability}

In this article, we study the small-time local controllability of system \eqref{syst} in the sense of \cref{def:State-STLC,def:WmSTLC,def:SSTLC} below, which require the following notions.

For $T>0$ and $m \in \N$, we consider the usual Sobolev space $W^{m,\infty}(0,T)$ equipped with its usual norm
$\|u\|_{W^{m,\infty}}:=\|u\|_{L^\infty} + \dotsb + \|u^{(m)}\|_{L^\infty}$, and $W^{m,\infty}_0(0,T)$ its subspace of functions vanishing at $0$ and $T$, along with their derivatives up to order $m-1$.
We denote by $u_j : (0,T) \to \R$ for $j \in \N$ the iterated primitives of $u$ defined by induction by $u_0 := u$ and $u_{j+1}(t) = \int_0^t u_j$.
We let
\begin{equation} \label{def:norm_W-1}
    \|u\|_{W^{-1,\infty}} := \|u_1\|_{L^\infty}.
\end{equation}
For scalar-input systems such as \eqref{syst}, this norm measures the size of the state (see \cref{p:small-state} and \cite[Lemma~20]{BeauchardMarbach2018}).
For convenience, we define $W^{-1,\infty}(0,T)$ as the set $L^1(0,T)$ equipped with the norm~\eqref{def:norm_W-1} and we let $W^{-1,\infty}_0(0,T) := W^{-1,\infty}(0,T)$.

\begin{definition}[Small-State-STLC] \label{def:State-STLC}
    We say that system \eqref{syst} is \emph{small-state-STLC} when, for all $T,\rho>0$, there exists $\delta>0$ such that, for all $x^\star \in B(0,\delta)$, there exists $u\in L^1(0,T)$ such that $x(T;u)=x^\star$ and $x([0,T];u) \subset B(0,\rho)$.
\end{definition}

\begin{definition}[Rough-STLC] \label{def:WmSTLC}
    Let $m \in \llbracket -1 , \infty \llbracket$.
    We say that system \eqref{syst} is \emph{$W^{m,\infty}$-STLC} (resp.\ $W^{m,\infty}_0$-STLC) when, for all $T,\rho>0$, there exists $\delta>0$ such that, for all $x^\star \in B(0,\delta)$, there exists $u \in W^{m,\infty}(0,T)$ (resp.\ $W^{m,\infty}_0(0,T)$) such that $x(T;u)=x^\star$ and $\|u\|_{W^{m,\infty}} \leq \rho$.
\end{definition}

\begin{definition}[Smooth-STLC] \label{def:SSTLC}
    We say that system \eqref{syst} is \emph{smoothly-STLC} when it is $W^{m,\infty}_0$-STLC for all $m \in \N$.
\end{definition}

The STLC notions used in the literature usually correspond to what we refer to as $L^\infty$-STLC (i.e.\ $m=0$ in \cref{def:WmSTLC} above), where controls have to be arbitrarily small in $L^\infty$ norm (see e.g.~\cite[Definition 3.2]{Coron2007} or STLC$_\varepsilon$ in \cite{Kawski1987_Survey}). 
For any $m\in\N^*$, one has the implication chain
\begin{equation} \label{Wm-STLC_implications}
    (\text{Smooth-STLC}) \Rightarrow
    (W^{m,\infty}\text{-STLC}) \Rightarrow 
    (L^\infty\text{-STLC}) \Rightarrow 
    (W^{-1,\infty}\text{-STLC}),
\end{equation}
where any reciprocal implication is false (see \cite[Appendix A.1]{BeauchardMarbach2026}).
See also \cite{BoscainCannarsaFranceschiSigalotti2023} for a recent comparison of various controllability definitions.
One advantage of $W^{-1,\infty}$-STLC is that it is equivalent to the small-state-STLC for scalar-input systems (see \cite[Section 8.2]{BeauchardMarbach2018}).

\subsection{Algebraic notations and Lie brackets}

The STLC is closely related to the evaluations at $0$ of the iterated Lie brackets of the vector fields $f_0$ and $f_1$. 
We therefore introduce the following definitions and notations, as in \cite[Section 1.3]{BeauchardMarbach2026}.

\bigskip

Let $X := \{X_0,X_1\}$ be a set of two \emph{non-commutative indeterminates}.

\begin{definition}[Free algebra]
    \label{def:free-algebra}
    We consider $\mathcal{A}(X)$ the \emph{free algebra} generated by $X$ over the field $\R$, i.e.\ the unital associative algebra of polynomials of the indeterminates $X_0$ and $X_1$.
\end{definition}

\begin{definition}[Free Lie algebra]
    Within $\mathcal{A}(X)$ one can define the Lie bracket of two elements as $[a,b] := ab - ba$. 
    This operation is anti-symmetric and satisfies the Jacobi identity.
    Let $\mathcal{L}(X)$ be the \emph{free Lie algebra} generated by $X$ over the field $\R$, i.e.\ the smallest linear subspace of $\mathcal{A}(X)$ containing $X$ and closed under the Lie bracket $[\cdot,\cdot]$.
\end{definition}

\begin{definition}[Iterated brackets] \label{def:BR-Eval}
    Let $\Br(X)$ be the \emph{free magma over $X$}, or, more visually, the set of \emph{iterated brackets} of elements of $X$, defined by induction: $X_0, X_1 \in \Br(X)$ and if $a, b \in \Br(X)$, then the ordered pair $(a,b)$ belongs to $\Br(X)$. 

    There is a natural \emph{evaluation} mapping $\eval$ from $\Br(X)$ to $\mathcal{L}(X)$ defined by induction by $\eval(X_i) := X_i$ for $i=0,1$ and $\eval((a, b)) := [\eval(a), \eval(b)]$. 
    Through this mapping, $\Br(X)$ spans $\mathcal{L}(X)$.
\end{definition}
    
\begin{definition}[Homogeneous layers within $\mathcal{L}(X)$] 
    For $b \in \Br(X)$, $n_0(b)$ (respectively $n_1(b)$) denotes the number of occurrences of the indeterminate $X_0$ (resp.\ $X_1$) in $b$.
    For $A_1, A_0 \subset \N$, $S_{A_1}(X)$ and $S_{A_1,A_0}(X)$ are the vector subspaces of $\mathcal{L}(X)$ defined by
    \begin{align}
        \label{eq:S_A1}
        S_{A_1}(X) & := \vect\{ \eval(b) \mid b \in \Br(X), \enskip n_1(b) \in A_1 \}, \\
        \label{eq:S_A1A0}
        S_{A_1,A_0}(X) & := \vect\{\eval(b) \mid b\in\Br(X),\enskip n_1(b)\in A_1, \enskip n_0(b) \in A_0\}.
    \end{align}
    For $i,j \in \N$, we write\footnote{Some authors use the notation $S_i(X)$ for what is referred to here as $S_{\intset{1,i}}(X)$.} $S_i(X)$ and $S_{i,j}(X)$ instead of $S_{\{i\}}(X)$ and $S_{\{i\},\{j\}}(X)$.
\end{definition}

\begin{definition}[Bracket integration $b0^\nu$] \label{def:0nu}
    For $b \in \Br(X)$ and $\nu \in \N$, we use the unconventional shorthand $b 0^\nu$ to denote the right-iterated bracket $((\dotsb(b, X_0), \dotsc), X_0)$, where $X_0$ appears $\nu$ times.
\end{definition}

\begin{definition}[Lie bracket of vector fields]
    For smooth vector fields $f$ and $g$, we define
    \begin{equation}
        [f,g] := (Dg) f - (Df) g.
    \end{equation}
\end{definition}

\begin{definition}[Evaluated Lie bracket]
    \label{def:evaluated_Lie_bracket}
    Let $f_0, f_1$ be $\CC^\infty$ vector fields on an open subset $\Omega$ of $\R^d$ and $f = \{ f_0, f_1 \}$.
    For $B \in \mathcal{L}(X)$, we define $f_B:=\Lambda(B)$, where $\Lambda:\mathcal{L}(X) \to \CC^\infty(\Omega;\R^d)$ is the unique homomorphism of Lie algebras such that $\Lambda(X_0)=f_0$ and $\Lambda(X_1) = f_1$.

    To simplify the notation, we will write $f_b$ instead of $f_{\eval(b)}$ when $b \in \Br(X)$. 
    The vector field $f_b$ is obtained by replacing the indeterminates $X_i$ with the corresponding vector field $f_i$ in the formal bracket $b$. 
    For instance, if $b=(X_1,(X_0,X_1))$ then $f_{b}=[f_1,[f_0,f_1]]$ and if $B=\alpha_1 \eval(b_1) + \dots + \alpha_n \eval(b_n) \in \mathcal{L}(X)$ where $b_1,\dots,b_n \in \Br(X)$ and $\alpha_1,\dots,\alpha_n \in \R$ then $f_B=\alpha_1 f_{b_1}+\dots+\alpha_n f_{b_n}$.

    Finally, for a subset $\mathcal{N}$ of $\Br(X)$ (or of $\mathcal{L}(X)$) we use the notation
    \begin{equation}
        \mathcal{N}(f)(0) := \vect \{ f_b(0) \mid b \in \mathcal{N} \} \subset \R^d.
    \end{equation}
\end{definition}

All the known necessary conditions for STLC are stated in the following way. 
One focuses on a ``bad'' bracket $\bb \in \Br(X)$ and one identifies a subset $\mathcal{N}$ of $\Br(X)$ containing all the brackets susceptible to neutralize $\bb$.
Then the necessary condition for STLC is $f_{\bb}(0) \in \mathcal{N}(f)(0)$.

This is linked with Krener's fundamental result \cite[Theorem 1]{Krener1973}, which states that, if two control systems of the form \eqref{syst} have linearly isomorphic brackets evaluated at $0$, then they are diffeomorphic.
Thus the entire information about STLC is contained in the subset of $\R^d$ made of the evaluations at $0$ of the Lie brackets of the vector fields $f_0$ and $f_1$.

\subsection{A recently introduced basis of the free Lie algebra}

We formulate our control results using a basis of $\mathcal{L}(X)$ which we recently introduced in \cite[Section~3]{BeauchardMarbach2026}.
This basis is of the form $\eval(\Bs)$, where $\Bs$ is a Hall set of $\Br(X)$ (see \cref{s:hall-sets}).
We recall its definition in \cref{s:b-star}.
The first elements of $\Bs$ are listed in the following statement (proved in \cite[Section 3.2]{BeauchardMarbach2026}). 
The main interest of~$\Bs$ is the particular form of the associated coordinates of the second kind (see \cref{Subsec:Coord2_Bstar}), which appear to be very well suited for functional analysis.

\begin{proposition} \label{Prop:Bstar_S_14}
    The first $X_1$-homogeneous layers 
    $\Bs_k:=\{b\in\Bs \mid n_1(b)=k\}$
    of $\Bs$ are 
    \begin{align}
        \label{Bstar_S1}
        \Bs_1 & = \{ M_\nu \}, \\
        \label{Bstar_S2}
        \Bs_2 & = \{ W_{j,\nu} \}, \\
        \label{Bstar_S3}
        \Bs_3 & = \{ P_{j,k,\nu} ;  j\leq k \}, \\
        \label{Bstar_S4}
        \Bs_4 & = \{ Q_{j,k,l,\nu} ; j \leq k \leq l \} \cup \{ Q^\sharp_{j,\mu,k,\nu}; j < k \} \cup \{Q^\flat_{j,\mu,\nu} \},
    \end{align}
    where, implicitly, $j, k, l \in \N^*$, $\mu,\nu\in\N$ and we define successively
    \begin{align}
        \label{def:Mj}
        & M_\nu := X_1 0^\nu, \\
        \label{def:Wjnu}
        & W_{j,\nu} := (M_{j-1},M_j) 0^\nu, \\
        \label{def:Pljnu}
        & P_{j,k,\nu} := (M_{k-1},W_{j,0}) 0^\nu, \\
        \label{def:Q}
        & Q_{j,k,l,\nu} := (M_{l-1},P_{j,k,0}) 0^\nu,
        \quad
        Q^\sharp_{j,\mu,k,\nu} := (W_{j,\mu},W_{k,0}) 0^\nu,
        \quad
        Q^\flat_{j,\mu,\nu} := (W_{j,\mu},W_{j,\mu+1}) 0^\nu.
    \end{align}
    To lighten the notations, $W_j$, $P_{j,k}$ and $Q_{j,k,l}$ will denote $W_{j,0}$, $P_{j,k,0}$ and $Q_{j,k,l,0}$.

    Moreover, to avoid cluttering the formulas, all these symbols will indifferently denote either the elements of $\Br(X)$ themselves or their evaluation by $\eval$ in $\mathcal{L}(X)$ (recall \cref{def:BR-Eval}).

    For $A \subset \N$, we also use the notation $\Bs_A := \{ b \in \Bs \mid n_1(b) \in A \}$, e.g.\ $\Bs_{\intset{1,3} \setminus \{2\}} = \Bs_1 \cup \Bs_3$.
\end{proposition}

\subsection{Organization of the two parts of the paper}

\emph{To facilitate understanding, our paper is split into two parts.}
\begin{itemize}
    \item 
    \emph{\cref{part:smooth} concerns the case of smooth-STLC.
    In this case, the classification of good and bad brackets is satisfying, the statements and proofs are rather straightforward, and one can focus on the key ideas.
    We describe the corresponding results below.
    }

    \item 
    \emph{\cref{part:rough} concerns the case of rough-STLC, which in particular includes the usual historical $L^\infty$-STLC notion.
    In this case, the classification of good and bad brackets is less stringent, the statements are more intricate, and the proofs more technical.
    The description of our results in this case is postponed to \cref{s:rough-intro}, including our main obstruction result of \cref{thm:Qjk_intro}.
    }
\end{itemize}

\subsection{Main results in the smooth case}
\label{s:main-results}

The main contributions of this article in the smooth case are the \textbf{quartic sufficient condition} of \cref{thm:S0-B4} and the \textbf{quartic necessary condition} of \cref{thm:Qjk-smooth} whose complementarity is illustrated by the classification result of \cref{thm:classification}.
We refer the reader to \cite[Chapter 3]{Coron2007} or \cite{Kawski1987_Survey} for general surveys of the state of the art.
We define
\begin{align} 
    \label{eq:b4good}
    \Bgood & := \left\{ Q_{j,k,l,\nu} ; \enskip j \leq k < l  \right\}, \\
    \label{eq:b4bad}
    \Bbad & := \{ Q_{j,k,k,\nu} ; \enskip j \leq k \} \cup \{ Q^\sharp_{j,\mu,k,\nu}; \enskip j < k \} \cup \{ Q^\flat_{j,\mu,\nu} \},
\end{align}
where implicitly, $j,k,l \in \N^*$ and $\mu,\nu \in \N$, so that $\Bs_4$ is the disjoint union of $\Bgood$ and $\Bbad$.

\subsubsection{Good quartic brackets}

Our first contribution is to identify a large subset of quartic Lie brackets which are ``simultaneously good'' in the sense of the following results.
This can be seen as a far-reaching generalization of Kawski's observation in \cite{Kawski1987_Necessary} that the Lie bracket $Q_{1,1,3}$, of type (even, odd), can yield controllability.

The set $\Bgood$ includes both some brackets of type (even, even), which are known from Sussmann's work \cite{Sussmann1987} to behave nicely with respect to the input symmetry $u \mapsto \check{u}$ (time reversal), as well as some brackets of type (even, odd), depending on the parity of $2j+k+l+\nu-3$.

A first illustration of the ``good'' character of these brackets is the following result.

\begin{theorem} \label{thm:xi4-surj}
    Let $G$ be a finite subset of $\Bs_1 \cup \Bs_3 \cup \Bgood$ and $T > 0$.
    The following map is onto:
    \begin{equation}
        \begin{cases}
            \CC^\infty_c((0,T);\R) & \to \R^G, \\
            u & \mapsto \left( \xi_b(T,u) \right)_{b \in G},
        \end{cases}
    \end{equation}
    where $\xi_b$ denotes the coordinate of the second kind associated with $b$ (see \cref{Subsec:Coord2_Bstar}). 
    
    It admits a right inverse $R_G$ with $R_G(0) = 0$ which is continuous with values in any $W^{k,\infty}_0$.
\end{theorem}

This ``simultaneous surjectivity'' property enables one to prove positive control results of various forms (depending on the chosen dilation), such as the following one.

\begin{theorem} \label{thm:S0-B4}
    Assume that the Lie algebra rank condition $S_{\intset{1,4}}(f)(0) = \R^d$ holds and that
    \begin{align}
        \label{eq:S2-S1}
        \forall b \in \Bs_2, \quad
        & f_b(0) \in S_{1}(f)(0), \\
        \label{eq:S4b-S3}
        \forall b \in \Bbad, \quad
        & f_b(0) \in S_{\intset{1,3}}(f)(0).
    \end{align}
    Then system \eqref{syst} is smoothly-STLC.
\end{theorem}

\begin{remark}
    Assumption \eqref{eq:S2-S1} is a necessary condition for smooth-STLC (see \cite[Theorem 3]{BeauchardMarbach2018} or \cite[Theorem 1.11]{BeauchardMarbach2026}).
    Assumption \eqref{eq:S4b-S3} is however not a necessary condition for smooth-STLC. 
    An example of a smoothly-STLC system, not satisfying \eqref{eq:S4b-S3}, is given in \cref{prop:magnitude}.
    In this case, controllability stems from a competition between ``bad'' quartic brackets (in the sense of \cref{thm:Qjk-smooth}) of equal functional strength. 
\end{remark}

Many sufficient conditions for controllability are known (see the surveys \cite{Kawski1987_Survey} or \cite[Section~3.4]{Coron2007}, or the original works \cite{AgrachevGamkrelidze1993_Semigroups,BianchiniStefani1986,Hermes1982,Krastanov2009,Sussmann1983,Sussmann1987}).
We compare our result with Sussmann's $\mathcal{S}(\theta)$ condition based on dilations and time reversal (see \cite{Sussmann1987} or \cite[Theorem 3.29]{Coron2007}).

\begin{remark}[Comparison with Sussmann's condition]
    For systems satisfying $S_{\intset{1,4}}(f)(0) = \R^d$, \cref{thm:S0-B4} can be seen as an enhancement of Sussmann's $\mathcal{S}(0)$ condition:
    \begin{equation}
        \label{eq:S0}
        \forall b \in \Br(X), \quad 
        n_1(b) \text{ even and } n_0(b) \text{ odd} \enskip \Rightarrow \enskip
        f_b(0) \in \vect \{ f_a(0) \mid n_1(a) < n_1(b) \}.
    \end{equation}
    We claim that \eqref{eq:S0} implies \eqref{eq:S2-S1} and \eqref{eq:S4b-S3}.
    
    First, for any $j \geq 1$, $W_j$ is of type $(2,2j-1)$ so $f_{W_j}(0) \in S_1(f)(0)$ by \eqref{eq:S0}.
    Since $f_0(0) = 0$ this implies that, $f_{W_{j,1}}(0) = [f_{W_j}, f_0](0) \in S_1(f)(0)$ (see \cref{lem:Lie-subalg}).
    Iteratively, we obtain that $f_{W_{j,\nu}}(0) \in S_1(f)(0)$ for all $\nu \ge 0$.
    Hence \eqref{eq:S0} indirectly implies that $S_2(f)(0) \subset S_1(f)(0)$, which is explicitly required in \eqref{eq:S2-S1}.

    Second, for any $1 \leq j \leq k$, $Q_{j,k,k}$ is of type $(4, 2j+2k-3)$ so $f_{Q_{j,k,k}}(0) \in S_{\intset{1,3}}(f)(0)$ by~\eqref{eq:S0}.
    As for the $W_j$, this implies that $f_{Q_{j,k,k,\nu}}(0) \in S_{\intset{1,3}}(f)(0)$ for all $\nu \ge 0$.
    Moreover, similar cascade compensation mechanisms linked with repeated applications of \cref{lem:Lie-subalg} entail that the same neutralization happens for all brackets $Q^{\sharp}_{j,\mu,k,\nu}$ and $Q^\flat_{j,\mu,\nu}$ of \eqref{eq:b4bad} (see \cref{lem:sharp-flat-S3}). 
    
    This prompts the following comments concerning quartic brackets:
    \begin{itemize}
        \item All brackets of $\Bbad$ are considered as bad (implicitly or explicitly) by both results.
        \item Brackets of $\Bgood$ with $n_0(b)$ even are considered as good by both results.
        \item Brackets of $\Bgood$ with $n_0(b)$ odd are required to be compensated by \eqref{eq:S0}, but not in \eqref{eq:S4b-S3}.
        Our result entails that they are actually good, so that, in this sense, \cref{thm:S0-B4} is stronger than the $\mathcal{S}(0)$ sufficient condition for quartic systems.
    \end{itemize}
    As an example, consider the system involving the bracket $Q_{1,1,3} = \ad_{M_2} \ad_{X_1}^3(X_0)$ already studied by Kawski in \cite{Kawski1987_Necessary}, which is of type $(4,3)$:
    \begin{equation}
        \begin{cases}
            \dot{x}_1 = u \\
            \dot{x}_2 = x_1 \\
            \dot{x}_3 = x_2 \\
            \dot{x}_4 = x_1^3 x_3
        \end{cases}
    \end{equation}
    By \cref{p:canonical}, the only non-vanishing Lie brackets in $\Bs$ at $0$ are $f_{X_1}(0) = e_1$, $f_{M_1}(0) = e_2$, $f_{M_2}(0) = e_3$ and $f_{Q_{1,1,3}}(0) = 6 e_4$.
    Therefore, this system does not satisfy \eqref{eq:S0}, but is smoothly-STLC by \cref{thm:S0-B4}.
\end{remark}

\subsubsection{Quartic obstructions}
\label{s:intro-obs-smooth}

We investigate obstructions caused by the brackets $Q_{j,k,k}$ for $j \leq k \in \N^*$, which are of type (even, odd), so are required to be compensated both in Sussmann's $\mathcal{S}(\theta)$ condition and in \cref{thm:S0-B4}.
Our result proves that imposing some compensation condition is indeed necessary.

\medskip

For $1 \leq j \leq k$ and $M \in \intset{1,\infty}$, we define a neutralizing subset of $\Br(X)$ by
\begin{equation} \label{eq:NjkM}
    \mathcal{N}_{j,k}^M := \Bs_{\intset{1,M}} \setminus \{ Q_{j,k,k} \}.
\end{equation}
In particular, $\mathcal{N}_{j,k}^\infty$ denotes all brackets of $\Bs \setminus \{ X_0 \}$ except $Q_{j,k,k}$.

\medskip

\begin{theorem} \label{thm:Qjk-smooth}
    Let $1 \leq j \leq k$.
    If system \eqref{syst} is smoothly-STLC, then 
    \begin{equation}
        f_{Q_{j,k,k}}(0) \in \mathcal{N}_{j,k}^4(f)(0) = \vect \{ f_b(0) \mid b \in \mathcal{N}_{j,k}^4 \}.
    \end{equation}
\end{theorem}

We state and prove a more advanced version of \cref{thm:Qjk-smooth} in \cref{thm:Qjk_intro}, where the cutoff~$M$ of the neutralizing subspace \eqref{eq:NjkM} is tuned depending on the regularity $m$ to derive an optimal necessary condition for $W^{m,\infty}$-STLC.

The proof of \cref{thm:Qjk-smooth} relies on a unified approach of obstructions to STLC, introduced in \cite{BeauchardMarbach2026}, combining a new basis of the free Lie algebra (see \cref{s:b-star}), a representation formula for the state (see \cref{s:magnus}) and interpolation inequalities.
In this work, we refine this method to derive our quartic obstructions, notably relying on new interpolation inequalities derived in \cite{Marbach2023}.

\begin{example}
    \label{ex:obs-smooth}
    Consider the system involving the brackets $Q_{1,2,2}$ and $\ad_{X_1}^5(X_0)$:
    \begin{equation}
        \begin{cases}
            \dot{x}_1 = u \\
            \dot{x}_2 = x_1 \\
            \dot{x}_3 = x_1^2 x_2^2 + x_1^5
        \end{cases}
    \end{equation}
    Using \cref{p:canonical} and a change of coordinates, one can prove that the only non-vanishing Lie brackets in $\Bs$ at $0$ are $f_{X_1}(0) = e_1$, $f_{M_1}(0) = e_2$, $f_{Q_{1,2,2}}(0) = 4 e_3$ and $f_{\ad_{X_1}^5(X_0)}(0) = 5! e_3$.

    By \cref{thm:Qjk-smooth}, this system is not smoothly-STLC.
    Heuristically, the quartic term introduces a \emph{drift} in the dynamics (see \cref{s:approach-Wm}).
    More precisely, we will prove in \cref{p:interpol-Qjk-smooth} that there exists $C > 0$ such that, for all $T \in (0,1]$ and $u \in W^{2,\infty}_0((0,T);\R)$,
    \begin{equation}
        \int_0^T |u_1|^5 \leq C \| u \|_{W^{2,\infty}} \int_0^T u_1^2 u_2^2.
    \end{equation}
    Thus, if $\| u \|_{W^{2,\infty}} \leq \frac{1}{C}$, one has $x_3(T;u) \geq 0$, which rules out controllability.
\end{example}

\subsubsection{The classification problem}

In \cite[Section 4]{Kawski1987_Survey}, Kawski introduced the question of finding bases of $\mathcal{L}(X)$ that separate the homogeneous components $S_{i,j}(X)$ of $\mathcal{L}(X)$ in bad and good brackets.
We call this the ``classification problem'' and discuss its definition in \cref{s:classification-intricacies}.

Kawski states that, starting with $S_{4,5}(X)$, none of the previously known bases of $\mathcal{L}(X)$ (either the length-compatible Hall bases or the Chen-Fox-Lyndon ones) classify  $S_{4,5}(X)$, which is of dimension 14 (see \eqref{eq:witt}).
We give a detailed account of how $\Bs$ classifies $S_{4,5}(X)$ in \cref{s:S45}.

Moreover, we claim that $\Bs$ actually classifies the whole of $S_{\intset{1,4}}(X)$ as the disjoint union of the good brackets ($\Bs_1 \cup \Bs_3 \cup \Bgood$) and bad brackets ($\Bs_2 \cup \Bbad$) in the following sense.

\begin{theorem} \label{thm:classification}
    The following statements hold.
    \begin{itemize}
        \item Assume that the quartic Lie algebra rank condition $S_{\intset{1,4}}(f)(0) = \R^d$ holds and that, for all $\bb \in \Bs_2 \cup \Bbad$, $f_\bb(0) = 0$.
        Then system \eqref{syst} is smoothly-STLC.
        
        \item Assume that system \eqref{syst} is smoothly-STLC.
        Then, for all $\bb \in \Bs_2 \cup \Bbad$, 
        \begin{equation} \label{eq:fbb-comp}
            f_{\bb}(0) \in \vect \left\{ f_b(0) ; b \in \Bs_{\intset{1,4}} \setminus \{ \bb \} \right\}.
        \end{equation}
    \end{itemize}
\end{theorem}

\begin{proof}
    The first statement is a direct consequence of \cref{thm:S0-B4}, since the vanishing bad brackets assumption implies the conditions \eqref{eq:S2-S1} and \eqref{eq:S4b-S3}.
    Let us prove the second statement.
    Assume that system \eqref{syst} is smoothly-STLC.
    \begin{itemize}
        \item If $\bb \in \Bs_2$, then it is known (see e.g.\ \cite[Theorem 3]{BeauchardMarbach2018}) that $f_\bb(0) \in S_1(f)(0)$, so \eqref{eq:fbb-comp} holds.
        \item If $\bb = Q_{j,k,k}$ with $j \leq k$, then \eqref{eq:fbb-comp} holds since $f_\bb(0) \in \mathcal{N}_{j,k}^4(f)(0)$ by \cref{thm:Qjk-smooth}.
        \item If $\bb = Q_{j,k,k,\nu}$ with $j \leq k$ and $\nu > 0$, then \eqref{eq:fbb-comp} holds by bracketing $\nu$ times relation~\eqref{eq:fbb-comp} for $f_{Q_{j,k,k}}(0)$ with $f_0$ which vanishes at $0$; see \cref{lem:bad-stab-X0}.
        \item If $\bb = Q^\sharp_{j,\mu,k,\nu}$ or $\bb = Q^\flat_{j,\mu,\nu}$, then $f_\bb(0) \in S_{\intset{1,3}}(f)(0)$ by the first item and \cref{lem:sharp-flat-S3}.
        \qedhere
    \end{itemize}
\end{proof}

\begin{remark} \label{rk:comp-badbad}
    One could wish to require, in \eqref{eq:fbb-comp}, that $f_\bb(0)$ be compensated only by good brackets (see also \cref{s:classification-intricacies}).
    Unfortunately, this is not possible.
    Competitions between bad quartic brackets can yield smooth-STLC, as illustrated by \cref{prop:magnitude}.
\end{remark}

An interesting feature of the basis $\Bs$ is that it is easy to conjecture, using the associated coordinates of the second kind (see \cref{Subsec:Coord2_Bstar}) if a given bracket is good or bad.
This leads to the following open problem, which can be seen as an extended and more precise version of Kawski's question concerning $S_{4,5}(X)$.

\begin{open} \label{open:weak-classification}
    Does there exist a subset $\Bs_\bad$ of $\Bs$ such that the following hold?
    \begin{itemize}
        \item Assume that the Lie algebra rank condition $\mathcal{L}(f)(0) = \R^d$ holds and that, for all $\bb \in \Bs_\bad$, $f_\bb(0) = 0$.
        Then system \eqref{syst} is smoothly-STLC.
        
        \item Assume that system \eqref{syst} is smoothly-STLC.
        Then, for every $\bb \in \Bs_\bad$,
        \begin{equation} \label{eq:fbb-comp-general}
            f_{\bb}(0) \in \vect \left\{ f_b(0) ; b \in \Bs \setminus \{ \bb \} \right\}.
        \end{equation}
    \end{itemize}
\end{open}

To establish the first statement, a key difficulty is that, even for a single good bracket, it is not clear how one can succeed in moving both in the direction $f_b(0)$ and in the direction $-f_b(0)$.
Indeed, it is now known that the basic input symmetries $u \mapsto -u$ and $u \mapsto \check{u}$ (time reversal) do not exhaust the possibilities to generate a direction and its opposite as tangent vectors, as illustrated by the techniques used in \cite{AgrachevGamkrelidze1993_Semigroups,Kawski1987_Necessary,Krastanov2009} or in \cref{s:good-4}.

Concerning the second statement, an important difficulty is the fact that the strength of the drift in the direction of $f_\bb(0)$ can be quantified by a very weak functional, which is not directly linked with a Sobolev norm.
For example, in \cite[Theorem 1.14]{BeauchardMarbach2026}, it is proved that $\ad^2_{P_{1,1}}(X_0)$ leads to such an obstruction, quantified by $\int_0^T (\int_0^t u_1^3)^2 \dd t$.
Proving drifts in such cases might require even more advanced interpolation inequalities than the ones derived in \cite{Marbach2023} for our purpose.

\subsubsection{An example involving a competition of quartic bad brackets}

The following system involves a competition between $Q_{1,3,3}$ and $Q_{2,2,2}$, which can either lead to an STLC or not-STLC system, depending on the value of a parameter.
This illustrates some of the difficulties, linked with functional inequalities, that one may encounter when trying to fully classify the controllability of systems, even when restricting to the quartic order.

\begin{proposition}
    \label{prop:magnitude}
    Given $\lambda \in \R$, consider the control-affine system
    \begin{equation} \label{eq:syst-magnitude}
        \begin{cases}
            \dot{x}_1 = u, \\
            \dot{x}_2 = x_1, \\
            \dot{x}_3 = x_2, \\
            \dot{x}_4 = x_1^2 x_3^2 - \lambda x_2^4.
        \end{cases}
    \end{equation}
    There exists $\lambda^* > 0$ such that this system
    \begin{itemize}
        \item is not $W^{-1,\infty}$-STLC when $\lambda \leq \lambda^*$;
        \item is smoothly-STLC when $\lambda > \lambda^*$.
    \end{itemize}
    In $\Bs$, the only basis brackets whose evaluations at $0$ may be nonzero are: $f_{X_1}(0) = e_1$, $f_{M_1}(0) = e_2$, $f_{M_2}(0) = e_3$, $f_{Q_{1,3,3}}(0) = 4 e_4$ and $f_{Q_{2,2,2}}(0) = - 24 \lambda e_4$.
\end{proposition}

\begin{remark}
    System \eqref{eq:syst-magnitude} involves a competition between $Q_{1,3,3}$ and $Q_{2,2,2}$.
    To obtain such a change of nature at a critical threshold, it is important that both brackets have the same functional strength with respect to $T$ and $u$.
    As claimed in \cite[above Example 5.3]{Kawski1987_Survey}, the minimal bracket length for which a competition between brackets having the same time and control homogeneity (i.e.\ involving the same number of $X_0$ and $X_1$ terms) is possible is 9, which is the case here.
\end{remark}

To the best of our knowledge, none of the known sufficient or necessary conditions apply in the case of this competition. 
This example also illustrates that the \emph{relative amplitudes} of collinear bracket evaluations matter, not only the directions that they span.

\cref{prop:magnitude} is proved in \cref{s:bad-bad}.

\subsection{Structure of this part of the article}

\cref{s:prerequisites} contains key prerequisites for the remainder of the paper, defining Hall sets, the $\Bs$ basis, associated coordinates of the second kind, a representation formula for the state, etc.
The remaining sections of \cref{part:smooth} contain the proofs of the results presented in the previous paragraph:
\begin{itemize}
    \item \cref{s:good-4} contains the proofs of \cref{thm:xi4-surj,thm:S0-B4};
    \item \cref{s:obs-smooth} contains the proof of \cref{thm:Qjk-smooth};
    \item \cref{s:classification} contains a discussion on the classification problem;
    \item \cref{s:bad-bad} contains the proof of \cref{prop:magnitude}.
\end{itemize}

\emph{A good understanding of \cref{s:prerequisites,s:obs-smooth,s:classification} is required for \cref{part:rough}.}

\newpage
\section{Prerequisites}  
\label{s:prerequisites}

\subsection{Hall sets} 
\label{s:hall-sets}

In this section, we recall the notion of Hall sets and Hall bases.
For more details on these bases of $\mathcal{L}(X)$, we refer to \cite{BeauchardLeBorgneMarbach2022,Casselman2020_Free}, \cite[Chapter 4]{Reutenauer1993} or \cite[Chapter 1]{Viennot1978}.
These bases are particularly useful for control theory as they are linked with Lazard's elimination process \cite{Lazard1960} and Sussmann's infinite product expansion \cite{Sussmann1986} (see \cite[Section 2.5]{BeauchardLeBorgneMarbach2023}).

\begin{definition}[Length, left and right factors]
    For $b \in \Br(X)$, $|b|$ denotes the length of $b$.
    If $|b| > 1$, $b$ can be written in a unique way as $b = (b_1, b_2)$, with $b_1, b_2 \in \Br(X)$. 
    We use the notations $\lambda(b) = b_1$ and $\mu(b) = b_2$, which define maps $\lambda,\mu: \Br(X)\setminus X \to \Br(X)$.
\end{definition}

\begin{definition}[Hall set] \label{def:Hall}
 A \emph{Hall set} is a subset $\mathcal{B}$ of $\Br(X)$, totally ordered by a relation $<$ and such that
\begin{itemize}
    \item $X \subset \mathcal{B}$,
    \item for $b = ( b_1, b_2 ) \in \Br(X)$, $b \in \mathcal{B}$ iff $b_1, b_2 \in \mathcal{B}$, $b_1 < b_2$ and either $b_2 \in X$ or $\lambda(b_2) \leq b_1$, 
    \item for every $b_1, b_2 \in \mathcal{B}$ such that $(b_1,b_2) \in \mathcal{B}$, one has $b_1 < (b_1,b_2)$.
\end{itemize}
\end{definition}

The main interest of Hall sets is that their images by $\eval$ yield algebraic bases of $\mathcal{L}(X)$, called Hall bases, as proved in \cite[Corollary 1.1, Proposition 1.1 and Theorem 1.1]{Viennot1978}. 

\begin{theorem}[Viennot]
    \label{thm:viennot}
    Let $\mathcal{B} \subset \Br(X)$ be a Hall set. 
    Then $\eval(\mathcal{B})$ is a basis of $\mathcal{L}(X)$.
\end{theorem}

\begin{definition}[Support]
    Let $\mathcal{B}$ be a Hall set of $\Br(X)$ and $a \in \mathcal{L}(X)$. For $b \in \mathcal{B}$, we denote by $\langle a, b \rangle_{\mathcal{B}}$ the coefficient along $\eval(b)$ in the decomposition of $a$ on the basis $\eval(\mathcal{B})$. 
    We define
    \begin{equation}
        \supp_\mathcal{B} (a) := \left\{ b \in \mathcal{B} ; \langle a, b \rangle_{\mathcal{B}} \neq 0 \right\}.
    \end{equation}
    If $A \subset \mathcal{L}(X)$, we denote by $\supp_{\mathcal{B}} (A) := \cup_{a \in A} \supp_{\mathcal{B}}(a)$.
    We drop the subscripts $\mathcal{B}$ when there is no possible confusion on which basis is used.
\end{definition}

The following structural property follows from the classical recursive rewriting algorithm within Hall bases, described for example in \cite[Section 2.1]{BeauchardLeBorgneMarbach2022} or \cite[Section 9]{Reutenauer2003}.

\begin{lemma} \label{p:hall-left}
    Let $\mathcal{B}$ be a Hall set of $\Br(X)$ and $b_1 < b_2 \in \mathcal{B}$. 
    Then, either $(b_1,b_2)\in\mathcal{B}$, or any $b \in \supp_{\mathcal{B}} [b_1,b_2] $ satisfies $\lambda(b)>b_1$.
    In both cases, each $b \in \supp_{\mathcal{B}} [b_1,b_2]$ satisfies $\lambda(b) \geq b_1$.
\end{lemma}

\begin{proof}
    This is item (2.2) of \cite[Theorem 2.1]{BeauchardLeBorgneMarbach2022}. 
\end{proof}

\subsection{The \texorpdfstring{$\Bs$}{B*} basis} 
\label{s:b-star}

In  this section, we recall the definition of the Hall set $\Bs$ introduced in \cite[Section 3]{BeauchardMarbach2026}.

\bigskip

First, we define by induction a subset $G$ of $\Br(X)$ by requiring that, $X_0, X_1 \in G$ and, for every $a,b \in G$ with $a \neq X_0$, $(a,b) \in G$.
Heuristically, $G$ is the subset of $b \in \Br(X)$ for which $X_0$ is never the left factor of any sub-bracket within $b$.
This leads to the following result.

\begin{definition}[Germ] \label{def:germ}
    For any $b \in G \setminus \{ X_0 \}$, there exists a unique pair $(b^*,\nu_b) \in G \times \N$ such that $b = b^* 0^{\nu_b}$, with $b^* = X_1$ or $b^* = (b_1,b_2)$ with $b_1 \neq X_0$ and $b_2 \neq X_0$.
    We call $b^*$ the \emph{germ} of $b$ and we say that $b$ is a germ when $b = b^*$ (i.e.\ $\nu_b = 0$).
    Let $G^*$ be the subset of $G$ made of germs.
\end{definition}

\begin{theorem}
    There exists a unique Hall set $(\Bs,<)$ of $\Br(X)$, contained in $G$, such that \\
    (B0) $X_0$ is the maximal element, \\
    (B1) for $a, b \in G \setminus \{X_0\}$, $a < b$ if and only if $a^* < b^*$ or $a^* = b^*$ and $\nu_a < \nu_b$, \\
    (B2) for $a^*,b^* \in G^*$, $a^* < b^*$ if and only if
        \begin{itemize}
            \item either $n_1(a^*) < n_1(b^*)$,
            \item or $n_1(a^*) = n_1(b^*)$ and $\lambda(a^*) < \lambda(b^*)$,
            \item or $n_1(a^*) = n_1(b^*)$ and $\lambda(a^*) = \lambda(b^*)$ and $\mu(a^*) < \mu(b^*)$.
        \end{itemize}
\end{theorem}

The elements of $\Bs_{\llbracket 1 , 4 \rrbracket}$ are explicitly determined in \cite[Section 3.2]{BeauchardMarbach2026}, which leads to \cref{Prop:Bstar_S_14}. 
By \cref{def:Hall}, $(M_j)_{j\in\N}$ is an increasing sequence. 
By definition of the order of~$\Bs$ we have $\Bs_1<\Bs_2<\Bs_3<\Bs_4$.

\subsection{Coordinates of the second kind} \label{Subsec:Coord2_Bstar}

When decomposing the solution of a differential equation on a Hall basis, one uses so-called \emph{coordinates of the second kind} (see \cite[Section 2.5.3]{BeauchardLeBorgneMarbach2023}), which are defined by induction.
In the particular case of $\Bs$, one has the following formulas (see \cite[Proposition 3.7]{BeauchardMarbach2026}).

\begin{proposition} \label{Prop:Coord_Bstar}
    For every $j \leq k \leq l \in\N^*$,  $\mu, \nu \in \N$, we have
    \begin{align} 
    \label{xi_X0}
    & \xi_{X_0}(t,u) = t 
    \\
    \label{xi_S1}
    & \xi_{M_\nu}(t,u)=\int_0^t \frac{(t-s)^\nu}{\nu!} u(s) \dd s = u_{\nu+1}(t),
    \\ 
    \label{xi_Wjnu}
    & \xi_{W_{j,\nu}}(t,u)= \frac 1 2 \int_0^t \frac{(t-s)^\nu}{\nu!} u_j^2(s) \dd s,
    \\ 
    \label{xi_S3}
    & \xi_{P_{j,k,\nu}}(t,u)=
    \alpha_{j,k}
    \int_0^t \frac{(t-s)^\nu}{\nu!} u_k(s) u_j^2(s) \dd s,
    \\ 
    \label{xi_Qljknu}
    & \xi_{Q_{j,k,l,\nu}}(t,u)= \beta_{j,k,l}
    \int_0^t \frac{(t-s)^\nu}{\nu!} u_l(s) u_k(s) u_j^2(s) \dd s,
    \\ 
    \label{xi_Qbemolknu}
    & \xi_{Q^\flat_{j,\mu,\nu}}(t,u)=\frac{1}{8} \int_0^t \frac{(t-s)^\nu}{\nu!} 
    \left( \int_0^s  \frac{(s-s')^\mu}{\mu!} u_j^2(s') \dd s' \right)^2 \dd s,
    \\
    \label{xi_Qdieseknujmu}
    &\xi_{Q^\sharp_{j,\mu,k,\nu}}(t,u)= \frac 1 4
    \int_0^t \frac{(t-s)^\nu}{\nu!} 
   \left( \int_0^s \frac{(s-s')^\mu}{\mu!} u_j^2(s') \dd s' \right) u_k^2(s) \dd s,
    \end{align}
    where $j < k$ in \eqref{xi_Qdieseknujmu} (only), and the coefficients are given by
    \begin{align}
        \label{eq:alpha_jk}
        \alpha_{j,k}&=\frac{1}{2!} \delta_{j<k} + \frac{1}{3!} \delta_{j=k}, \\
        \label{eq:beta_jkl}
        \beta_{j,k,l}&=\alpha_{j,k} \delta_{k < l} + \frac{1}{(2!)^2} \delta_{j<k=l}+\frac{1}{4!} \delta_{j=k=l}.
    \end{align}
\end{proposition}

In particular the coordinates of the second kind associated with brackets of the form $Q_{j,k,k}$ are positive. This is a key point in the proof of our obstructions. 

\begin{definition}[Maximal factorization]
    \label{def:factorization}
    Let $\mathcal{B} \subset \Br(X)$ be a Hall set such that $X_0$ is maximal.
    Then, for every $b \in \mathcal{B} \setminus X$, $b$ admits a unique \emph{maximal factorization} of the form
    \begin{equation}
        \label{eq:factorization}
        b = \ad_{b_r}^{m_r} \dotsb \ad_{b_1}^{m_1} (X_0),
    \end{equation}
    where $r \geq 1$, $m_1, \dotsc, m_r \geq 1$ and $b_1 < \dotsb < b_r \in \mathcal{B} \setminus \{ X_0 \}$ are the \emph{factors} of $b$.
\end{definition}

The following result, proved in \cite[Proposition 5.25]{BeauchardLeBorgneMarbach2026}, shows that, when $X_0$ is maximal, the coordinates of the second kind can be realized as the coordinates of the state of a ``canonical system'', which has exactly the expected Lie brackets.

\begin{proposition}[Canonical system]
    \label{p:canonical}
    Let $\mathcal{B} \subset \Br(X)$ be a Hall set such that $X_0$ is maximal.
    Let $B \subset \mathcal{B}$ be a non-empty finite subset of $\mathcal{B} \setminus \{ X_0 \}$ such that, for every $b \in B \setminus X$, the factors of $b$ belong to $B$.
    In particular, $X_1 \in B$.
    On $\R^{|B|}$, using the elements of $B$ to index the coordinates of the state, consider the system $\dot{x}_{X_1} = u$ and, for every $b \in B \setminus \{ X_1 \}$,
    \begin{equation}
        \dot{x}_b = \frac{x_{b_r}^{m_r}}{m_r!} \dotsb \frac{x_{b_1}^{m_1}}{m_1!},
    \end{equation}
    as given by the maximal factorization of $b$ of \cref{def:factorization}.
    
    Then, for every $b \in B$, $f_b(0) = e_b$ and, for every $b \in \mathcal{B} \setminus B$, $f_b(0) = 0$.
    
    Moreover, for every $u \in L^1(0,t)$ and $b \in B$, $x_b(t;u) = \xi_b(t,u)$.
\end{proposition}

\begin{example}
    In $\Bs$, consider $B := \{ X_1, M_1, W_1, P_{1,2} \}$.
    The associated canonical system is
    \begin{equation}
        \begin{cases}
            \dot{x}_{X_1} = u, \\
            \dot{x}_{M_1} = x_{X_1}, \\
            \dot{x}_{W_1} = \frac 12 x_{X_1}^2, \\
            \dot{x}_{P_{1,2}} = \frac 12 x_{M_1} x_{X_1}^2.
        \end{cases}
    \end{equation}
    By \cref{p:canonical}, in $\Bs$, the only non-vanishing Lie brackets at $0$ are $f_b(0) = e_b$ for $b \in B$.
\end{example}

\subsection{Coordinates of the pseudo-first kind}
\label{s:eta}

The coordinates of the second kind are given by very nice formulas.
Unfortunately, the Magnus representation of \cref{thm:Magnus} which we use below to compute the state of the system involves slightly different coordinates, which we introduced in \cite[Section 2.4.2]{BeauchardLeBorgneMarbach2023} as \emph{coordinates of the pseudo-first kind} $\eta_b(t,u)$.
Heuristically, for controllability, many things behave as if one had $\xi_b = \eta_b$ (see nevertheless \cite[Section 4.5]{BeauchardMarbach2026} and \cref{s:low-paradox} below).

More precisely, since the $\xi_b$ and $\eta_b$ are related by a kind of Baker--Campbell--Hausdorff formula (see \cite[Proposition 2.16]{BeauchardMarbach2026}), one can bound the difference $|\xi_b-\eta_b|$ as below (see also \cref{p:PZM-xibb-OXi} below in the same spirit).

\begin{definition}[$\mathcal{F}$] \label{def:calF}
    Given $q \geq 2$ and $b_1, \dotsc, b_q \in \Br(X)$, we define $\mathcal{F}(b_1,\dotsc,b_q)$ as the vector subspace of $\mathcal{L}(X)$ spanned by Lie brackets of $b_1, \dotsc, b_q$ involving each of these elements exactly once.
    For example
    \begin{align}
        \mathcal{F}(b_1,b_2) & = \R [b_1, b_2], \\
        \mathcal{F}(b_1,b_2,b_3) & = \R [b_1, [b_2, b_3]] + \R [[b_1,b_2],b_3],
    \end{align}
    thanks to Jacobi's identity $[b_2,[b_1,b_3]]=[[b_2,b_1],b_3]+[b_1,[b_2,b_3]]$.
\end{definition}

\begin{proposition} \label{p:etab-xib-XI}
    Let $b \in \Bs$.
    There exists $C > 0$ such that the following property holds.
    Assume that there exists $\Xi : \R_+^* \times \lloc \to \R_+$ such that, for all $q \geq 2$, $b_1 \geq \dotsb \geq b_q \in \Bs \setminus \{X_0\}$ such that $b \in \supp_{\Bs} \mathcal{F}(b_1, \dotsc, b_q)$, for every $t > 0$ and $u \in \lone$,
    \begin{equation} \label{eq:prod-xibi-XI}
        |\xi_{b_1}(t,u) \dotsb \xi_{b_q}(t,u) | 
        \leq \Xi(t,u).
    \end{equation}
    Then, for every $t > 0$ and $u \in \lone$,
    \begin{equation}
        |\eta_b(t,u) - \xi_b(t,u)| \leq C \Xi(t,u).
    \end{equation}
\end{proposition}

\begin{proof}
    This is \cite[Proposition 2.19]{BeauchardMarbach2026} in the case of $\Bs$.
\end{proof}

\begin{corollary} \label{cor:eta-S1}
    For every $b \in \Bs_1$, $\xi_b = \eta_b$.
\end{corollary}

\begin{proof}
    Let $q \geq 2$ and $b_1, \dotsc, b_q \in \Bs \setminus \{ X_0 \}$.
    For any $a \in \supp_{\Bs} \mathcal{F}(b_1, \dotsc, b_q)$, $n_1(a) = n_1(b_1) + \dotsb + n_1(b_q) \geq 2$.
    So \eqref{eq:prod-xibi-XI} holds with $\Xi = 0$ and the conclusion follows from \cref{p:etab-xib-XI}.
\end{proof}

\subsection{Magnus representation formula}
\label{s:magnus}

The coordinates of the pseudo-first kind introduced in the previous paragraph can be used to compute the state of the system, as illustrated by the Magnus representation formula of \cref{thm:Magnus}.

\begin{definition} \label{subsec:O}
    Given two observables $A(x,u)$ and $B(x,u)$, we write $A(x,u) = O(B(x,u))$ when there exist $C, \rho>0$ such that, for all $t \in (0,\rho)$ and $u \in \lone$ with $\|u\|_{W^{-1,\infty}} \leq \rho$ (recall definition~\eqref{def:norm_W-1}), one has
    \begin{equation}
        |A( x(t;u), u )| \leq C B( x(t;u) , u ).
    \end{equation}
\end{definition}

As examples, one has $t = O(1)$ and $\|u_1\|_{L^\infty} = \|u\|_{W^{-1,\infty}} = O(1)$ .
A deeper result is the following estimate which states that, for scalar-input systems of the form \eqref{syst}, the $W^{-1,\infty}$ norm of the control is an upper bound for the size of the state.

\begin{lemma} \label{p:small-state}
    Let $f_0$, $f_1$ be analytic vector fields on a neighborhood of $0$ with $f_0(0) = 0$.
    Then
    \begin{equation}
        x(t;u) = O(\|u_1\|_{L^\infty}).
    \end{equation}
\end{lemma}

\begin{proof}
    This follows from \cite[Proposition 145]{BeauchardLeBorgneMarbach2023}.
\end{proof}

\begin{theorem} \label{thm:Magnus}
    Let $f_0$, $f_1$ be analytic vector fields on a neighborhood of $0$ with $f_0(0) = 0$.
    Let $M \in \N^*$.
    Then
    \begin{equation} \label{eq:Magnus}
        x(t;u)=\mathcal{Z}_M(t,u)(0) + O\left(\|u_1\|_{L^{M+1}(0,t)}^{M+1} + |x(t;u)|^{1+\frac{1}{M}} \right). 
    \end{equation}
    where
    \begin{equation}    
        \label{eq:ZM=eta}
        \mathcal{Z}_M(t,u)=\sum_{b \in \Bs_{\llbracket 1 , M \rrbracket}} \eta_b(t,u) f_b
    \end{equation}
    where the series converges absolutely in the sense of analytic vector fields in a neighborhood of $0$ and the $\eta_b$ are the coordinates of the pseudo-first kind (see \cref{s:eta}).
\end{theorem}

\begin{proof}
    The approximation result \eqref{eq:Magnus} is contained in \cite[Proposition 161, Item 3]{BeauchardLeBorgneMarbach2023} and the convergence of the series \eqref{eq:ZM=eta} in \cite[Proposition 103]{BeauchardLeBorgneMarbach2023} since $\Bs$ is a Hall basis.
\end{proof}

\subsection{A consequence of the Jacobi identity}

The following straightforward consequences of the Jacobi identity will be useful to compute the decomposition of brackets of two elements within $\Bs$.

\begin{lemma} \label{Lem:Jacobi}
    The following decompositions hold.
    \begin{enumerate}
    \item For any $\nu \in \N$ and any $a,b \in \mathcal{L}(X)$,
        \begin{equation} \label{eq:jacobi.rtl}
            [a, b 0^\nu] = \sum_{\nu'=0}^{\nu} \binom{\nu}{\nu'} (-1)^{\nu'} [a 0^{\nu'}, b] 0^{\nu-\nu'}.
        \end{equation}
    \item For any $\nu \in \N^*$, there exist coefficients $\alpha^\nu_j = (-1)^j \binom{\nu-j-1}{j} \in \Z$ for $1 \leq 2j+1 \leq \nu$, such that, for any $b \in \mathcal{L}(X)$,
        \begin{equation} \label{eq:jacobi.balance}
            [b, b0^\nu] = \sum_{1 \leq 2j+1 \leq \nu} \alpha^\nu_j [b0^j, b0^{j+1}] 0^{\nu-2j-1}.
        \end{equation}
    \end{enumerate}
\end{lemma}

\begin{proof}
    The validity of \eqref{eq:jacobi.rtl} for any $a,b$ can be proved by induction on $\nu\in\N$, the induction steps relies on the Jacobi identity and the binomial relation $\binom{\nu-1}{\nu'}+\binom{\nu-1}{\nu'-1}=\binom{\nu}{\nu'}$ for $\nu'=1,\dots,\nu-1$. 
    The validity of \eqref{eq:jacobi.balance} for any $b$ can be proved by induction on $\nu\in\N^*$; the Jacobi relation leads to $\alpha_j^\nu=\alpha_j^{\nu-1} - \alpha_{j-1}^{\nu-2}$.
\end{proof}

\newpage 

\section{Good brackets yielding controllability}
\label{s:good-4}

In this section, we prove \cref{thm:xi4-surj} and \cref{thm:S0-B4}, concerning the possibility to use the brackets of $\Bgood$ (see \eqref{eq:b4good}) to obtain sufficient conditions for small-time local controllability.

\bigskip

To prove controllability results using nonlinearities, a difficulty with respect to the linear setting is that one can no longer add two controls to add the associated motions.
A typical strategy to restore additivity is to concatenate controls.
This strategy is at the heart of the classical tangent vectors method (see e.g.\ the course \cite[Section 4]{Marbach2026}).

For $1 \leq j \leq k < l$, the coordinate of the second kind of the good bracket $Q_{j,k,l}$ is proportional to $\int u_j^2 u_k u_l$ (recall \eqref{xi_Qljknu}), so expressed using iterated primitives of the control.
To implement more easily the additivity by concatenation, we will look for elementary controls of the form $u = D^d \phi$ for some smooth compactly supported $\phi$, then concatenate them.

Motivated by this approach, we study quartic differential functionals in \cref{s:diff}, then prove that they entail \cref{thm:xi4-surj} in \cref{s:xi4-surj} and \cref{thm:S0-B4} in \cref{s:S0-B4}.

\subsection{Differential functionals}
\label{s:diff}

\subsubsection{Introduction}

We construct controls of the form $u = D^d \phi$ with $\phi \in \CC^\infty_c((0,T);\R)$ for the germs of $\Bs_3 \cup \Bgood$.
For such controls, the coordinates of the second kind are differential functionals of $\phi$.

\begin{lemma}
    \label{lem:flat}
    Let $d \in \N^*$, $T > 0$ and $\phi \in \CC^\infty_c((0,T);\R)$.
    Then $(D^d\phi)_j = D^{d-j}\phi$ for $0 \leq j \leq d$.
    Consequently $\xi_{M_\nu}(T,D^d\phi) = 0$ for $\nu < d$ and, for $1 \leq j \leq k
    \leq d$ and $k < l \leq d$,
    \begin{equation}
        \label{eq:flat}
        \xi_{P_{j,k}}(T,D^d\phi) = \alpha_{j,k} \, \mathcal{P}_{d-j,d-k}(\phi),
        \qquad
        \xi_{Q_{j,k,l}}(T,D^d\phi) = \beta_{j,k,l} \, \mathcal{Q}_{d-j,d-k,d-l}(\phi),
    \end{equation}
    with the coefficients $\alpha_{j,k}$ and $\beta_{j,k,l}$ of \eqref{eq:alpha_jk} and \eqref{eq:beta_jkl}, where, for $\phi \in \CC^\infty_c(\R;\R)$, we set
    \begin{equation}
        \label{eq:calP-calQ}
        \mathcal{P}_{J,K}(\phi) := \int_\R (D^J\phi)^2 D^K \phi
        \qquad \text{and} \qquad
        \mathcal{Q}_{J,K,L}(\phi) := \int_\R (D^J\phi)^2 (D^K\phi)(D^L\phi).
    \end{equation}
\end{lemma}

\begin{proof}
    This is a direct consequence of the explicit formulas of \cref{Prop:Coord_Bstar}.
\end{proof}

The indices occurring in \eqref{eq:flat} satisfy $0 \leq K \leq J$ and $0 \leq L < K \leq J$ respectively, and the map $P_{j,k} \mapsto (d-j,d-k)$, $Q_{j,k,l} \mapsto (d-j,d-k,d-l)$ is injective.
We are thus reduced to the joint surjectivity of the functionals $\mathcal{P}$ and $\mathcal{Q}$, which is the object of the rest of this paragraph.
Given finite families of such indices, we let
\begin{equation}
    \label{eq:bfP-bfQ}
    \mathbf{P}(\phi) := \big( \mathcal{P}_{J_n,K_n}(\phi) \big)_{n \in
    \intset{1,N'}}
    \qquad \text{and} \qquad
    \mathbf{Q}(\phi) := \big( \mathcal{Q}_{J_m,K_m,L_m}(\phi) \big)_{m \in
    \intset{1,N}}.
\end{equation}
Writing $\tau_h \phi := \phi(\cdot - h)$, and denoting by $\mathbf{R}$ either $\mathbf{P}$ or $\mathbf{Q}$, and by $p \in \{3,4\}$ its degree, these maps are homogeneous ($\mathbf{R}(\lambda\phi) = \lambda^p \mathbf{R}(\phi)$ for $\lambda \in \R$), invariant by translation ($\mathbf{R}(\tau_h \phi) = \mathbf{R}(\phi)$) and additive on profiles with disjoint supports.
Under a dilation, $\mathcal{P}_{J,K}(\phi(\cdot/\tau)) = \tau^{1-S} \mathcal{P}_{J,K}(\phi)$ with $S := 2J+K$, and similarly for $\mathcal{Q}_{J,K,L}$ with $S := 2J+K+L$.

\medskip

The following elementary fact is used in the constructions below.

\begin{lemma}
    \label{lem:theta}
    Let $J \in \N$ and $q \ge 2$.
    There exists $\theta \in \CC^\infty_c(\R;\R)$ such that $\int_\R (D^J\theta)^q = 1$.
\end{lemma}

\begin{proof}
    If $\int (D^J\theta)^q$ vanished for every $\theta$, then, differentiating at some $\theta_0$ with $g_0 := D^J\theta_0 \not\equiv 0$, we would get $\int g_0^{q-1} D^J\theta = 0$ for every $\theta$, i.e.\ $D^J(g_0^{q-1}) = 0$, so that the compactly supported function $g_0^{q-1}$ would be a polynomial, a
    contradiction.
    One concludes by rescaling.
\end{proof}

\subsubsection{Quartic germs}

The properties of homogeneity, invariance by translation and additivity under disjoint support lead to the following observation.

\begin{lemma}
    \label{lem:bfQ-cone}
    Define $\mathbf{Q}$ as in \eqref{eq:bfP-bfQ} for a given finite set of distinct tuples $(J_n, K_n, L_n)_{n \in \intset{1,N}}$ satisfying $0 \leq L_n < K_n \leq J_n$.
    The set $\mathbf{Q}(\CC^\infty_c(\R;\R))$ is a convex cone of $\R^N$.
\end{lemma}

\begin{proof}
    Let $\phi, \psi \in \CC^\infty_c(\R;\R)$.
    Choosing $h$ such that $\phi$ and $\tau_h \psi$ have disjoint support, $\mathbf{Q}(\phi) + \mathbf{Q}(\psi) = \mathbf{Q}(\phi + \tau_h \psi)$; hence the set is additive.
    The conclusion follows from the homogeneity property.
\end{proof}

We intend to prove that this set is in fact equal to $\R^N$, using the following induction principle.

\begin{lemma}
    \label{lem:cone-induction}
    Let $N \ge 1$ and $\pi : \R^N \to \R^{N-1}$ the canonical projection on the $N-1$ first directions.
    Let $\mathcal{C}$ be a convex cone of $\R^N$ such that $\pi(\mathcal{C}) = \R^{N-1}$ and $\pm e_N \in \overline{\mathcal{C}}$.
    Then $\mathcal{C} = \R^N$.
\end{lemma}

\begin{proof}
    Since $\overline{\mathcal{C}}$ is a convex cone containing $\pm e_N$, it contains $\R e_N$.
    Moreover, since $\pi(\mathcal{C}) = \R^{N-1}$, $\R^N = \mathcal{C} + \R e_N \subset \overline{\mathcal{C}} + \overline{\mathcal{C}} \subset \overline{\mathcal{C}}$.
    Thus $\overline{\mathcal{C}} = \R^N$.
    Hence $\mathcal{C}$ is a dense convex subset of $\R^N$, so $\mathcal{C} = \R^N$.
\end{proof}

We now establish the following separation property needed in the induction.

\begin{lemma}
    \label{lem:calQ-induction}
    Let $J, K, L$ satisfy $0 \leq L < K \leq J$ and define $S := 2 J + K + L$.
    For any $\eta > 0$, there exist $\phi^\pm \in \CC^\infty_c(\R;\R)$ such that
    \begin{equation}
        \mathcal{Q}_{J,K,L}(\phi^\pm) = \pm 1
        \qquad \text{and} \qquad 
        |\mathcal{Q}_{J',K',L'}(\phi^\pm)| \leq \eta
    \end{equation}
    for all tuples such that $0 \leq L' < K' \leq J'$ and $(S',J',K') <_{\mathrm{lex}} (S,J,K)$ where $S' := 2J'+K'+L'$.
\end{lemma}

\begin{proof}
    To lighten the notations, let $\mathcal{Q}$ and $\mathcal{Q}'$ denote the quartic functionals $\mathcal{Q}_{J,K,L}$ and $\mathcal{Q}_{J',K',L'}$.
    Let $\sigma = \pm 1$.
    By homogeneity, it suffices to construct a family $(\phi_\varepsilon)_{\varepsilon > 0}$ such that $\sigma  \mathcal{Q}(\phi_\varepsilon) > 0$ and, for all admissible primed index tuples,
    \begin{equation}
        \label{eq:calQ'-o}
        \mathcal{Q}'(\phi_\varepsilon) = o(\mathcal{Q}(\phi_\varepsilon))
        \quad \text{as} \quad \varepsilon \to 0.
    \end{equation}

    \medskip \noindent \emph{Step 1: Construction of the family.}
    Choose $\rho \in \CC^\infty_c(\R;\R)$ with $\rho \equiv 1$ near $0$ and set
    \begin{equation}
        \label{eq:calQ-chi}
        \chi(t) := \rho(t) \left( \sigma \frac{t^L}{L!} + \frac{t^K}{K!} \right)
    \end{equation}
    Let $\theta \in \CC^\infty_c(\R;\R)$ be given by \cref{lem:theta} such that $\int_\R (D^J \theta)^2 = 1$ when $K < J$ and $\int_\R (D^J \theta)^3 = 1$ when $K = J$. 

    We take $\phi_\varepsilon$ of the form
    \begin{equation}
        \label{eq:calQ-phi-psi}
        \phi_{\varepsilon}(t)=\psi_{\varepsilon}(t/\varepsilon)
        \quad \text{where} \quad \psi_\varepsilon(t) := \chi(t) + \varepsilon^a \theta(\varepsilon^{-a r} t)
    \end{equation} 
    where $0 < a < \frac{1}{20}$ and $\frac1r = J - \frac 23$ so that
    \begin{equation}
        \label{eq:calQ-ab}
        1 - r J = - \frac{2r}{3} < 0
        \quad \text{and} \quad 
        1 - r q \geq \frac{r}{3} > 0 \quad \text{for all} \quad q < J
    \end{equation}
    and
    \begin{equation}
        \label{eq:calQ-ab-bis}
        r + 2 (1 - r J) = - \frac{r}{3} < 0
        \quad \text{and} \quad 
        r + 3 (1 - r J) = - r < 0.
    \end{equation}

    \medskip \noindent \emph{Step 2: Leading term of $\mathcal{Q}(\psi_\varepsilon)$.}
    From \eqref{eq:calQ-chi}, we have
    \begin{equation}
        \label{eq:calQ-chi-DL}
        D^L \chi(0) = \sigma, \qquad
        D^K \chi(0) = 1, \qquad 
        D^q \chi(0) = 0 \quad \text{for } 0 \leq q < K, \enskip q \neq L.
    \end{equation}
    \emph{Case $K < J$.}
    One has $\supp \theta(\varepsilon^{- a r} \cdot) = \varepsilon^{a r} \supp \theta$, which shrinks to $0$.
    Thus, the identities \eqref{eq:calQ-ab} and \eqref{eq:calQ-chi-DL} yield, uniformly on $\supp \theta(\varepsilon^{- a r} \cdot)$,
    \begin{equation}
        D^J \psi_\varepsilon = \varepsilon^{a (1 - r J)} (D^J \theta)(\varepsilon^{-a r}\cdot) + O(1),
        \qquad 
        D^K \psi_\varepsilon = 1 + o(1),
        \qquad 
        D^L \psi_\varepsilon = \sigma + o(1).
    \end{equation}
    After the change of variables $s = \varepsilon^{-a r} t$, using $\int (D^J \theta)^2 = 1$, we obtain
    \begin{equation}
        \label{eq:calQ-theta2}
        \mathcal{Q}(\psi_\varepsilon) = (\sigma + o(1)) \varepsilon^{a (r + 2 (1 - r J))}.
    \end{equation}
    \emph{Case $K = J$.}
    Similarly, using $\int (D^J \theta)^3 = 1$, we obtain
    \begin{equation}
        \label{eq:calQ-theta3}
        \mathcal{Q}(\psi_\varepsilon) = (\sigma + o(1)) \varepsilon^{a (r + 3 (1 - r J))}.
    \end{equation}
    In particular, in both cases, as $\varepsilon \to 0$, $\sigma \mathcal{Q}(\psi_\varepsilon) > 0$ and $|\mathcal{Q}(\psi_\varepsilon)| \to +\infty$ by \eqref{eq:calQ-ab-bis}.

    \medskip \noindent \emph{Step 3: Separation from preceding tuples.}
    From \eqref{eq:calQ-phi-psi}, the change of variables $s = t/\varepsilon$ yields
    \begin{equation}
        \label{eq:calQ-scaling}
        \frac{|\mathcal{Q}'(\phi_\varepsilon)|}{|\mathcal{Q}(\phi_\varepsilon)|} = \varepsilon^{S - S'} \frac{|\mathcal{Q}'(\psi_\varepsilon)|}{|\mathcal{Q}(\psi_\varepsilon)|}.
    \end{equation}
    We consider the three possible situations in which $(S',J',K') <_{\mathrm{lex}} (S,J,K)$:
    \begin{itemize}
        \item \emph{Case $S' < S$.}
        Then $2 J' \leq S' < S < 4J$ so $J' \leq 2 J$.
        From \eqref{eq:calQ-phi-psi}, for any $q \leq J' \leq 2 J$,
        \begin{equation}
            D^q \psi_\varepsilon = O(1 + \varepsilon^{a (1 - r q)}) = O(\varepsilon^{- a c_q})
            \quad \text{where} \quad 
            c_q := \left( r q - 1 \right)_+ \leq 5.
        \end{equation}
        Consequently,
        \begin{equation}
            \label{eq:calQ'-Gamma}
            \mathcal{Q}'(\psi_\varepsilon) = O(\varepsilon^{-a \Gamma'})
            \quad \text{where} \quad
            \Gamma' := 2 c_{J'} + c_{K'} + c_{L'} \leq 20.
        \end{equation}
        Since $a < \frac{1}{20}$, $a \Gamma' < 1 \leq S - S'$. 
        Since $|\mathcal{Q}(\psi_\varepsilon)| \to +\infty$, \eqref{eq:calQ-scaling} and \eqref{eq:calQ'-Gamma} imply \eqref{eq:calQ'-o}.

        \item \emph{Case $S' = S$ and $J' < J$.}
        All derivatives of $\psi_\varepsilon$ of order at most $J'$ are uniformly bounded thanks to \eqref{eq:calQ-ab}.
        Hence $|\mathcal{Q}'(\psi_\varepsilon)| = O(1)$.
        Since $|\mathcal{Q}(\psi_\varepsilon)| \to +\infty$, the scaling estimate \eqref{eq:calQ-scaling} with $S = S'$ implies 
        \eqref{eq:calQ'-o}.

        \item \emph{Case $S' = S$, $J' = J$ and $K' < K$.}
        Then $L < L' < K' < K$.
        By \eqref{eq:calQ-chi-DL}, on the support of $\theta(\varepsilon^{-a r} \cdot)$, one has
        \begin{equation}
            D^{J'} \psi_\varepsilon = \varepsilon^{a (1 - r J)} (D^J \theta)(\varepsilon^{- a r} \cdot) + O(1),
            \qquad 
            D^{K'} \psi_\varepsilon = o(1), 
            \qquad
            D^{L'} \psi_\varepsilon = o(1).
        \end{equation}
        Therefore
        \begin{equation}
            \mathcal{Q}'(\psi_\varepsilon) = O(1) + o(\varepsilon^{a r} \varepsilon^{2 a (1 - r J)}).
        \end{equation}
        Comparing with \eqref{eq:calQ-theta2} when $K < J$ and with \eqref{eq:calQ-theta3} when $K = J$ proves \eqref{eq:calQ'-o}. \qedhere
    \end{itemize}
\end{proof}

We can now prove our main claim.

\begin{proposition}
    \label{prop:bfQ-surj}
    Given $N \geq 0$ and a finite set $(J_n, K_n, L_n)_{n \in \intset{1, N}}$ of tuples of integers satisfying $0 \leq L_n < K_n \leq J_n$, the functional $\mathbf{Q}$ of \eqref{eq:bfP-bfQ} is onto.
\end{proposition}

\begin{proof}
    We proceed by induction on $N \geq 0$.
    For $N = 0$, there is nothing to prove.
    Let $N > 0$ and assume that the result holds for all sets of size at most $N - 1$.
    Let $(J_n, K_n, L_n)_{n \in \intset{1, N}}$ be a set of $N$ admissible tuples.
    Assume that they are ordered lexicographically based on the key $(S_n, J_n, K_n)$ where $S_n := 2 J_n + K_n + L_n$.
    Let $\mathcal{C} := \mathbf{Q}(\CC^\infty_c(\R;\R)) \subset \R^N$.
    By \cref{lem:bfQ-cone}, $\mathcal{C}$ is a convex cone.
    By the induction hypothesis, $\pi(\mathcal{C}) = \R^{N-1}$, where $\pi : \R^N \to \R^{N-1}$ is the projection on the first $N-1$ directions.
    By \cref{lem:calQ-induction}, $\pm e_N \in \overline{\mathcal{C}}$.
    By \cref{lem:cone-induction}, $\mathcal{C} = \R^N$.
\end{proof}

\subsubsection{Cubic germs}

We recall the standard cubic surjectivity result in a form adapted to our needs.

\begin{proposition}
    \label{lem:bfP-surj}
    Given $N' \geq 0$ and a finite set $(J_n, K_n)_{n \in \intset{1, N'}}$ of tuples of integers satisfying $0 \leq K_n \leq J_n$, the functional $\mathbf{P}$ of \eqref{eq:bfP-bfQ} is onto.
\end{proposition}

\begin{proof}
    The image of $\mathbf{P}$ is a vector subspace.
    Indeed, cubic homogeneity gives stability under multiplication by arbitrary real scalars, while invariance by translation and additivity on disjoint supports give stability under addition.

    It therefore suffices to prove that the functionals $\mathcal{P}_{J,K}$ are linearly independent.
    Assume that a finite linear combination vanishes:
    \begin{equation}
        \forall \phi \in \CC^\infty_c(\R;\R), \quad
        \sum_{J,K}c_{J,K}\mathcal{P}_{J,K}(\phi) = 0.
    \end{equation}
    Polarizing this cubic identity and extending it complex-linearly, then testing it on
    \begin{equation}
        \phi_r(t)=\chi(t/R)e^{\mathrm{i}\lambda_r t},
        \qquad r\in\{1,2,3\},
        \qquad \lambda_1+\lambda_2+\lambda_3=0,
    \end{equation}
    where $\chi\in \CC^\infty_c(\R)$ and $\int\chi^3\neq0$, yields, after division by $R$ and passage to the limit $R\to+\infty$, the lower-order terms containing derivatives of $\chi(t/R)$ being $o(R)$,
    \begin{equation}
        \label{eq:cubic-symbol}
        \sum_{J,K}c_{J,K}\mathrm{i}^{2J+K}
        q_{J,K}(\lambda_1,\lambda_2,\lambda_3)=0,
    \end{equation}
    where
    \begin{equation}
        q_{J,K}(x,y,z)
        :=x^Jy^Jz^K+x^Jz^Jy^K+y^Jz^Jx^K.
    \end{equation}
    The terms of different total degrees $S=2J+K$ in
    \eqref{eq:cubic-symbol} vanish separately.
    For fixed $S$, set $(x,y,z)=(1,t,-1-t)$.
    If $K<J$, the polynomial $q_{J,K}(1,t,-1-t)$ vanishes at $t=0$ to exact order $K$, with leading coefficient $(-1)^J$.
    If $K=J$, it equals $3(-t(1+t))^J$ and again has exact order $K$.
    Since, for fixed $S$, distinct pairs have distinct values of $K$, these polynomials are linearly independent.
    Thus every $c_{J,K}$ vanishes, which proves the claim.
\end{proof}

\subsubsection{Simultaneous cubic and quartic germs}

Cubic functionals are odd and quartic ones are even: this is what makes the two
constructions independent of one another.

\begin{lemma}
    \label{lem:PQ}
    Under the assumptions of \cref{prop:bfQ-surj,lem:bfP-surj}, the map $\phi
    \mapsto (\mathbf{P},\mathbf{Q})(\phi)$ is onto.
\end{lemma}

\begin{proof}
    Let $\mathcal{D}$ denote the image, which is stable by addition (translate the
    supports apart).
    Let $\phi \in \CC^\infty_c(\R;\R)$ and let $h$ be such that $\phi$ and $\tau_h\phi$ have disjoint supports.
    Then $\Phi := \phi - \tau_h \phi$ satisfies $\mathbf{P}(\Phi) = \mathbf{P}(\phi) - \mathbf{P}(\phi) = 0$ and $\mathbf{Q}(\Phi) = 2 \mathbf{Q}(\phi)$.
    With \cref{prop:bfQ-surj}, this proves that $\{0\} \times \R^N \subset \mathcal{D}$.
    Now let $y \in \R^{N'}$ and let $\psi$ be given by \cref{lem:bfP-surj} with $\mathbf{P}(\psi) = y$; adding to $\psi$ an element of the previous family with disjoint support and quartic value $-\mathbf{Q}(\psi)$, we obtain $(y,0) \in \mathcal{D}$.
    Hence $\R^{N'} \times \{0\} \subset \mathcal{D}$ and we conclude by additivity.
\end{proof}

\begin{corollary}
    \label{cor:germ-dual}
    Let $d \in \N^*$, $T > 0$ and let $G^*$ be a finite set of brackets of the form $P_{j,k}$ with $j \leq k \leq d$, or $Q_{j,k,l}$ with $j \leq k < l \leq d$.
    There exists a family $(\phi^\pm_b)_{b \in G^*}$ of $\CC^\infty_c((0,T);\R)$ such that, for all $b \in G^*$,
    \begin{equation}
        \label{eq:germ-dual}
        \xi_a(T, D^d \phi^\pm_b) = \pm \delta_{a = b}
        \quad (\forall a \in G^*),
        \qquad
        \xi_{M_{\nu-1}}(T, D^d \phi^\pm_b) = 0
        \quad (\forall \nu \in \intset{1, d}).
    \end{equation}
\end{corollary}

\begin{proof}
    By \cref{lem:flat}, the index tuples associated with the elements of $G^*$ are admissible and pairwise distinct, so that \cref{lem:PQ} provides, for each $b \in G^*$ profiles $\psi^\pm_b \in \CC^\infty_c(\R;\R)$ satisfying \eqref{eq:germ-dual} up to the positive constants $\alpha_{j,k}$, $\beta_{j,k,l}$ of \eqref{eq:flat}, which one absorbs by homogeneity.
    Finally, let
    \begin{equation}
        \phi^\pm_b := \lambda_b \psi^\pm_b \left(\frac{\cdot - t_b}{\tau_b}\right)
        \quad \text{where} \quad 
        \lambda_b := \tau_b^{(S_b-1)/n_1(b)}
    \end{equation}
    and $t_b$ and $\tau_b > 0$ are chosen so that the support lies in $(0,T)$.
    Translations leave the functionals invariant and dilations multiply them by $\tau_b^{1-S_b}$, which is compensated by the amplitude scaling.
\end{proof}

\subsection{Surjectivity of the coordinates of the second kind}
\label{s:xi4-surj}

Our surjectivity result \cref{thm:xi4-surj} of the introduction can be derived in multiple ways using classical folklore techniques of the control theory literature.
We give below a sample derivation. 
We start with a definition.
We give related results in \cref{s:dual}.

\begin{definition}[Dual family]
    \label{def:dual}
    Let $\mathcal{B}$ be a Hall set of $\Br(X)$, $B$ be a finite subset of $\mathcal{B}$ and $T > 0$.
    A \emph{dual family for~$B$ on $[0,T]$} is a family $(u^\pm_b)_{b \in B}$ of $\CC^\infty_c((0,T);\R)$ such that
    \begin{equation}
        \label{eq:dual}
        \forall a, b \in B, \quad
        \xi_a(T, u^\pm_b) = \pm \delta_{a = b}.
    \end{equation}
\end{definition}

By homogeneity, this definition does not depend on $T$ (see \cref{lem:homog}).

\medskip

In \cref{s:diff}, we produced a dual family for germs of $\Bs_3 \cup \Bgood$.
It remains to handle brackets of $\Bs_1$, and the iterates by $\ad_{X_0}$ of the form $b^* 0^\nu$ for some germ $b^* \in \Bs_3 \cup \Bgood$.

\begin{proposition}
    \label{p:dual-hatG}
    Let $d \in \N^*$, let $G^*$ be as in \cref{cor:germ-dual} and let $N_a \in \N$ for $a \in G^*$.
    Then
    \begin{equation}
        \label{eq:hatG}
        \widehat{G} := \{ M_\nu \mid \nu < d \} \cup \{ a 0^\nu \mid a \in G^*, \enskip 0 \leq \nu \leq N_a \}
    \end{equation}
    is a factor-stable subset of $\Bs \setminus \{X_0\}$ and has a dual family.
\end{proposition}

\begin{proof}
    \step{Factor-stability}
    Since $M_\nu = \ad_{M_{\nu-1}}(X_0)$ for $\nu \geq 1$ and $W_j =
    \ad_{M_{j-1}}^2(X_0)$, the maximal factorizations of the germs are
    \begin{equation*}
        P_{j,k} =
        \begin{cases}
            \ad_{M_{k-1}} \ad_{M_{j-1}}^2 (X_0), & j < k, \\
            \ad_{M_{j-1}}^3 (X_0), & j = k,
        \end{cases}
        \qquad
        Q_{j,k,l} =
        \begin{cases}
            \ad_{M_{l-1}} \ad_{M_{k-1}} \ad_{M_{j-1}}^2 (X_0), & j < k, \\
            \ad_{M_{l-1}} \ad_{M_{j-1}}^3 (X_0), & j = k,
        \end{cases}
    \end{equation*}
    and $a 0^\nu = \ad_{a0^{\nu-1}}(X_0)$ for $\nu \geq 1$.
    In particular the factors of the elements of $G^*$ are the $M_{\nu}$ with $\nu < d$, since $j,k,l \leq d$, and the only factor of $a0^\nu$ is $a0^{\nu-1}$.
    Hence $\widehat{G}$ is factor-stable.

    \step{Germs}
    By \cref{lem:linear}, $L_d := \{ M_\nu \mid \nu < d \}$ has a dual family.
    The elements of $G^*$ do not belong to $L_d$ and their factors do, and the controls $D^d \phi^\sigma_b$ of \cref{cor:germ-dual} satisfy the hypotheses of \cref{lem:dual-new} with $B := L_d$ and $H := G^*$.
    Therefore $L_d \cup G^*$ has a dual family.

    \step{Iterates by $X_0$}
    We now add the brackets $a0^\nu$, for $a \in G^*$ and $1 \leq \nu \leq N_a$, one at a time and by increasing $\nu$.
    At each stage, the current set $B$ is stable, and $a0^{\nu-1} \in B$ is terminal in $B$: indeed, by Step 1, the only element of $\Bs$ which admits $a0^{\nu-1}$ as a factor and could belong to $B$ is $a0^{\nu}$, which has not been added yet.
    \cref{lem:dual-b0} yields a dual family for $B \cup \{a0^\nu\}$.
\end{proof}

\begin{proof}[Proof of \cref{thm:xi4-surj}]
    Let $G$ be a finite subset of $\mathcal{G} = \Bs_1 \cup \Bs_3 \cup \Bgood$.
    Choose $d \in \N^*$ larger than every $\nu$ such that $M_\nu \in G$, and larger
    than or equal to every index $j,k,l$ occurring in an element $P_{j,k,\nu}$ or
    $Q_{j,k,l,\nu}$ of $G$.
    Let $G^*$ be the set of the germs $P_{j,k}$ and $Q_{j,k,l}$ of the elements of $G \setminus \Bs_1$ and, for $a \in G^*$, let $N_a := \max \{ \nu ; a0^\nu \in G \}$.
    Then $G \subset \widehat{G}$, where $\widehat{G}$ is given by \eqref{eq:hatG}.
    By \cref{p:dual-hatG,p:dual-inverse}, $\widehat{G}$ has a continuous right inverse $R_{\widehat{G}}$.
    Denoting by $\iota : \R^G \to \R^{\widehat{G}}$ the extension by zero, the map $R_G := R_{\widehat{G}} \circ \iota$ has all the properties required by
    \cref{thm:xi4-surj}.
\end{proof}

\subsection{Proof of the sufficient condition of controllability}
\label{s:S0-B4}

We explain how \cref{thm:xi4-surj} entails \cref{thm:S0-B4}, our main quartic sufficient condition for controllability.
The proof relies on Sussmann's control dilation $u^\varepsilon(t) := \varepsilon^{1-\theta} u(t / \varepsilon^\theta)$ of \cite{Sussmann1987} for a well-chosen $\theta > 0$, reference control $u$ and $0 < \varepsilon \ll 1$.
Using these dilations, the abstract result of \cref{thm:Stheta-Holder} reduces controllability to a compensation condition \eqref{eq:omega-compensation}, and the continuous inversion of a finite family of coordinates.

\begin{proof}[Proof of \cref{thm:S0-B4}]
	Let $\mathcal{G} := \Bs_1 \cup \Bs_3 \cup \Bgood$.
	The rank condition $S_{\intset{1,4}}(f)(0) = \R^d$ and the assumptions \eqref{eq:S2-S1} and \eqref{eq:S4b-S3} yield $\mathcal{G}(f)(0) = \R^d$.
	Enumerating $\mathcal{G}$ by non-decreasing $n_1$, we can extract greedily $b_1, \dotsc, b_d \in \mathcal{G}$ such that the $f_{b_i}(0)$ for $i \in \intset{1,d}$ form a basis of $\R^d$, with the first elements belonging to $\Bs_1$ and spanning $S_1(f)(0)$, then to $\Bs_3$ and spanning $S_1(f)(0) + S_3(f)(0)$. 
	
	Let $N := \max_i |b_i|$.
	We consider the finite set $G := \{ b \in \mathcal{G} : |b| \le N \}$.
	It is factor-stable and, by \cref{thm:xi4-surj}, admits a continuous right inverse in the sense of \cref{def:C0-inverse}.
	
	Let $\theta \in (0,\frac 1 N)$. 
	We claim that \eqref{syst} satisfies the compensation condition of \eqref{eq:omega-compensation}.
	Let $b \in \Bs \setminus G$.
	If $b \in \Bs_2$, by \eqref{eq:S2-S1}, $f_b(0)$ is spanned by some $f_{b_i}(0)$ with $n_1(b_i) < n_1(b)$ so $\omega(b) - \omega(b_i) \geq 1 - N \theta > 0$.
	If $b \in \Bbad$, the same holds by \eqref{eq:S4b-S3}.
	If $n_1(b) \ge 5$, the same holds by the rank condition  $S_{\intset{1,4}}(f)(0) = \R^d$.
	Now if $b \in (\mathcal{G} \setminus G)$, by definition of $G$, $|b| > N$.
	By construction, $f_b(0)$ is spanned by brackets $b_i$ with $n_1(b_i) \leq n_1(b)$ and $n_0(b_i) \leq N$. 
	Thus $\omega(b) - \omega(b_i) \geq \theta > 0$.
	
	By \cref{thm:Stheta-Holder}, \eqref{syst} is $W^{m,\infty}_0$-STLC for any $-1 \le m < \frac 1 \theta - 1$.
	Taking a sequence $\theta \to 0$ proves that \eqref{syst} is smoothly-STLC.
\end{proof}

\newpage
\section{Obstructions to smooth-STLC}
\label{s:obs-smooth}

Let $j \leq k \in \N^*$.
We prove \cref{thm:Qjk-smooth} on the obstruction to smooth small-time local controllability caused by the bad bracket $\q := Q_{j,k,k}$.
This section is intended as a lightweight version of \cref{s:obs}, which tackles the rough-STLC case and is much more technical.

The proof relies on our unified approach of obstructions to STLC, introduced in \cite{BeauchardMarbach2026}.
We recall the main lines of this approach below, and refer to \cite[Section 1.6]{BeauchardMarbach2026} for more details.

\subsection{Introduction} 
\label{s:approach-Wm}

\subsubsection{Obstructions to controllability as drifts}

Our \emph{obstructions to controllability} are results of the form: $\text{STLC} \Rightarrow f_\q(0) \in \mathcal{N}(f)(0)$, where $\q \in \Bs$ and $\mathcal{N} \subset \Br(X)$.
See \cite[Section 1.6.3]{BeauchardMarbach2026} for a discussion on how to conjecture the correct set~$\mathcal{N}$.
We prove these results by contraposition, starting from the assumption
\begin{equation} \label{Heuristic:Hyp_non_STLC}
    f_\q(0) \notin \mathcal{N}(f)(0).
\end{equation}
Our strategy consists in proving that, when \eqref{Heuristic:Hyp_non_STLC} holds, the state $x(t;u)$ ``drifts'' in the direction of $+ f_\q(0)$, in the sense of \cref{def:drift} below, which requires the following notion. 

\begin{definition}[Component]
    \label{def:component}
    Let $N$ be a vector subspace of $\R^d$ and $e \in \R^d \setminus N$.
    We say that a linear form $\mathbb{P}:\R^d\to \R$ is \emph{a component along $e$ parallel to~$N$} when $\mathbb{P} e = 1$ and $N \subset \ker \mathbb{P}$.
\end{definition}

\begin{definition}[Drift] 
    \label{def:drift}
    Let $\q \in \Bs$ and $\mathcal{N} \subset \Br(X)$.
    We say that system \eqref{syst} has a \emph{drift along~$f_\q(0)$, parallel to $\mathcal{N}(f)(0)$, as $t \to 0$ and $u \to 0$}, when there exist $C>0$ and $\beta>1$ such that, for all $\varepsilon>0$, for small enough\footnote{Even if the smallness assumption on $u$ depends on $t$, \eqref{eq:def-drift} still prevents STLC.} $t$ and $u$, 
    \begin{equation} \label{eq:def-drift}
        \mathbb{P} x(t;u) \geq (1-\varepsilon) \xi_\q(t,u) - C |x(t;u)|^\beta,
    \end{equation}
    where $\mathbb{P}$ gives a component along $f_\q(0)$ parallel to $\mathcal{N}(f)(0)$ and $(\xi_{b})_{b\in\Bs}$ are the coordinates of the second kind associated with $\Bs$ (see \cref{Prop:Coord_Bstar}).
\end{definition}

When $\xi_\q(t,\cdot) \ge 0$, the presence of a drift denies STLC for the considered class of controls, since one cannot reach targets of the form $- a f_\q(0)$ for $0 < a \ll 1$.
See \cite[Section 1.6.1]{BeauchardMarbach2026} for a proof, comments on the geometry of the reachable set, and an illustrative example.

\subsubsection{Heuristic of the drift}
\label{sec:heuristic}

The starting point of our strategy is the approximate representation formula recalled in \cref{s:magnus}.
Let $M \in \N^*$. 
As $(t,\|u_1\|_{L^\infty}) \to 0$,
\begin{equation} \label{repform}
    x(t;u)=\sum_{b\in \Bs_{\intset{1,M}}} \eta_b(t,u) f_b(0) +O\left( \|u_1\|_{L^{M+1}}^{M+1} + |x(t;u)|^{1+\frac{1}{M}} \right),
\end{equation}
Under assumption \eqref{Heuristic:Hyp_non_STLC}, we can consider $\mathbb{P}:\R^d\to \R$, a component along $f_\q(0)$ parallel to $\mathcal{N}(f)(0)$. 
Recalling that the subset $\mathcal{N}$ we use is defined in \eqref{eq:NjkM} as $\Bs_{\intset{1,M}} \setminus \{ \q \}$, we obtain
\begin{equation} \label{Heuristic:rep_form}
    \mathbb{P} x(t;u) = \eta_\q(t,u)  +O\left(  \|u_1\|_{L^{M+1}}^{M+1}  + |x(t;u)|^{1+\frac{1}{M}}\right).
\end{equation}
Starting from \eqref{Heuristic:rep_form}, the proof of the presence of the drift relies on the following ideas:
\begin{itemize}
    \item the coordinate of the pseudo-first kind $\eta_\q(t,u)$ is close to $\xi_\q(t,u)$,
    \item $\xi_\q(t,u)$, proportional to $\int_0^t u_j^2 u_k^2$, behaves like the $W^{-\frac{j+k}{2},4}$ Sobolev norm of $u$,
    \item for a good choice of $M$ and $m$, this suffices to absorb the remainders when $\|u\|_{W^{m,\infty}} \ll 1$.
\end{itemize}

\subsubsection{Organization of the proof}
\label{s:obs-organization}

The proof of \cref{thm:Qjk-smooth} involves three independent arguments of different natures.

\begin{itemize}
    \item \textbf{Algebraic argument.} 
    \cref{s:eta-Qjk-smooth} studies algebraic relations linked with the structure constants of $\mathcal{L}(X)$ relative to $\Bs$ to quantify the heuristic $\eta_\q \approx \xi_\q$.
    This part is purely algebraic (it involves neither a control system, nor interpolation arguments).
    Roughly phrased, its main conclusion \cref{p:eta_Qjk-smooth} proves that 
    \begin{equation}
        \label{eq:eta-xi-approx}
        |\eta_\q(t,u) - \xi_\q(t,u)| \lesssim |U_k(t)| \times \text{other terms},
    \end{equation}
    where we introduce the ``pointwise boundary term''
    \begin{equation} \label{eq:borduk}
        U_k(t) := (u_1, \dotsc, u_k)(t).
    \end{equation}

    \item \textbf{Geometric argument.}
    \cref{s:closed-loop-smooth} studies geometric consequences of the assumption \eqref{Heuristic:Hyp_non_STLC}, to obtain estimates of the boundary term $|U_k(t)|$.
    This part is geometric in the sense that it involves ``vectorial relations'' between the iterated Lie brackets of the vector fields of the considered system.
    Roughly phrased, its main conclusion \cref{p:borduk-smooth} proves that
    \begin{equation}
        \label{eq:borduk-approx}
        |U_k(t)| \lesssim |x(t;u)| + \text{other terms}.
    \end{equation}
    We say that \eqref{eq:borduk-approx} is a ``closed-loop estimate'' because it bounds the control by the state.
    It entails that, for any $\beta > 1$, $|U_k(t)|^\beta$ can be absorbed in the second part of the definition of the drift in \eqref{eq:def-drift}, which corresponds to slightly bending the unreachable space.

    \item \textbf{Analysis argument.}
    \cref{s:interpolation-smooth} studies interpolation inequalities to bound from above the remainder $\|u_1\|_{L^{M+1}}^{M+1}$ of \eqref{Heuristic:rep_form}, as well as the other terms appearing in \eqref{eq:eta-xi-approx} and \eqref{eq:borduk-approx}.
    This part is purely functional analysis (it involves neither algebraic considerations, nor the geometric assumption \eqref{Heuristic:Hyp_non_STLC}).
    Its main result \cref{p:interpol-Qjk-smooth} entails that
    \begin{equation}
        \label{eq:heuristic-interpol}
        \|u_1\|_{L^{M+1}}^{M+1} \lesssim
        \|u\|_{W^{m,\infty}}^{M-4} \times \xi_\q(t,u)
    \end{equation}
    for a good choice of $m$ and $M$.
    When $j = k$, this follows easily from usual Gagliardo--Nirenberg--Sobolev inequalities.
    However, when $j < k$, we rely on new multiplicative interpolation inequalities which we recently derived (for this purpose) in \cite{Marbach2023}.
\end{itemize}
We combine these arguments in \cref{s:obs-smooth-proof-drift} to conclude the proof.

\subsubsection{Estimates on coordinates of the second kind}

We will use the following crude universal estimate on coordinates of the second kind.

\begin{lemma}
    \label{lem:xi-naive-inf}
    For all $t > 0$, $u \in \lone$ and $b \in \Bs$,
    \begin{equation}
        |\xi_b(t,u)| \leq t^{n_0(b)} \| u_1 \|_{L^\infty}^{n_1(b)}.
    \end{equation}
\end{lemma}

\begin{proof}
    We proceed by induction on the length $|b| \geq 1$.
    First, from \eqref{xi_X0} and \eqref{xi_S1}, $\xi_{X_0}(t,u) = t$ and $\xi_{X_1}(t,u) = u_1(t)$ satisfy the estimate.
    Now, given $b \in \Bs \setminus X$, write its maximal factorization as in \cref{def:factorization}.
    By definition of the coordinates of the second kind (see \cite[Definition 62]{BeauchardLeBorgneMarbach2023})
    \begin{equation}
        \label{eq:xib-factorization}
        \xi_b(t,u) = \int_0^t \frac{\xi_{b_r}^{m_r}}{m_r!} \dotsb \frac{\xi_{b_1}^{m_1}}{m_1!} (s,u) \dd s.
    \end{equation}
    From \eqref{eq:factorization}, we get $|b| = 1 + m_1 |b_1| + \dotsb + m_r |b_r|$.
    Thus $|b_i| < |b|$ and we can apply the induction assumption to $b_i$.
    From \eqref{eq:factorization}, we also get the relations $n_0(b) = 1 + m_1 n_0(b_1) + \dotsb + m_r n_0(b_r)$ and $n_1(b) = m_1 n_1(b_1) + \dotsb + m_r n_1(b_r)$.
    Substituting the bounds for the $b_i$ in \eqref{eq:xib-factorization}, dropping the denominators and the integration factor $1/n_0(b)$ proves the claimed bound for $b$.
\end{proof}

\subsection{Algebraic computation of the coordinate of the pseudo-first kind}
\label{s:eta-Qjk-smooth}

We prove \cref{p:eta_Qjk-smooth}, which quantifies the heuristic $\eta_\q \approx \xi_\q$, using \cref{p:etab-xib-XI}.
As recalled in \cref{s:eta}, this estimate wraps up the Baker--Campbell--Hausdorff formula which relates the coordinates $\xi$ and $\eta$.
Proceeding too naïvely (see also \cref{rk:w1-naive}), e.g.\ reasoning simply by homogeneity with respect to $t$ and $u$ as in \cref{lem:xi-naive-inf} would only yield a very weak bound such as
\begin{equation}
    | \eta_\q(t,u) - \xi_\q(t,u) | \lesssim t^{2k+2j-3} \| u_1 \|_{L^\infty}^4.
\end{equation}
Here, the right-hand side is functionally much stronger than $\xi_\q$ itself. 
To improve this bound, the next paragraph analyzes when $\q$ lies in the support of a Lie bracket of elements of $\Bs$.

\subsubsection{Algebraic decompositions in $\Bs$}
\label{s:algebra-smooth}

Using \cref{s:hall-sets}'s notations, we present structural lemmas for Lie bracket decompositions on $\Bs$.

\begin{lemma} \label{p:2+2}
    Let $B_1, B_2 \in S_2(X)$.
    Then $\langle [B_1, B_2], \q \rangle_{\Bs} = 0$.
\end{lemma}

\begin{proof}
    By linearity and since $\Bs_2$ is a basis of $S_2(X)$, it suffices to prove that, for all $b_1 < b_2 \in \Bs_2$, one has $\langle [b_1, b_2], \q \rangle_{\Bs} = 0$.
    Let $b \in \supp_{\Bs} [b_1, b_2]$.
    By \cref{p:hall-left}, one has $\lambda(b) \geq b_1$.
    Since the order on $\Bs$ is compatible with $n_1$, this implies that $n_1(\lambda(b)) \geq n_1(b_1) = 2$.
    Since $\lambda(\q) = M_{k-1}$, one has $n_1(\lambda(\q)) = 1$ so $\lambda(\q) < \lambda(b)$ and thus $b \neq \q$.
    Hence $\langle [b_1, b_2], \q \rangle_{\Bs} = 0$.
\end{proof}

\begin{lemma} \label{p:1+3-smooth}
    Let $\ell > k$ and $B_3 \in S_3(X)$.
    Then $\langle [M_{\ell-1}, B_3], \q \rangle_{\Bs} = 0$.
\end{lemma}

\begin{proof}
    By linearity and since $\Bs_3$ is a basis of $S_3(X)$, it suffices to prove the result for $B_3 = b_3 \in \Bs_3$. 
    Let $a \in \supp_{\Bs} [M_{\ell-1}, b_3]$.
    By \cref{p:hall-left}, $\lambda(a) \geq M_{\ell - 1} > M_{k-1} = \lambda(\q)$, so $a \neq \q$.
\end{proof}

\subsubsection{Estimate of the difference $\eta_\q - \xi_\q$}

We bound $\eta_\q - \xi_\q$ using the boundary term $U_k(t) := (u_1,\dotsc,u_k)(t)$ introduced in \eqref{eq:borduk}.

\begin{proposition} \label{p:eta_Qjk-smooth}
    There exists $C > 0$ such that, for all $t \in (0,1]$ and $u \in \lone$,
    \begin{equation} \label{eq:eta_Qjk-smooth}
        |\eta_\q(t,u) - \xi_\q(t,u)|
        \leq C |U_k(t)| \| u_1 \|_{L^\infty}^3.
    \end{equation}
\end{proposition}

\begin{proof}
    We intend to apply \cref{p:etab-xib-XI}.
    Let $q \geq 2$, $b_1 \geq \dotsb \geq b_q \in \Bs \setminus \{ X_0 \}$ such that $Q_{j,k,k} \in \supp_{\Bs} \mathcal{F}(b_1, \dotsc, b_q)$ (see \cref{def:calF}).
    Since $n_1(Q_{j,k,k}) = 4$, $q \leq 4$.
    \begin{itemize}
        \item \emph{Case $q = 4$.}
        Then each $b_i \in \Bs_1$ so is of the form $M_{l_i-1}$ for $l_i \ge 1$.
        Since $b_1 \ge \dotsb \ge b_4 \in \Bs_1$, $l_1 \ge \dots \ge l_4$.
        Since $n_0(b_1) + \dotsb + n_0(b_4) = n_0(\q) \le 4 k - 3$, $l_4 \le k$.
        By \eqref{xi_S1}, $\xi_{b_4}(t,u) = u_{l_4}(t)$.
        Since $l_4 \le k$, $|\xi_{b_4}(t,u)| \le |U_k(t)|$.
        Using \cref{lem:xi-naive-inf} to estimate $\xi_{b_1} \xi_{b_2} \xi_{b_3}$ and $t \le 1$,
        \begin{equation*}
            |\xi_{b_1} \xi_{b_2} \xi_{b_3} \xi_{b_4} |(t,u) \leq C |U_k(t)| \|u_1\|_{L^\infty}^3.
        \end{equation*}

        \item \emph{Case $q = 3$.}
        Then $b_1 \in \Bs_2$ and $b_2, b_3 \in \Bs_1$.
        By \cref{p:2+2} and \cref{p:1+3-smooth} one has $b_3 = M_{l_3-1}$ where $l_3 \leq k$.
        By \eqref{xi_S1}, $|\xi_{b_3}(t,u)| = |u_{l_3}(t)| \leq |U_k(t)|$.
        Using \cref{lem:xi-naive-inf} to estimate $\xi_{b_1} \xi_{b_2}$ and $t \le 1$,
        \begin{equation*}
            |\xi_{b_1} \xi_{b_2} \xi_{b_3}|(t,u) \leq C |U_k(t)| \|u_1\|_{L^\infty}^3.
        \end{equation*}

        \item \emph{Case $q = 2$.}
        By \cref{p:2+2}, one cannot have, $b_1, b_2 \in \Bs_2$. 
        So $b_1 \in \Bs_3$ and $b_2 = M_{l_2-1} \in \Bs_1$.
        By \cref{p:1+3-smooth}, $l_2 \leq k$.
        By \eqref{xi_S1}, $|\xi_{b_2}(t,u)| = |u_{l_2}(t)| \leq |U_k(t)|$.
        Using \cref{lem:xi-naive-inf} to estimate $\xi_{b_1}$ and $t \le 1$, 
        \begin{equation*}
            |\xi_{b_1} \xi_{b_2} |(t,u) \leq C |U_k(t)| \|u_1\|_{L^\infty}^3.
        \end{equation*}
    \end{itemize}
    Thus estimate \eqref{eq:eta_Qjk-smooth} follows from \cref{p:etab-xib-XI}.
\end{proof}

\begin{remark} \label{rk:w1-naive}
    In the case $q = 2$, excluding $b_1, b_2 \in \Bs_2$ via \cref{p:2+2} is critical. Without it, one would have to allow $b_1 = W_{1,\nu_1}$ and $b_2 = W_{1,\nu_2}$, for which one cannot do much better than $|\xi_{b_1} \xi_{b_2}|(t,u) \lesssim t^{2j+2k-5} \|u_1\|_{L^2}^4$. 
    For $j, k > 1$, this bound can be much larger than $\xi_\q(t,u)$.
\end{remark}

\subsection{Closed-loop estimates resulting from vectorial relations}
\label{s:closed-loop-smooth}

We prove \cref{p:borduk-smooth} which is a ``closed-loop estimate'' of the form \eqref{eq:borduk-approx} for the ``pointwise boundary term'' $U_k(t) = (u_1(t), \dotsc, u_k(t))$ from \eqref{eq:borduk}.
We establish that the quantities $u_l(t)$ are ``part of'' the state $x(t;u)$ in the sense that one can find components $\mathbb{P}_l$ such that
\begin{equation}
    \mathbb{P}_l x(t;u) = u_l(t) + \text{remainder}.
\end{equation}
To obtain a small remainder, we construct the components $\mathbb{P}_l$ parallel to the largest possible subspaces.
This requires  ``vectorial relations'', which are consequences of the assumption on $f_\q(0)$.

\subsubsection{Vectorial relations}

In this paragraph, we introduce a vectorial relation \eqref{eq:H-smooth} and relate its validity with appropriate assumptions on the system.
The proof relies on the following key observation.

\begin{lemma}
    \label{lem:Lie-subalg}
    Let $B_1, B_2 \in \mathcal{L}(X)$ such that $f_{B_1}(0) = f_{B_2}(0) = 0$.
    Then $f_{[B_1, B_2]}(0) = 0$.
\end{lemma}

In particular, since $f_0(0) = 0$, if $f_B(0) = 0$, then $f_{B 0}(0) = [f_B, f_0](0) = 0$.

\begin{lemma}
    \label{lem:vect-Qjk-smooth}
    Assume that $f_{W_j}(0) \in S_1(f)(0)$ and $f_\q(0) \notin \mathcal{N}_{j,k}^4(f)(0)$.
    Then
    \begin{equation}
        \label{eq:H-smooth}
        \tag{H}
        \forall i \in \intset{0, k-1}, \quad 
        f_{M_i}(0) \notin \vect \{ f_{M_j}(0) \mid j \neq i \}.
    \end{equation}
\end{lemma}

\begin{proof}
    By contradiction, assume that there exist $N \geq i$ and $\alpha_0, \dotsc, \alpha_N \in \R$ with $\alpha_i = 1$ such that $f_B(0) = 0$ for $B := \alpha_0 M_0 + \dotsb + \alpha_N M_N $. 
    Let $\nu := \min \{ n \mid \alpha_n \neq 0 \} \leq i$.
    Then $f_{B_1}(0) = 0$ for $B_1 := \frac{1}{\alpha_\nu} B 0^{k-1-\nu}$ and there exists $M' \in \vect \{ M_j \mid j \geq k \}$ such that $B_1 = M_{k-1} + M'$.
    
    Since $f_{W_j}(0) \in S_1(f)(0)$, there exists $M'' \in S_1(X)$ such that $f_{W_j + M''}(0) = 0$.
    
    Hence $f_{B_2}(0) = 0$ for $B_2 := \ad_{B_1}^2(W_j + M'')$.
    Expanding $B_2$, we obtain $B_2 \in Q_{j,k,k} + \vect \mathcal{N}^4_{j,k}$.
    Indeed, the terms involving $M''$ have $n_1 = 3$ and the quartic terms involving at least one occurrence of $M'$ have $n_0 > n_0(Q_{j,k,k})$.
    Hence $f_\q(0) \in \mathcal{N}_{j,k}^4(f)(0)$.
\end{proof}

\subsubsection{Closed-loop estimates}

We estimate $U_k(t)$ using \cref{thm:Magnus} and the vectorial relations \eqref{eq:H-smooth}.

\begin{proposition} \label{p:borduk-smooth}
    Assume that \eqref{eq:H-smooth} holds.
    As $(t, \|u_1\|_{L^\infty}) \to 0$, we have
    \begin{equation}
        | U_k(t) | = O\left(\|u_1\|_{L^2}^2 + |x(t;u)|\right).
    \end{equation}
\end{proposition}

\begin{proof} 
    We apply \cref{thm:Magnus}.
    By \cref{eq:Magnus} with $M = 1$,
    \begin{equation} \label{x=ZN+u^N+1-smooth}
        x(t;u)=\mathcal{Z}_1(t,u)(0)+O\left( \|u_1\|_{L^2}^{2} + |x(t;u)|^2 \right). 
    \end{equation}
    By \cref{p:small-state}, as $\|u_1\|_{L^\infty} \to 0$,
    \begin{equation}
        \label{eq:x32-Ox}
        |x(t;u)|^2 = O(|x(t;u)|).
    \end{equation}
    By \eqref{eq:ZM=eta}, \cref{cor:eta-S1} and \eqref{xi_S1},
    \begin{equation}
        \mathcal{Z}_1(t,u)(0)=    
        \sum_{l=1}^{\infty} u_{l}(t) f_{M_{l-1}}(0).
    \end{equation}
    Let $l \in \intset{1,k}$ and $\mathcal{N}:=\Bs_1 \setminus \{M_{l-1}\}$. 
    Assumption \eqref{eq:H-smooth} allows to consider $\mathbb{P}:\R^d \to \R$ giving a component along $f_{M_{l-1}}(0)$ parallel to $\mathcal{N}(f)(0)$. 
    Then
    \begin{equation} \label{PZ=ul-smooth}
        \mathbb{P} \mathcal{Z}_1(t,u)(0) = u_{l}(t).
    \end{equation}
    By combining \eqref{x=ZN+u^N+1-smooth}, \eqref{eq:x32-Ox} and \eqref{PZ=ul-smooth}, we obtain $u_l(t)=O\left( \|u_1\|_{L^{2}}^{2} + |x(t;u)| \right)$.
\end{proof}

\subsection{Interpolation inequalities}
\label{s:interpolation-smooth}

We start by recalling old and recent interpolation inequalities on derivatives of functions, before translating them to our specific context in the next paragraph, preparing the main proof.

\subsubsection{Interpolation inequalities for derivatives}

Recall the usual Gagliardo--Nirenberg interpolation inequality (see \cite{Gagliardo1959,Nirenberg1959}).

\newcommand{\fj}{{J}}
\newcommand{\fk}{{K}}
\newcommand{\fl}{{L}}

\begin{proposition} 
    \label{thm:GN}
    Let $p, q, r, s \in [1,\infty]$, $0 \leq \fj < \fl \in \N$, $\theta^* := \fj / \fl$, and $\theta \in [\theta^*,1]$ such that
    \begin{equation} \label{eq:GNS-condition}
        \frac{1}{p}-\fj= \theta \left( \frac{1}{r}-\fl\right) + (1-\theta) \frac{1}{q}.
    \end{equation}
    There exists $C > 0$ such that, for every $T > 0$ and $\phi \in W^{\fl,r}([0,T];\R)$,
    \begin{equation} \label{eq:GNS-estimate}
        \| D^\fj \phi \|_{L^p} \leq C \| D^\fl \phi \|_{L^r}^\theta \|\phi\|_{L^q}^{1-\theta}  + C T^{\frac{1}{p}-\fj-\frac{1}{s}} \|\phi\|_{L^s}.
    \end{equation}
\end{proposition} 

\begin{corollary}
    \label{cor:GN}
    Under the same assumptions, if moreover $\phi(0) = \dotsb = D^{L-1}\phi(0) = 0$,
    \begin{equation}
        \| D^\fj \phi \|_{L^p} \leq C \| D^\fl \phi \|_{L^r}^\theta \|\phi\|_{L^q}^{1-\theta}.
    \end{equation}
\end{corollary}

\begin{proof}
    Choose $s = q$ and use the conditions at $t = 0$ to prove that $\|\phi\|_{L^q} \leq T^{\frac 1 q - \frac 1 r + L} \| D^L \phi \|_{L^r}$.
    Substituting in \eqref{eq:GNS-estimate} proves the claim using \eqref{eq:GNS-condition}.
\end{proof}

The usual Gagliardo--Nirenberg interpolation inequality allows to prove drifts when the coordinate $\xi_\q$ of the bad bracket is given by a usual Sobolev norm of the control.
This idea is a key point of our quadratic obstructions results of \cite{BeauchardMarbach2018,BeauchardMarbach2026}.
In our quartic context, it still suffices in the easier case $j = k$ (see \cref{s:j=k}).
However, when $j < k$, we need an interpolation inequality able to take into account the non-standard term $\int u_j^2 u_k^2$, which is not a Sobolev norm.

We proved in \cite{Marbach2023} a generalization of the usual Gagliardo--Nirenberg inequality of \cref{thm:GN} which replaces the term $\|\phi\|_{L^q}$ by arbitrary pointwise products of the derivatives of $\phi$. 
In particular, \cite[Corollary 1.7, case $\kappa = 2$]{Marbach2023} entails the following estimates rescaled on $[0,T]$.

\begin{proposition} \label{thm:FM-GN}
    Let $p, q, r \in [1,\infty]$, $0 \leq \fk \leq \fj < \fl \in \N$,
	\begin{equation} \label{eq:theta*-FM-GN}
		\theta^* := \frac{\fj - \fk/2}{\fl - \fk/2}.
	\end{equation}
    and $\theta \in [\theta^*,1]$ such that 
	\begin{equation} \label{eq:main-relation-FM-GN}
		\frac{1}{p}-\fj = \theta \left( \frac{1}{r} - \fl \right)
		+ (1-\theta) \left(\frac{1}{2 q} - \frac{\fk}{2} \right).
	\end{equation}
    There exists $C >0 $ such that, for all $T >0$ and $\phi \in W^{\fl,r}((0,T);\R)$,
    \begin{equation}
        \label{eq:GN-FM-estimate}
        \| D^\fj \phi \|_{L^p} \leq C 
        \| D^\fl \phi \|_{L^r}^\theta
        \| \phi D^\fk \phi \|_{L^q}^{(1-\theta)/2} 
        + C T^\alpha \| \phi D^\fk \phi \|_{L^q}^{1/2}
    \end{equation}
    where $\alpha := \frac{\fk}{2} - \frac{1}{2q} - \fj + \frac{1}{p}$.
\end{proposition}

\begin{corollary}
    \label{cor:FM-GN}
    Under the same assumptions, if moreover $\phi(0) = \dotsb = D^{L-1}\phi(0) = 0$,
    \begin{equation}
        \label{eq:GN-FM-estimate-cor}
        \| D^\fj \phi \|_{L^p} \leq C 
        \| D^\fl \phi \|_{L^r}^\theta
        \| \phi D^\fk \phi \|_{L^q}^{(1-\theta)/2}.
    \end{equation}
\end{corollary}

\begin{remark} \label{rmk:pqr-FM-GN}
	In the critical case $\theta = \theta^*$, the relation \eqref{eq:main-relation-FM-GN} is equivalent to
	\begin{equation} \label{eq:critic-pqr-FM-GN}
		\frac{1}{p} = \frac{\theta}{r} + \frac{1-\theta}{2q}.
	\end{equation}
\end{remark}

\subsubsection{Interpolation inequalities based on $\xi_\q$}

We derive the following consequence of \cref{thm:FM-GN} in our context.

\begin{proposition}
    \label{p:interpol-Qjk-smooth}
    There exists $C > 0$ such that, for all $T \in (0,1]$ and $u \in W^{2k+2j-4,\infty}_0((0,T);\R)$,
    \begin{equation}
        \label{eq:u1-L5}
        \| u_1 \|_{L^\infty}^5 \leq C \| u \|_{W^{2k+2j-4,\infty}} \xi_\q(T,u).
    \end{equation}
\end{proposition}

\begin{proof}
    Set $m := 2 k + 2 j - 4$.
    Apply \cref{cor:FM-GN} with $\phi \gets u_k$, $(K, J, L) \gets (k - j, k - 1, k + m)$, $(p, q, r) \gets (\infty, 2, \infty)$ and $\theta = \frac 1 5$.
    These parameters satisfy $0 \leq K \leq J < L$, \eqref{eq:main-relation-FM-GN} and $\theta^* < \theta \leq 1$.
    Then \eqref{eq:GN-FM-estimate-cor} yields
    \begin{equation}
        \| u_1 \|_{L^\infty}^5 
        = \| D^{k-j} u_k \|_{L^\infty}^5
        \leq C \| D^{k+m} u_k \|_{L^\infty}^1 \| u_k D^{k-j} u_k \|_{L^2}^2
        = C \| D^m u \|_{L^\infty} \| u_k u_j \|_{L^2}^2,
    \end{equation}
    which is exactly \eqref{eq:u1-L5}, using \eqref{xi_Qljknu} to recognize the lower-order term as $\xi_\q(T,u)$. 
\end{proof}

\subsection{Proof of the presence of the drift}
\label{s:obs-smooth-proof-drift}

We prove \cref{thm:Qjk-smooth} as a consequence of the following statement, which denies $W^{m,\infty}_0$-STLC for $m = 2k + 2j -4$, and thus smooth-STLC too.

\begin{theorem} 
    \label{thm:qjk-drift-smooth}
    Let $j \leq k \in \N^*$.
    Assume that $f_\q(0) \notin \mathcal{N}_{j,k}^4(f)(0)$ and\footnote{The assumption $f_{W_j}(0) \in S_1(f)(0)$, not included in \cref{thm:Qjk-smooth}, is legitimate because it is a necessary condition for smooth-STLC, see \cite[Theorem 3]{BeauchardMarbach2018} or \cite[Theorem 1.11]{BeauchardMarbach2026}.} $f_{W_j}(0) \in S_1(f)(0)$.
    Then, system \eqref{syst} has a drift along $f_\q(0)$, parallel to $\mathcal{N}_{j,k}^4(f)(0)$, as $T \to 0$ and $\|u\|_{W^{m,\infty}} \to 0$ for controls $u \in W^{m,\infty}_0((0,T);\R)$ where $m := 2k + 2j - 4$.
\end{theorem}

\begin{proof}
    By \cref{eq:eta_Qjk-smooth} of \cref{p:eta_Qjk-smooth}, we obtain
    \begin{equation} 
        \label{eta_Qjk_1-smooth}
        \eta_\q(T,u) = \xi_\q(T,u)
           + O\left(  |U_k(T)| \| u_1\|_{L^\infty}^3 \right).
    \end{equation}
    By \cref{lem:vect-Qjk-smooth} and \cref{p:borduk-smooth},
    \begin{equation}
        U_k(T) = O\left( \| u_1 \|_{L^\infty}^2 + |x(T;u)| \right).
    \end{equation}
    And thus, by Young's inequality,
    \begin{equation}
        \label{eq:qjk-smooth-proof-eta-xi-q}
        |\eta_\q(T,u) - \xi_\q(T,u)| 
        = O\left( \|u_1\|_{L^\infty}^5 + |x(T;u)|^{\frac 5 2} \right).
    \end{equation}
    Let $\mathbb{P} :\R^d \to \R$ be a component along $f_\q(0)$ parallel to $\mathcal{N}_{j,k}^4(f)(0)$.
    By \cref{thm:Magnus},
    \begin{equation} \label{Magnus_Qjk-smooth}
        \mathbb{P} x(T;u) =  \eta_\q(T,u) + O\left( 
        \|u_1\|_{L^\infty}^{5} + |x(T;u)|^{1+\frac{1}{4}} \right). 
    \end{equation}
    Combining \eqref{eq:qjk-smooth-proof-eta-xi-q} and \eqref{Magnus_Qjk-smooth} and using \cref{p:interpol-Qjk-smooth} to estimate $\|u_1\|_{L^\infty}^5$, we obtain that there exists $C > 0$ such that, for all $T \in (0,1]$ and $u \in W^{m,\infty}_0((0,T);\R)$ small enough,
    \begin{equation}
    \begin{split}
         \left| \mathbb{P} x(T;u) -  \xi_\q(T,u) \right| 
         \leq 
          C \|u\|_{W^{m,\infty}}
         \xi_\q(T,u) + 
         C |x(T;u)|^{1+\frac 14},
    \end{split}
    \end{equation}
    which proves the drift in the sense of \cref{def:drift}.

    More precisely, for all $\varepsilon > 0$, if $T \le 1$ and $\| u \|_{W^{m,\infty}} \le \varepsilon / C$, \eqref{eq:def-drift} holds with $\beta = \frac 14$.
\end{proof}

\newpage 
\section{The classification problem}
\label{s:classification}

We discuss the classification problem, introduced by Kawski in \cite[Section 4]{Kawski1987_Survey}.
Heuristically, the goal of this problem is to find bases of $\mathcal{L}(X)$ which can be partitioned in ``good'' and ``bad'' brackets, for appropriate definitions of these notions.
Our result \cref{thm:classification} provides an answer up to quartic brackets.
We discuss some of the subtleties of this notion and give an example.

\subsection{Minor brackets which are only indirectly bad}
\label{s:minor}

In our context, and focusing on the notion of smooth-STLC (recall \cref{def:SSTLC}), we proved:
\begin{itemize}
    \item in \cref{s:good-4} that the $Q_{j,k,l,\nu}$ for $k < l$ can be considered as good,
    \item in \cref{s:obs-smooth} that the $Q_{j,k,k}$ can be considered as bad.
\end{itemize}
Thus, recalling the quartic brackets of $\Bs_4$ listed in \eqref{def:Q}, it remains to determine the status of the brackets $Q_{j,k,k,\nu}$ for $\nu > 0$ and $Q^\sharp_{j,\mu,k,\nu}$ and $Q^\flat_{j,\mu,\nu}$.

\medskip

Our viewpoint is that these brackets are ``minor'' and only ``indirectly bad''.
We do not want to label them as ``good'' because their coordinates of the second kind are signed, and we cannot write smoothly-STLC systems involving them on a single line.
We also do not want to label them as being directly ``bad'' because the occurrence of a drift in their direction requires the simultaneous presence of a stronger drift in another direction. 
More precisely, the following result holds.

\begin{lemma} 
    \label{lem:bad-stab-X0}
    Let $\bb \in \Bs$ such that \eqref{eq:fbb-comp} holds.
    Then $\bb 0 = (\bb, X_0)$ satisfies \eqref{eq:fbb-comp} too.
\end{lemma}

\begin{proof}
    By assumption, there exists a finite family $C \subset \Bs \setminus \{ \bb \}$ such that $f_\bb(0) = \sum_{c \in C} \alpha_c f_c(0)$ with $\alpha_c \in \R$.
    Since $f_0(0) = 0$, by \cref{lem:Lie-subalg}, $f_{\bb 0} (0) = \sum_{c \in C} \alpha_c f_{c 0}(0)$.
    For all $c \in C \subset \Bs \setminus \{ \bb \}$, $c 0 \neq \bb 0$ by injectivity of $\ad_{X_0}$ on $\Br(X)$.
    So $\bb 0$ satisfies \eqref{eq:fbb-comp}.
\end{proof}

\begin{lemma} 
    \label{lem:sharp-flat-S3}
    Assume that, for all $b \in \Bs_2$, $f_b(0) \in S_1(f)(0)$.
    Then
    \begin{itemize}
        \item for all $j < k \in \N^*$ and $\mu,\nu \in \N$, $f_{Q^\sharp_{j,\mu,k,\nu}}(0) \in S_{\intset{1,3}}(f)(0)$,
        \item for all $j \in \N^*$ and $\mu,\nu \in \N$, $f_{Q_{j,\mu,\nu}^\flat}(0) \in S_{\intset{1,3}}(f)(0)$.
    \end{itemize}
\end{lemma}

\begin{proof}
    All these statements are straightforward applications of \cref{lem:Lie-subalg}.
    Since $f_0(0) = 0$ and $S_{\intset{1,3}}(X)$ is stable by $\ad_{X_0}$ it suffices to prove the result for $\nu = 0$.

    Let $1 \leq j < k$ and $\mu \ge 0$.
    By assumption, there exist two elements $M, M' \in S_1(X)$ such that $f_{W_{j,\mu}}(0) + f_M(0) = 0$ and $f_{W_k}(0) + f_{M'}(0) = 0$.
    Thus $f_B(0) = 0$ where 
    \begin{equation}
        B = [W_{j,\mu} + M, W_k + M'] \in Q^\sharp_{j,\mu,k} + S_{\intset{2,3}}(X).
    \end{equation}
    This proves that $f_{Q^\sharp_{j,\mu,k}}(0) \in S_{\intset{1,3}}(f)(0)$.

    Given $j \geq 1$ and $\mu \in \N$, one proceeds similarly for $Q^\flat_{j,\mu} = (W_{j,\mu}, W_{j,\mu+1})$.
\end{proof}

\cref{lem:bad-stab-X0,lem:sharp-flat-S3} were used in \cref{s:main-results} to prove our classification result \cref{thm:classification}.

\subsection{Detailed classification of quartic brackets of up to five \texorpdfstring{$X_0$}{X0}}
\label{s:S45}

In order to provide a more visual and detailed answer to Kawski's open problem of \cite[Section 4]{Kawski1987_Survey} which concerned the classification of $S_{4,3}$ and $S_{4,5}$, we give an explicit list of these brackets along with their coordinates of the second kind and their nature.

For $n_0,n_1 \in \N^*$, the dimension of $S_{n_1,n_0}$ can be computed using the following formula (see e.g.\ \cite[Section 2, eq.\ (2)]{Moree2005}), stemming from Witt's formula \cite{Witt1956}:
\begin{equation} \label{eq:witt}
    \dim S_{n_1,n_0} = \frac{1}{n_0+n_1} \sum_{d|\operatorname{gcd}(n_0,n_1)} \mu(d)\frac{\left(\frac{n_0+n_1}{d}\right)!}{\left( \frac{n_0}{d}\right)!\left(\frac{n_1}{d} \right)!}
\end{equation}
where $\mu$ is the Möbius function, in particular $\mu(1)=1$, $\mu(2)=\mu(3)=-1$, $\mu(4)=0$.
In particular, one has $\dim S_{4,1} = 1$, $\dim S_{4,2} = 2$, $\dim S_{4,3} = 5$, $\dim S_{4,4} = 8$ and $\dim S_{4,5} = 14$.

In \cref{s:tables}, we give the full list of elements of $\Bs$ with $n_1 = 4$ and $n_0 \leq 5$, along with their coordinates of the second kind.
As examples:
\begin{itemize}
    \item the 5 elements of $\Bs_{4,3}$, including 2 good and 3 bad (of which 2 minor), are in \cref{tab:S43} (p.~\pageref{tab:S43});
    \item the 14 elements of $\Bs_{4,5}$, including 6 good and 8 bad (of which 6 minor), are in \cref{tab:S45} (p.~\pageref{tab:S45}).
\end{itemize}
As in \cref{p:canonical}, the nice form of $\Bs$ allows to write vector fields whose system is given by the associated coordinates of the second kind.
In the case of $\Bs_{4,3}$, we would obtain a system very similar to the one given by Kawski in \cite[eq.\ (22)]{Kawski1987_Survey}.
We give the system for $\Bs_{4,5}$ as a more complex illustration, thereby showing that our basis $\Bs$ solves the classification problem for $S_{4,5}$ (and in fact all $S_4$, since our results of the previous sections are not limited to $n_0 \leq 5$).

We consider the following system whose state is $(x,y) \in \R^{23+14}$:
\begin{equation}
    \left\{
    \begin{array}{lllll}
        \dot{x}_1 = u & \dot{x}_9 = x_8 & \dot{x}_{17} = \frac{1}{6} x_1^3 x_2 & \dot{y}_1 = x_{19} & \dot{y}_9 = x_{11} \\
        \dot{x}_2 = x_1 & \dot{x}_{10} = x_9 & \dot{x}_{18} = x_{17} & \dot{y}_2 = x_{21} & \dot{y}_{10} = x_{13} \\
        \dot{x}_3 = x_2 & \dot{x}_{11} = x_{10} & \dot{x}_{19} = x_{18} & \dot{y}_3 = x_{22} & \dot{y}_{11} = x_{15} \\
        \dot{x}_4 = x_3 & \dot{x}_{12} = \frac{1}{4} x_1^2 x_2^2 & \dot{x}_{20} = \frac{1}{6} x_1^3 x_3 & \dot{y}_4 = \frac{1}{9}x_1^3 x_5 & \dot{y}_{12} = \frac 1 2 x_7^2 \\
        \dot{x}_5 = x_4 & \dot{x}_{13} = x_{12} & \dot{x}_{21} = x_{20} & \dot{y}_5 = x_{23} & \dot{y}_{13} = x_{16} \\
        \dot{x}_6 = \frac 1 2 x_1^2 & \dot{x}_{14} = \frac 1 2 x_6^2 & \dot{x}_{22} = \frac{1}{6} x_1^3 x_4 & \dot{y}_6 = \frac 12 x_1^2 x_2 x_4 \quad\quad & \dot{y}_{14} = \frac{1}{2} x_7 x_2^2\\
        \dot{x}_7 = x_6 & \dot{x}_{15} = x_{14} & \dot{x}_{23} = \frac{1}{2} x_1^2 x_2 x_3 \quad\quad & \dot{y}_7 = \frac 14 x_1^2 x_3^2& \\
        \dot{x}_8 = \frac{1}{24} x_1^4 \quad\quad & \dot{x}_{16} = \frac 12 x_6 x_2^2 \quad\quad & & \dot{y}_8 = \frac{1}{24} x_2^4 & \\
    \end{array}
    \right.
\end{equation}
In this system, the $y$ part, of dimension 14, represents a basis of the $S_{4,5}$ subspace.
By \cref{p:canonical}, its brackets $f_b(0)$ for $b\in \Bs(X)$ with $n_1(b) \leq 4$ and $n_0(b) \leq 5$ are given by the associated canonical basis of~$\R^{37}$.
It is straightforward to check that, for any control $u \in L^1(0,T)$ and any $t \in [0,T]$, one has $y_i(t;u) \geq 0$ for $i \in \intset{7,14}$ (which correspond to the bad and minor brackets).
In Kawski's vocabulary of \cite[Section 4]{Kawski1987_Survey}, one has 8 \emph{supporting hyperplanes}.
Moreover, by \cref{thm:xi4-surj}, the subsystem $(x_{\intset{1,5}},x_{\intset{17,23}},y_{\intset{1,6}})$ is smoothly-STLC, where the last six lines correspond to the 6 good brackets of $\Bs_{4,5}$.

\subsection{On some intricacies of the classification notion}
\label{s:classification-intricacies}

We claim that our quartic results constitute an important progress in the sense that, for the first time, there is a kind of complementarity between the good brackets which can be used simultaneously and the bad ones which must be compensated.
Nevertheless, there are still many things to be understood, even at the quartic level.

\subsubsection{Bad-bad competitions and the cardinality of classifications}
We start with a definition.

\begin{definition}
    Let $k \in \N^*$ and $B$ be a bihomogeneous basis of $\mathcal{L}(X)$ (not necessarily a Hall basis).
    Let us say that $B$ \emph{weakly} (respectively \emph{strongly}) \emph{classifies} $S_{\intset{1,k}}(X)$ if there is a partition $B_\good \cup B_\bad$ of $B \cap S_{\intset{1,k}}(X)$ such that  both of the following properties hold:
    \begin{itemize}
        \item If $f_0,f_1$ are such that $B_\good(f)(0) = \R^d$ and $B_\bad(f)(0) = \{ 0 \}$, then \eqref{syst} is smoothly-STLC.
        \item If \eqref{syst} is smoothly-STLC, then for every $\bb \in B_\bad$,
        \begin{equation} \label{eq:def-weak-classif}
            f_\bb(0) \in \vect \{ f_b(0) ; b \in (B_\bad \cup B_\good) \setminus \{ \bb \} \},
        \end{equation}
        (respectively 
        \begin{equation} \label{eq:def-strong-classif}
            f_\bb(0) \in \vect \{ f_b(0) ; b \in B_\good \} ).
        \end{equation}
    \end{itemize}
\end{definition}

Using this definition, \cref{thm:classification} claims that $\Bs$ weakly classifies $S_{\intset{1,4}}$ while \cref{open:weak-classification} asks whether $\Bs$ weakly classifies the whole $\mathcal{L}(X)$.
As already mentioned in \cref{rk:comp-badbad} and illustrated by \cref{prop:magnitude}, $\Bs$ does not strongly classify $S_{\intset{1,4}}(X)$.
This leads to the following open problem.

\begin{open}
    Does there exist a basis of $\mathcal{L}(X)$ which strongly classifies $S_{\intset{1,4}}(X)$?
\end{open}

The motivation to replace \eqref{eq:def-weak-classif} by \eqref{eq:def-strong-classif} is twofold. 
First, it takes a step further towards finding conditions which are both sufficient and necessary for smooth-STLC.
Second, it aims to mitigate the following drawback of the notion of weakly classifying basis.

Assume that $B$ and $B'$ are two different bi-homogeneous basis weakly classifying $S_{\intset{1,4}}$.
Is it true that $\card (B_\good \cap S_{4,5}(X)) = \card (B'_\good \cap S_{4,5}(X))$?
In other words, in our decomposition of the 14 directions of $S_{4,5}(X)$ as 6 good, 2 bad and 6 minor, are these numbers basis-independent?
Even worse, it could happen that a given basis $B$ could admit two different weakly-classifying partitions.

The notion of strongly classifying basis could help resolve this lack of uniqueness.
For example, if $B$ and $B'$ were two different strongly classifying Hall bases with $X_0$ maximal, one could use \cref{p:canonical} to write their associated $\xi$ canonical systems, and prove the equality of the cardinalities. 
Moreover, as noted by Kawski in \cite[Section 5]{Kawski2002_Coordinates}, by the Lie algebra rank condition, one should always have $\card (B_\bad \cap S_{i,j}(X)) = 1$ (although the role of what we call ``minor'' brackets in the previous paragraph is unclear here). 

Unfortunately, while Hall bases play nicely with Lazard's elimination and control theory, it could happen that one has to consider more general bases to achieve strong classification.

\subsubsection{Working at the free level, with coordinates of the first or pseudo-first kind}

To overcome some of the ambiguities mentioned above, one can try to work at the free level, i.e.\ by attempting to describe the reachable set of the free equation 
\begin{equation}
    \label{eq:free}
    \dot{x}(t) = x(t) (X_0 + u(t)X_1)
    \quad \text{and} \quad x(0) = 1
\end{equation}
set in the formal series over $X$ (see e.g.\ \cite[Section 2]{BeauchardLeBorgneMarbach2023} for a gentle introduction).
Given a bihomogeneous basis $B$ of $\mathcal{L}(X)$, solutions to \eqref{eq:free} can be described using in particular the following three expansions (see \cite[Sections 2.3, 2.4 and 2.5]{BeauchardLeBorgneMarbach2023}):
\begin{equation}
    x(t) 
    = \exp \left( \sum_{b \in B} \zeta_b(t,u) b \right)
    = \exp (tX_0) \exp \left( \sum_{b \in B \setminus \{ X_0 \}} \eta_b(t,u) b \right)
    = \overset{\leftarrow}{\prod_{b \in B}} \exp (\xi_b(t,u) b).
\end{equation}
In particular, within the class of bihomogeneous bases, for $i,j \in \N^*$ the objects 
\begin{equation}
    \sum_{b \in B \cap S_{i,j}} \zeta_b(t,u)b
    \quad \text{and} \quad
    \sum_{b \in B \cap S_{i,j}} \eta_b(t,u)b
\end{equation}
are independent of $B$, so one could attempt to describe them.

A starting point is the one-dimensional subspace $S_{2,1}(X) = \R W_1$, which is responsible for the strongest obstruction, dating back to \cite{Sussmann1983}, and for which the basis is uniquely determined.
Using the Baker--Campbell--Hausdorff formula leads to
\begin{align} 
    \label{eq:zetaW1}
    \zeta_{W_1}(t,u) & = \frac{1}{2} \int_0^t u_1^2  - \frac{1}{2} u_1(t) u_2(t) + \frac{1}{12} t u_1^2(t), \\
    \label{eq:etaW1}
    \eta_{W_1}(t,u) & = \frac{1}{2} \int_0^t u_1^2  - \frac{1}{2} u_1(t) u_2(t).
\end{align}
In particular, one checks that for every $T > 0$, one can find $u^\pm \in \CC^\infty_c((0,T);\R)$ such that $\zeta_{W_1}(T,u^\pm) = \pm 1$ (and similarly for $\eta_{W_1}$).
Therefore, if one attempts to consider the behavior of the projections on $S_{i,j}(X)$ of $\sum \zeta_b b$ or $\sum \eta_b b$, one is led to the absurd conclusion that in $S_{2,1}(X)$, the  bracket $W_1$ is ``good'' as one can achieve a movement in both directions.

Therefore, one should not consider such projections independently on the projections along $S_{i',j'}(X)$ with $i' \leq i$ and $j' \leq j$.
Indeed, all the usual proofs of the easy necessary condition $f_{W_1}(0) \in S_1(f)(0)$ involve an argument bounding $|u_1(t)|$, which we call ``closed-loop estimate'' (see e.g.\ \cite[Section 5.3]{BeauchardMarbach2026}) and allows to ignore the boundary terms in \eqref{eq:zetaW1} and \eqref{eq:etaW1}, reducing them to their integral part, equal to $\xi_{W_1}(t,u) \geq 0$.

\newpage
\section{Bad-bad competition}
\label{s:bad-bad}

We prove \cref{prop:magnitude} which illustrates that, depending on the value of a numerical constant, the competition between two bad brackets can generate smooth-STLC.

We denote by $W^{3,1}_0(0,1)$ the Sobolev space of functions $\varphi$ such that $\varphi, \varphi', \varphi'', \varphi''' \in L^3(0,1)$ and $\varphi, \varphi', \varphi''$ vanish at $t = 0$ and $t = 1$.

\subsection{A functional inequality}

Let us start by giving the definition of the threshold constant $\lambda^*$, linked with a functional inequality.

\begin{lemma} \label{lem:lambda*}
    There exists a finite constant $\lambda^* > 0$ which is the supremum of the $\lambda > 0$ such that, for all $\varphi \in W^{3,1}_0(0,1)$, the following inequality holds:
    \begin{equation} \label{ineq:main}
        \lambda \int_0^1 (\varphi')^4 \leq \int_0^1 (\varphi \varphi'')^2.
    \end{equation}
\end{lemma}

\begin{proof}
    Since we consider functions $\varphi \in W^{3,1}_0(0,1)$, usual Sobolev embeddings entail that both sides of \eqref{ineq:main} are well-defined.
    
    First, let us prove that the set of such $\lambda > 0$ is not empty.
    Indeed, by integration by parts and the Cauchy--Schwarz inequality, for any $\varphi \in W^{3,1}_0(0,1)$
    \begin{equation}
        \begin{split}
            \int_0^1 (\varphi')^4 & = - 3 \int_0^1 \varphi \varphi'' (\varphi')^2 \\
            & \leq 3 \left( \int_0^1 (\varphi')^4 \right)^{\frac 12} \left( \int_0^1 (\varphi\varphi'')^2 \right)^{\frac 12}.
        \end{split}
    \end{equation}
    This proves that \eqref{ineq:main} holds for $\lambda = 1/9$, so the considered set is not empty, and one has the lower bound $\lambda^* \geq 1/9$.
    
    Second, let us prove that the supremum is finite.
    This is a straightforward consequence of the fact that the left-hand side of \eqref{ineq:main} vanishes only for $\varphi = 0$.
    In particular, using $\varphi(t) := t^3 (1-t)^3$ (which belongs to $W^{3,1}_0(0,1)$) in \eqref{ineq:main} yields the upper bound $\lambda^* \leq 22/9$.
\end{proof}

\begin{remark}
    Both the lower and upper bounds $1/9 \leq \lambda^* \leq 22/9$ are probably not sharp.
    Nevertheless, to the authors' knowledge, the optimal constant such that \eqref{ineq:main} holds is unknown and does not easily follow from known optimal Sobolev embeddings constants.
\end{remark}

\subsection{No controllability for \texorpdfstring{$\lambda \leq \lambda^*$}{lambda <= lambda*}}

We prove that, for any $\lambda \leq \lambda^*$, system \eqref{eq:syst-magnitude} is not $W^{-1,\infty}$-STLC.
In fact, we prove the following stronger result, which rules out controllability even without any size constraint on the control or small-time assumption.

\begin{lemma} \label{p:no}
    Let $\lambda \leq \lambda^*$.
    For any $T > 0$ and any $u \in L^1(0,T)$ such that the solution to~\eqref{eq:syst-magnitude} starting from $0$ satisfies $x_1(T) = x_2(T) = x_3(T) = 0$, one has $x_4(T) \geq 0$.
\end{lemma}

\begin{proof}
    Given $T > 0$ and such a control $u$, let $\varphi(t) := x_3(t T)$ for $t \in (0,1)$.
    Thus $\varphi'(t) = T x_2(t T)$, $\varphi''(t) = T^2 x_1(t T)$ and $\varphi'''(t) = T^3 u(t T)$.
    Hence if $u \in L^1(0,T)$ drives $0$ to a final state of the form $\delta e_4$, one has $\varphi \in W^{3,1}_0(0,1)$.
    And
    \begin{equation}
        x_4(T) 
        = \int_0^T x_1^2x_3^2 - \lambda x_2^4 
        = T^{-3} \int_0^1 (\varphi \varphi'')^2 - \lambda (\varphi')^4 \geq 0,
    \end{equation}
    by \eqref{ineq:main} and \cref{lem:lambda*} since $\lambda \leq \lambda^*$.
\end{proof}

\subsection{Smooth controllability for \texorpdfstring{$\lambda > \lambda^*$}{lambda > lambda*}}

We prove that, for any $\lambda > \lambda^*$, system \eqref{eq:syst-magnitude} is smoothly-STLC.

\begin{lemma} \label{p:yes}
    Let $\lambda > \lambda^*$ and $m \in \N$.
    For all $T, \rho > 0$, there exists $\delta > 0$ such that, for all $x^* \in \R^4$ with $|x^*| \leq \delta$, there exists $u \in \CC^\infty_c((0,T);\R)$ with $\|u\|_{W^{m,\infty}} \leq \rho$ such that the solution to \eqref{eq:syst-magnitude} satisfies $x(T;u) = x^*$.
\end{lemma}

\begin{proof}
    Let $\lambda > \lambda^*$ and $m \in \N$.
    By a scaling argument, as in the proof of \cref{p:no}, it is sufficient to perform the proof for $T = 1$.
    Working as in \cite[Section 8.1]{Coron2007} or \cref{s:tangent}, it is sufficient to prove that there exist controls allowing to move in the oriented directions $\pm e_i$ for $1 \leq i \leq 4$.
    \begin{itemize}
        \item 
        First, it is known from the linear theory that there exist controls $u^1$, $u^2$ and $u^3 \in \CC^\infty_c((0,1);\R)$ such that $a u^i$ drives $x(0) = 0$ to $x(1) = a e_i + O(a^4)$ as $a \to 0$ (see \cref{lem:S1-tgt}).

        \item 
        Second, let us explain how to move in the direction $+e_4$.
        Let $\chi \in \CC^\infty_c(0,1)$ with $\chi \neq 0$.
        For $\sigma > 0$, let
        \begin{equation}
            \varphi_\sigma(t) := \chi(t) \left(1+\sigma \sin \frac{t}{\sigma} \right).
        \end{equation}
        Then, as $\sigma \to 0^+$, $\int_0^1 (\varphi_\sigma')^4 = O(1)$ and
        \begin{equation}
            \int_0^1 (\varphi_\sigma \varphi_\sigma'')^2 \sim \frac{1}{2\sigma^2} \int_0^1 \chi^4.
        \end{equation}
        Thus, for any fixed $\lambda \in \R$, there exists $\sigma > 0$ small enough such that 
        \begin{equation}
            I_\sigma := \int_0^1 (\varphi_\sigma \varphi_\sigma'')^2 - \lambda \int_0^1 (\varphi_\sigma')^4 > 0.
        \end{equation}
        Given such a fixed $\sigma > 0$, setting
        \begin{equation}
            u^{4,+} := I_\sigma^{-\frac 14} \varphi_\sigma'''  
        \end{equation}
        yields a smooth control driving $0$ to $+e_4$.
        
        Using the homogeneity of \eqref{eq:syst-magnitude}, for any $a \in \R$, $x(1;a u^{4,+}) = + a^4 e_4$.
        
        \item 
        Third, let us explain how to move in the direction $-e_4$.
        Since $\lambda > \lambda^*$, by \cref{lem:lambda*}, there exists $\varphi_\lambda \in W^{3,1}_0(0,1)$ such that
        \begin{equation}
            J := \lambda \int_0^1 (\varphi_\lambda')^4 - \int_0^1 (\varphi_\lambda \varphi_\lambda'')^2 > 0.
        \end{equation}
        Since both sides are continuous functionals on $W^{3,1}_0(0,1)$, there exists $\overline{\varphi}_\lambda \in \CC^\infty_c(0,1)$, a smooth approximation of $\varphi_\lambda$ for which $\bar{J} > 0$.
        Then, setting 
        \begin{equation}
            u^{4,-} := \bar{J}^{-\frac 14} \overline{\varphi}_\lambda'''    
        \end{equation}
        yields a smooth control driving $0$ to $-e_4$.

        Using the homogeneity of \eqref{eq:syst-magnitude}, for any $a \in \R$, $x(1;a u^{4,-}) = - a^4 e_4$.
        
    \end{itemize}
    One concludes using Brouwer's fixed point theorem as in \cite[Section 8.1]{Coron2007}.
    
    Indeed, using the vocabulary of \cref{s:tangent}, we have proved that $\pm e_1, \pm e_2, \pm e_3$ are \emph{control-tangent directions of order $1$} and that $\pm e_4$ are \emph{control-tangent directions of order $4$}.
    Thus the claimed result follows from the general result \cref{p:tgt}.
\end{proof}

\newpage

\part{Rough-STLC}
\label{part:rough}

\section{Main results}
\label{s:rough-intro}

We describe the main results of \cref{part:rough}.

\medskip

In this part, one fixes \emph{a priori} some regularity index $m \in \llbracket -1 , \infty \llbracket$ quantifying the expected smallness of the control, and looks for sufficient or necessary conditions for $W^{m,\infty}$-STLC.

The case $m = 0$ (so controls small in $L^\infty$) was the most commonly studied in the literature (see e.g.~\cite[Definition 3.2]{Coron2007} or STLC$_\varepsilon$ in \cite{Kawski1987_Survey}).
It is quite natural when one has in mind to transfer the results concerning the control-affine system \eqref{syst} to systems of the form $\dot{x} = f(x,u)$.

The case $m = -1$ is heuristically close to the historical setting (as in \cite{Sussmann1987}) where one fixes a neighborhood e.g.\ $U = [-1,1]$ of $0$ of admissible values for $u(t)$, and then lets $T \to 0$.
It is equivalent to the small-state STLC for scalar-input systems (see \cite[Section 8.2]{BeauchardMarbach2018}).

More generally, studying obstructions for a fixed regularity index can be relevant for applications where the control's regularity depends on physical constraints (speed, acceleration, jerk, ...) or when the functional framework is determined by the well-posedness theory of the considered partial differential equation (see the survey \cite{Beauchard2026}).

\medskip

Our main result for rough-STLC is the necessary condition \cref{thm:Qjk_intro} involving the bad Lie bracket $Q_{j,k,k}$, which is the counterpart of \cref{thm:Qjk-smooth} in the smooth-STLC case.
It involves a threshold $M$, of which we prove the optimality, as well as a geometric assumption.
We present in \cref{s:intro-lonely} examples illustrating that, without this geometric assumption, the ``bad'' Lie bracket $Q_{j,k,k}$ can yield a controllable direction.

\medskip

These results hint that the classification problem discussed in \cref{s:classification} might be considerably more intricate in the case of rough-STLC than smooth-STLC.

\subsection{Quartic obstructions}

We resume our investigation started in \cref{s:intro-obs-smooth} of obstructions caused by the brackets $Q_{j,k,k}$ for $j \leq k \in \N^*$.
By \cref{thm:Qjk-smooth}, we know that some kind of neutralizing condition is necessary. 
We determine it here more precisely.
Recall the neutralizing subspace \eqref{eq:NjkM}.

\begin{theorem} \label{thm:Qjk_intro}
    Let $j \leq k\in\N^*$ and $m \in \llbracket -1 , \infty \llbracket$.
    Assume system \eqref{syst} is $W^{m,\infty}$-STLC and
    \begin{itemize}
        \item either $k \leq (2j+m)$,
        \item or $(2j+m)<k$ and $f_{W_j}(0) \in \vect \{f_b(0);b\in\Bs_{\intset{1,\pi} \setminus\{2\}} \setminus \{M_{k-1}\} \}$ where $\pi=1+\left\lceil \frac{2j-2}{m+1} \right\rceil$.
    \end{itemize}
    Then $f_{Q_{j,k,k}}(0) \in \mathcal{N}_{j,k}^M(f)(0)$ where $M=M(j,k,m):=3 + \left\lceil \frac{2(k+j-2)}{m+1} \right\rceil$.
    
    Moreover, this value of $M$ is optimal: for any $M' < M$, there exists a $W^{m,\infty}$-STLC system for which $f_{Q_{j,k,k}}(0) \notin \mathcal{N}_{j,k}^{M'}(f)(0)$.
\end{theorem}

In this statement, we use the conventions $\lceil 0/0 \rceil = 0$ and $\lceil n / 0 \rceil = +\infty$ when $n > 0$.

\begin{example}
    The system presented in \cref{ex:obs-smooth}, involving a competition between $Q_{1,2,2}$ and $\ad_{X_1}^5(X_0)$ satisfies $f_{Q_{1,2,2}}(0) \in \mathcal{N}^M_{1,2}(f)(0)$ if and only if $M \ge 5$.
    Thus, by \cref{thm:Qjk_intro} it is not $W^{m,\infty}$-STLC when $m \ge 1$.
    We show in \cref{p:limiting-R5} that it is $W^{m,\infty}$-STLC when $m \in \{ -1, 0 \}$.
\end{example}

\begin{remark} \label{rk:stefani}
    When $j = k = 1$, the bracket $Q_{1,1,1} = \ad_{X_1}^4(X_0)$ was already known to generate an obstruction to $W^{-1,\infty}$-STLC (see \cite[Theorem 1]{Stefani1986} and \cite[Theorem 1.10]{BeauchardMarbach2026}), unless the necessary condition $f_{Q_{1,1,1}}(0) \in S_{\intset{1,3}}(f)(0)$ is satisfied, which corresponds to $M(1,1,m) = 3$ for all $m \geq -1$.
\end{remark}

Apart from the case $(j,k) = (1,1)$, all the necessary conditions of \cref{thm:Qjk_intro} are new.
We refer to \cite[Section 1.5]{BeauchardMarbach2026} for an account of other obstructions of related nature, involving either quadratic or sextic drifts.
From \cref{thm:Qjk_intro} one can extract particular cases which are easier to state.
For example, one has the following necessary condition for the usual notion. 

\begin{corollary}
    \label{cor:Qjjj}
    A necessary condition for $L^\infty$-STLC is that $f_{Q_{j,j,j}}(0)\in \mathcal{N}_{j,j}^{4j-1}(f)(0)$.
\end{corollary}

\begin{example}
    Consider the system involving the brackets $Q_{2,2,2}$ and $\ad_{X_1}^8(X_0)$:
    \begin{equation}
        \begin{cases}
            \dot{x}_1 = u \\
            \dot{x}_2 = x_1 \\
            \dot{x}_3 = x_2^4 - x_1^8
        \end{cases}
    \end{equation}
    Using \cref{p:canonical} and a change of coordinates, one can prove that the only non-vanishing Lie brackets in $\Bs$ at $0$ are $f_{X_1}(0) = e_1$, $f_{M_1}(0) = e_2$, $f_{Q_{2,2,2}}(0) = 4! e_3$ and $f_{\ad_{X_1}^8(X_0)}(0) = - 8! e_3$.

    By \cref{cor:Qjjj}, since $f_{Q_{2,2,2}}(0) \notin \mathcal{N}^7_{2,2}(f)(0)$, it is not $L^\infty$-STLC.
    By \cref{p:limiting-R5}, it is $W^{-1,\infty}$-STLC.
    It is also $L^\infty$-STLC if $x_1^8$ is replaced with $x_1^7$. 
\end{example}

Kawski conjectured in \cite[p.\ 63]{Kawski1986} that $f_{W_j}(0) \in \vect \{ f_b(0) \mid b \in \Bs, \enskip b \neq W_j, \enskip n_1(b) \leq 2j+1 \}$ is a necessary condition for $L^\infty$-STLC, which we proved in \cite[Theorem 1.11]{BeauchardMarbach2026}.
Since $W_j = \ad_{M_{j-1}}^2(X_0)$ and $Q_{j,j,j} = \ad_{M_{j-1}}^4(X_0)$, \cref{cor:Qjjj} can be seen as a quartic version of this quadratic result.

\medskip

For $j \leq k\in\N^*$ fixed, the truncation level $M=M(j,k,m)$ depends on the regularity of the control: it is a non-increasing function of $m$, stationary on the value $M=4$ for large values of~$m$.
This dependence of $M$ with respect to $m$ is linked with interpolation inequalities.
Such a phenomenon making a link between an assumption on Lie brackets and the functional setting was already observed in \cite[Theorem 3]{BeauchardMarbach2018} and \cite[Theorems 1.11 and 1.12]{BeauchardMarbach2026}.

For $j \leq k\in\N^*$ fixed, in the low regularity case $(2j+m)<k$, the situation is more intricate and involves an additional assumption on $f_{W_j}(0)$.
We recall that $f_{W_j}(0) \in \vect \{ f_b(0) ; b \in \Bs_{\intset{1,\pi} \setminus \{2\}} \}$ is a necessary condition for $W^{m,\infty}$-STLC (see \cite[Theorem 1.11]{BeauchardMarbach2026}).
Thus, this extra assumption on $f_{W_j}(0)$ in \cref{thm:Qjk_intro} only concerns the absence of component of $f_{W_j}(0)$ along $f_{M_{k-1}}(0)$. We will see in \cref{s:intro-lonely} that this extra assumption is necessary.

\medskip

The proof of \cref{thm:Qjk_intro} in \cref{s:obs} follows the approach described in \cref{s:approach-Wm}, though each argument must be executed (much) more carefully.
In particular, using the new family of interpolation inequalities of \cref{thm:FM-GN} derived in \cite{Marbach2023} is mandatory.

\subsection{Low regularity STLC using a lonely ``bad'' bracket}
\label{s:intro-lonely}

For $1 \leq j < k$, \cref{thm:Qjk_intro} proves that the bracket $Q_{j,k,k}$  can generate an obstruction to controllability.
At low regularity, i.e. when $(2j+m)<k$, the obstructions observed above involve an extra assumption on $f_{W_j}(0)$.
The following result, proved in \cref{s:low-stlc}, shows that such an assumption cannot be removed.
In these examples, with nilpotent systems, controllability is obtained by means of a single quartic bracket which is the only non-vanishing bracket on its line.

\begin{theorem} \label{thm:cex}
    Let $1 \leq j < k$ with $k \geq 2j$.
    For every $\lambda \in \R^*$, the nilpotent system
    \begin{equation} \label{syst:cex}
        \begin{cases}
            \dot{x}_1 = u \\
            \dot{x}_2 = x_1 \\
            \dotsc \\
            \dot{x}_k = x_{k-1} + \lambda x_j^2 \\
            \dot{x}_{k+1} = 2 x_j^2 x_k \\
            \dot{x}_{k+2} = x_j^2 x_k^2 + \lambda x_j^2 x_{k+1}
        \end{cases}
    \end{equation}
    is $W^{k-2j-1,\infty}_0$-STLC despite satisfying $f_{Q_{j,k,k}}(0) \notin \mathcal{N}^\infty_{j,k}(f)(0)$.
    
    In $\Bs$, the only non-vanishing Lie brackets at $0$ are: $f_{M_{i-1}}(0) = e_i$ for $i \in \intset{1,k}$, $f_{W_j}(0) = 2 \lambda e_k$, $f_{P_{j,k}}(0) = 4 e_{k+1}$ and $f_{Q_{j,k,k}}(0) = 4 e_{k+2}$.
\end{theorem}

The above system is $W^{m,\infty}_0$-STLC for every $m\leq (k-2j-1)$, i.e.\ such that $(2j+m)<k$. 
Thus functional regularity thresholds of \cref{thm:Qjk_intro} are optimal.
The shortest systems for which this phenomenon occurs are given by the particular cases $(j,k) = (1,2)$ and $(j,k) = (1,3)$.

\begin{example}
    \label{ex:Q122}
    For every $\lambda \in \R^*$, the nilpotent system
    \begin{equation} \label{syst:cex-Q122}
        \begin{cases}
            \dot{x}_1=u \\
            \dot{x}_2=x_1+ \lambda x_1^2 \\
            \dot{x}_3= 2 x_1^2 x_2 \\
            \dot{x}_4=x_1^2 x_2^2 + \lambda x_1^2 x_3
        \end{cases}
    \end{equation}
    is $W^{-1,\infty}$-STLC despite satisfying $f_{Q_{1,2,2}}(0) \notin \mathcal{N}_{1,2}^{\infty}(f)(0)$.
\end{example}

\begin{example}
    \label{ex:Q133}
    For every $\lambda \in \R^*$, the nilpotent system
    \begin{equation} \label{syst:cex-Q133}
        \begin{cases}
            \dot{x}_1=u \\
            \dot{x}_2=x_1\\
            \dot{x}_3=x_2+ \lambda x_1^2\\
            \dot{x}_4=2 x_1^2 x_3 \\
            \dot{x}_5=x_1^2 x_3^2 + \lambda x_1^2 x_4
        \end{cases}
    \end{equation}
    is $L^{\infty}$-STLC despite satisfying $f_{Q_{1,3,3}}(0) \notin \mathcal{N}_{1,3}^{\infty}(f)(0)$.
\end{example}

For these systems, despite the fact that, within the chosen coordinates, the last line of \eqref{syst:cex} involves two monomials, we stress that there is no ``competition'' between two Lie brackets on these lines.
System \eqref{syst:cex} involves only $f_{Q_{j,k,k}}(0)$ along $e_{k+2}$.

Moreover, by \cref{thm:Qjk_intro}, these systems are not smoothly-STLC.
One can also check that the subsystem $(x_1, \dotsc x_{k+1})$ of \eqref{syst:cex} is smoothly-STLC; for example because\footnote{The historical proofs \cite{Hermes1982,Sussmann1983} only yield $L^\infty$-STLC, but one can recover smooth-STLC by noting that the Hermes condition is invariant under addition of integrators (as in \cite[p.\ 143]{Coron2007}).} it satisfies the Hermes sufficient condition of \cite{Hermes1982,Sussmann1983}.

The fact that, in the presence of a competition between two Lie brackets, one can have a regularity threshold for the STLC was already known (see e.g.\ \cite[Section 2.4.1]{BeauchardMarbach2018}).
The novelty of the above examples is the presence of a regularity threshold for the STLC despite the fact that, on the problematic line, only a single Lie bracket is involved.
Hence, we believe that the phenomenon at stake is entirely different.
See also \cref{s:low-paradox} for further comments.

Systems \eqref{syst:cex-Q122} and \eqref{syst:cex-Q133} emphasize that the validity of a statement such as: ``\emph{a bad bracket, when not compensated by appropriate other brackets, induces a drift as $(T,\|u\|) \to 0$}'' can actually depend both on the choice of the norm $\|\cdot\|$ and on the presence in the system of other Lie brackets.

\begin{remark}
    The controllability of \cref{ex:Q122,ex:Q133} only holds when $\lambda \neq 0$, i.e.\ when $f_{W_1}(0) \neq 0$.
    This is a very strange behavior.
    $W_1$ is the best known obstruction to controllability.
    It was first highlighted by Sussmann in \cite[Proposition 6.3]{Sussmann1983}.
    It is a sufficiently strong obstruction that it has also been observed for PDEs of various natures (see  \cite{BeauchardMarbach2020,BeauchardMarbachPerrin2025,Coron2006,CoronKoenigNguyen2024,NiuXiang2025}). 
    Here, having this term present inside the subspace $S_1(f)(0)$ restores controllability.
\end{remark}

\newpage
\section{Obstructions to rough-STLC}
\label{s:obs}

Let $j \leq k \in \N^*$ and $m \in \llbracket -1, \infty \llbracket$.
We prove \cref{thm:Qjk_intro} on the obstruction to $W^{m,\infty}$ small-time local controllability caused by the bad bracket $\q := Q_{j,k,k}$.

\bigskip

\emph{We strongly encourage the reader to start with the case of obstructions to smooth small-time local controllability presented in \cref{s:obs-smooth}.
The reader can also skip to \cref{s:limiting}, which gives examples illustrating the optimality of our results for the $W^{m,\infty}$-STLC and is readable independently from the quite technical proofs of \cref{s:eta-Qjk,s:closed-loop,s:interpolation,s:obs-m,s:obs-1}.}

\subsection{Organization of the proof}

The general strategy of the proof is the one presented in \cref{s:approach-Wm}.
In particular, we will follow the three arguments described in \cref{s:obs-organization}.
\begin{itemize}
    \item \textbf{Algebraic argument.} 
    In \cref{s:eta-Qjk}, we prove \cref{p:eta_Qjk-new}, which provides an estimate of the form \eqref{eq:eta-xi-approx} for $|\eta_\q - \xi_\q|$, based on algebraic relations within $\Bs$, and involving the pointwise boundary term $U_k(t)$ of \eqref{eq:borduk}.
    
    \item \textbf{Geometric argument.} 
    In \cref{s:closed-loop}, we prove \cref{p:borduk}, which provides a closed-loop estimate of the form \eqref{eq:borduk-approx} for $U_k(t)$, based on vectorial relations between the iterated Lie brackets of the vector fields.
    
    \item \textbf{Analysis argument.} 
    In \cref{s:interpolation}, we prove multiple interpolation inequalities, including \cref{p:interp-u1}, which establishes a bound of the form \eqref{eq:heuristic-interpol} for $\|u_1\|_{L^{M+1}}^{M+1}$, for the choice of $M = M(j,k,m)$ stated in \cref{thm:Qjk_intro}.
\end{itemize}
These arguments are combined in \cref{s:obs-m} to prove \cref{thm:Qjk_intro} when $m \geq 0$.
The case $m=-1$, which involves another difficulty, is then covered in \cref{s:obs-1}.
Eventually, \cref{s:limiting} constructs examples illustrating that the choice of cutting threshold $M(j,k,m)$ is optimal.

\subsubsection{Estimates on coordinates of the second kind}

In \cref{s:obs-smooth}, the crude universal estimate of \cref{lem:xi-naive-inf} was sufficient.
In the next subsections however, we will need the following more precise estimates, proved in \cite[Appendix A.4]{BeauchardMarbach2026}. 

\begin{proposition}
    The following bounds hold.
    \begin{enumerate}
        \item Let $p \in [1,\infty]$ and $j_0 \in \N^*$.
        There exists $c > 0$ such that, for every $j \geq j_0$, $t > 0$ and $u \in \lone$, $\ell := |M_j| \geq j_0+1$ and
        \begin{equation}
            \label{bound-xiMj/J0}
            |\xi_{M_j}(t,u)| 
            \leq \frac{(ct)^{\ell}}{\ell!} 
            t^{-(j_0+1)} t^{1-\frac 1 p} \|u_{j_0}\|_{L^p}.
        \end{equation}
        
        \item Let $p \in [1,\infty]$ and $j_0 \in \N^*$.
        There exists $c > 0$ such that, for every $j \geq j_0$, $\nu \geq 0$, $t > 0$ and $u \in \lone$, $\ell:= |W_{j,\nu}| \geq 2j_0+1$ and
        \begin{equation}
            \label{bound-xiWjnu/j0}
            |\xi_{W_{j,\nu}}(t,u)| 
            \leq \frac{(ct)^{\ell}}{\ell!}
            t^{-(2j_0+1)} t^{1-\frac 1 p} \|u_{j_0}\|_{L^{2p}}^2.
        \end{equation}
    \end{enumerate}
\end{proposition}

For brackets in $\Bs_3$, we only need the following direct consequence of the explicit formula \eqref{xi_S3}.

\begin{lemma}
    \label{lem:xi-S3}
    Let $1 \leq j \leq k$ and $\nu \geq 0$.
    For all $t > 0$ and $u \in \lone$,
    \begin{equation}
        | \xi_{P_{j,k,\nu}}(t,u) | \leq t^\nu \| u_j^2 u_k \|_{L^1}.
    \end{equation}
\end{lemma}

\subsection{Algebraic computation of the coordinate of the pseudo-first kind}
\label{s:eta-Qjk}

We prove \cref{p:eta_Qjk-new}, which quantifies the heuristic $\eta_\q \approx \xi_\q$.

\subsubsection{Algebraic decompositions in $\Bs$}

As in \cref{s:algebra-smooth}, using \cref{s:hall-sets}, we prove purely algebraic structural lemmas on some decompositions of Lie brackets on $\Bs$.
In order to get an optimal estimate on $|\eta_\q - \xi_\q|$, we need to describe with much more precision which Lie brackets of elements of $\Bs$ involve $\q$ in their support.
In this section, we drop the index $\Bs$ in $\supp_{\Bs}$ and $\langle \cdot, \cdot \rangle_{\Bs}$.

\begin{lemma} \label{p:mu=X0}
    For $b \in \Br(X)$, $\supp \eval(b0) \subset \{ c \in \Bs ; \mu(c)=X_0 \}$.
    In particular, $\langle b0, \q \rangle = 0$.
\end{lemma}

\begin{proof}
    Expand $\eval(b)$ on $\eval(\Bs)$ as $\eval(b) = \sum_{c \in \Bs} \alpha_c \eval(c)$.
    In $\mathcal{L}(X)$, one has $\eval(b0) = \sum_{c \in \Bs} \alpha_c \eval(c0)$.
    Moreover, for each $c \in \Bs$, either $c = X_0$ and $\eval(c0) = [X_0,X_0] = 0$ or $(c, X_0) \in \Bs$ because $X_0$ is maximal in $\Bs$.
    Thus the $\Bs$ support of $\eval(b0)$ only involves elements whose right factor is $X_0$.
\end{proof}

\begin{lemma} \label{p:1+2}
    Let $j_1,k_1 \in \N^*$.
    \begin{enumerate}
        \item \label{p:1+2/1}
        If $j_1 \leq k_1$, then $(M_{k_1-1},W_{j_1}) = P_{j_1,k_1} \in \Bs$.
        
        \item \label{p:1+2/2} 
        If $j_1 > k_1$, then $\supp [M_{k_1-1}, W_{j_1}] \subset \{ P_{j',k',\nu'}; j' < j_1 \}$.

        \item \label{p:1+2/3}
        If $\nu_1 \in \N$, then $\supp [M_{k_1-1}, W_{j_1,\nu_1}] \subset \{ P_{j',k',\nu'}; j' \leq j_1 \}$. 
    \end{enumerate}
\end{lemma}

\begin{proof}
    For \cref{p:1+2/2}, let $b \in \supp [M_{k_1-1}, W_{j_1}]$.
    Since $\Bs_3$ spans $S_3(X)$, $b = P_{j',k',\nu'}$ with $j' \leq k' \in \N^*$ and $\nu' \in \N$.
    On the one hand $n_0(b) = 2j'+k'-2+\nu' \geq 3j'-2$.
    On the other hand $[M_{k_1-1}, W_{j_1}] \in S_{3,k_1+2j_1-2}(X)$, where $k_1+2j_1-2 \leq 3j_1-3$.
    Thus $3j'-2 \leq 3j_1-3$ so $j' < j_1$.

    Thanks to \eqref{eq:jacobi.rtl}, \cref{p:1+2/3} follows from \cref{p:1+2/1,p:1+2/2}.
\end{proof}

\begin{lemma} \label{p:1+3}
    Let $j_0 \leq k_0 \in \N^*$, $l_0 \in \N^*$ and $\nu_0 \in \N$.
    If $\langle [M_{l_0-1}, P_{j_0,k_0,\nu_0}], \q \rangle \neq 0$, then $j \leq j_0$.
\end{lemma}

\begin{proof}
    By \eqref{eq:jacobi.rtl} and \cref{p:mu=X0}, we can assume that $\nu_0 = 0$.
    We can also assume that $l_0 < k_0$ (otherwise $(M_{l_0-1},P_{j_0,k_0}) = Q_{j_0,k_0,l_0} \in \Bs$, $j = j_0$, $k = k_0 = l_0$).
    We proceed by induction on $j_0$.

    \emph{Initialization for $j_0 = 1$.}
    Using Jacobi's identity, we write 
    \begin{equation}
        [M_{l_0-1}, P_{1,k_0}] = [M_{k_0-1},[M_{l_0-1},W_1]] + [[M_{l_0-1},M_{k_0-1}], W_1]
        \in Q_{1,l_0,k_0} + [S_2(X), S_2(X)].
    \end{equation}
    By \cref{p:2+2}, $j = 1$ and $k = l_0 = k_0$.

    \emph{Induction step for $j_0 > 1$.}
    Using Jacobi's identity, we write 
    \begin{equation}
        [M_{l_0-1}, P_{j_0,k_0}] = [M_{k_0-1},[M_{l_0-1},W_{j_0}]] + [[M_{l_0-1},M_{k_0-1}], W_{j_0}].
    \end{equation}
    The second term can be discarded by \cref{p:2+2}.
    When $l_0 \geq j_0$, the first term is $Q_{j_0,l_0,k_0} \in \Bs$ so $j = j_0$, $k = l_0 = k_0$.
    When $l_0 < j_0$, using \cref{p:1+2/2} of \cref{p:1+2}, the conclusion $j \leq j_0$ follows from the induction hypothesis.
\end{proof}

\begin{lemma} \label{p:1+1+2}
    Let $j_0, k_1, k_2 \in \N^*$ and $\nu_0 \in \N$.
    Assume that $\langle [M_{k_2-1}, [M_{k_1-1}, W_{j_0,\nu_0}]], \q \rangle \neq 0$.
    Then
    \begin{enumerate}
        \item 
        \label{p:1+1+2/1}
        One has $\max \{ k_1, k_2 \} \leq k$ and $j \leq j_0$.
        \item 
        \label{p:1+1+2/2}
        Moreover, when $\nu_0 = 0$,
        \begin{itemize}
            \item either $j_0 \leq \min \{ k_1, k_2 \}$ and then $j = j_0$ and $k = k_1 = k_2$,
            \item or $j_0 > \min \{ k_1, k_2 \}$ and then $j < j_0$.
        \end{itemize}
    \end{enumerate}
\end{lemma}

\begin{proof}
    First, using Jacobi's identity and \cref{p:2+2}, one can assume that $k_1 \leq k_2$.

    \emph{Proof of \cref{p:1+1+2/1}.}
    By \cref{p:1+2/3} of \cref{p:1+2}, $[M_{k_1-1}, W_{j_0,\nu_0}] = \sum \alpha_{j',k',\nu'} P_{j',k',\nu'}$ with $j' \leq j_0$. 
    Since $\langle [M_{k_2-1}, [M_{k_1-1}, W_{j_0,\nu_0}]], \q \rangle \neq 0$, there exists $(j',k',\nu')$ such that $\langle [M_{k_2-1},P_{j',k',\nu'}], \q \rangle \neq 0$.
    By \cref{p:1+3}, $j \leq j' \leq j_0$.
    By \cref{p:hall-left}, $\lambda(\q) = M_{k-1} \geq M_{k_2-1}$ so $k_1 \leq k_2 \leq k$.

    \emph{Proof of \cref{p:1+1+2/2}.}
    In the first case, $j_0 \leq k_1 \leq k_2$, thus $(M_{k_2-1}, (M_{k_1-1}, W_{j_0})) = Q_{j_0,k_1,k_2} \in \Bs$ so $j = j_0$ and $k = k_1 = k_2$.
    In the second case, since $k_1 < j_0$, by \cref{p:1+2/2} of \cref{p:1+2}, $[M_{k_1-1}, W_{j_0}] = \sum \alpha_{j',k',\nu'} P_{j',k',\nu'}$ with $j' < j_0$. 
    Since $\langle [M_{k_2-1}, [M_{k_1-1}, W_{j_0}]], \q \rangle \neq 0$, there exists $(j',k',\nu')$ such that $\langle [M_{k_2-1},P_{j',k',\nu'}], \q \rangle \neq 0$.
    By \cref{p:1+3}, $j \leq j' < j_0$.
\end{proof}

\begin{lemma} \label{p:1+3/no-delta}
    Let $j' \leq k' \in \N^*$, $l \in \N^*$, $\nu \in \N$, such that $\langle [M_{l-1}, P_{j',k',\nu}], \q \rangle \neq 0$.
    Then $l + \nu \leq k$ and, either $(j',k',l+\nu) = (j,k,k)$ or $j < j' \leq k' \leq k$ and $3j' + k' \geq 2j+2k+1$.
\end{lemma}

\begin{proof}
    By \eqref{eq:jacobi.rtl} and \cref{p:mu=X0}, $\langle [M_{l-1}, P_{j',k',\nu}], \q \rangle = (-1)^{\nu} \langle [M_{l+\nu-1}, P_{j',k'}], \q \rangle$.
    Assume that this quantity is non-zero.
    By \cref{p:1+1+2/1} of \cref{p:1+1+2}, $l+\nu \leq k$, $k' \leq k$ and $j \leq j'$.
    
    If $l+\nu \geq k'$, then $(M_{l+\nu-1}, P_{j',k'}) = Q_{j',k',l+\nu} \in \Bs$ so $(j',k',l+\nu) = (j,k,k)$.

    If $l+\nu < k'$, by \cref{p:1+1+2/2} of \cref{p:1+1+2}, $j' > j$ and $j' > l+\nu$.
    Writing $2j + 2k - 3 = n_0(\q) = n_0([M_{l-1},P_{j',k',\nu}])$, we obtain $l + \nu = 2j + 2k - 2j' - k'$.
    Hence $3j'+k' \geq 2j+2k+1$.
\end{proof}

\subsubsection{Estimate of the difference $\eta_\q - \xi_\q$}

We bound $\eta_\q - \xi_\q$ using the boundary term $U_k(t) := (u_1,\dotsc,u_k)(t)$ introduced in \eqref{eq:borduk}.

\begin{proposition} \label{p:eta_Qjk-new}
    There exists $C > 0$ such that, for $t \in (0,1]$ and $u \in \lone$,
    \begin{equation} \label{eq:eta_Qjk-new}
        \begin{split}
        |\eta_\q(t,u) - \xi_\q(t,u)|
        \leq C |U_k(t)| \Big( 
            |U_k(t)|^3 
            & + t^3 \| u_k \|_{L^1}^3 
            + |U_k(t)| \| u_j \|_{L^2}^2 
            \\ & + \| u_j^2 u_k \|_{L^1} 
            + \sum_{j', k'} \| u_{j'}^2 u_{k'} \|_{L^1} \Big),
        \end{split}
    \end{equation}
    where the sum ranges over indices such that $j < j' \leq k' \leq k$ and $2j + 2k + 1 \le 3j' + k'$.
\end{proposition}

\begin{proof}
    We apply \cref{p:etab-xib-XI}.
    Let $q \geq 2$, $b_1 \geq \dotsb \geq b_q \in \Bs \setminus \{ X_0 \}$ such that $Q_{j,k,k} \in \supp \mathcal{F}(b_1, \dotsc, b_q)$ (see \cref{def:calF}).
    Since $n_1(Q_{j,k,k}) = 4$, $q \leq 4$.
    \begin{itemize}
        \item \emph{Case $q = 4$.}
        Then each $b_i \in \Bs_1$ so there exists $l_i \ge 1$ such that $b_i = M_{l_i - 1}$ and, by \eqref{xi_S1}, $\xi_{b_i}(t,u) = u_{l_i}(t)$.
        Since $n_0(b_1) + \dotsb + n_0(b_4) = n_0(\q) \leq 4 k - 3$, one has $l_4 \leq k$.
        
        If $l_i \leq k$, $|u_{l_i}(t)| \leq |U_k(t)|$.
        If $l_i > k$, $|u_{l_i}(t)| \leq t^{l_i-k} \| u_k \|_{L^1} \leq t \| u_k \|_{L^1}$.
        
        Hence, for $t \leq 1$, using Young's inequality,
        \begin{equation*}
            |\xi_{b_1} \xi_{b_2} \xi_{b_3} \xi_{b_4} |(t,u) \leq C |U_k(t)| \left( |U_k(t)|^3 + t^3 \|u_k\|_{L^1}^3 \right).
        \end{equation*}

        \item \emph{Case $q = 3$.}
        Then $b_1 \in \Bs_2$ and $b_2, b_3 \in \Bs_1$.
        By \cref{p:2+2} and \cref{p:1+1+2/1} of \cref{p:1+1+2}, one has $b_1 = W_{j',\nu'}$ and $b_2 = M_{l_2'-1}$ and $b_3 = M_{l_3'-1}$ with $j' \geq j$ and $l_2', l_3' \leq k$.
        Thus $|\xi_{b_2}(t,u)| \leq |U_k(t)|$ and $|\xi_{b_3}(t,u)| \leq |U_k(t)|$.
        Moreover, by \eqref{bound-xiWjnu/j0}, $|\xi_{b_1}(t,u)| \leq C t^{2(j'-j)+\nu'}\|u_j\|_{L^2}^2$.
        Hence,
        \begin{equation*}
            |\xi_{b_1} \xi_{b_2} \xi_{b_3} |(t,u) 
            \leq C \left( |U_k(t)|^2 \|u_j\|_{L^2}^2 \right).
        \end{equation*}

        \item \emph{Case $q = 2$.}
        By \cref{p:2+2}, one cannot have, $b_1, b_2 \in \Bs_2$. 
        Hence $b_1 = P_{j',k',\nu} \in \Bs_3$ and $b_2 = M_{l-1} \in \Bs_1$.
        First, by \cref{p:1+3/no-delta}, $l \leq k$ so that, by \eqref{xi_S1}, $|\xi_{b_2}(t,u)| \leq |U_k(t)|$. 
        Second, by \cref{p:1+3/no-delta}, either ($j' = j$ and $k' = k$) or ($j < j' \leq k' \leq k$ and $3j'+k' \geq 2j+2k+1$).
        So, using \cref{lem:xi-S3}, for $t \leq 1$,
        \begin{equation}
            |\xi_{b_1} \xi_{b_2} |(t,u) 
            \leq C |U_k(t)| \Big( \| u_j^2 u_k \|_{L^1} + \sum_{j',k'} \| u_{j'}^2 u_{k'} \|_{L^1} \Big).
        \end{equation}
    \end{itemize}
    Thus estimate \eqref{eq:eta_Qjk-new} follows from \cref{p:etab-xib-XI}.
\end{proof}

\subsection{Closed-loop estimates resulting from vectorial relations}
\label{s:closed-loop}

We prove \cref{p:borduk}, which provides a closed-loop estimate for $U_k(t)$.

\subsubsection{Vectorial relations}

In this paragraph, we introduce vectorial relations and relate their validity with appropriate assumptions on $f_\q(0)$ and $f_{W_j}(0)$.
In the whole section, $\pi, N, M \in \intset{1,\infty}$.
\begin{enumerate}[label=(H\arabic*)]
    \item \label{H1}
    $f_{W_j}(0) \in \vect \{ f_b(0) ; b \in \Bs_{\intset{1,\pi} \setminus \{2\}} \setminus \{M_0,\dots,M_{k-1}\}\}$.
    
    \item \label{H2}
    $\forall i \in \intset{0,k-1}, \quad f_{M_i}(0) \notin \vect \{ f_b(0) ; b \in \Bs_{\intset{1,N}} \setminus\{ M_i, W_{j',\nu} ; j' \geq j, \nu \in \N \} \}$.
    
    \item \label{H3}
    $\forall i \in \intset{0,k-1}, \quad f_{M_i}(0) \notin \vect \{ f_b(0) ; b \in \Bs_{\intset{1,N}} \setminus\{ M_i \} \}$.
\end{enumerate}

In the proofs, we use the following convenient notations: for $\ell \in \N$, $M_{\geq \ell}$ (resp.\ $M_{>\ell}$) denotes an element of $\vect \{ M_{\ell'}; \ell' \geq \ell \}$ (resp.\ $> \ell$).
The proofs below rely implicitly on \cref{lem:Lie-subalg}.

\begin{lemma} \label{p:H1}
    Assume that $f_\q(0) \notin \mathcal{N}_{j,k}^{M}(f)(0)$.
    Then \hyp{H1} holds if
    \begin{itemize}
        \item $j<k$, $M \geq 3 \max\{\pi,2\}$, and $f_{W_j}(0) \in \vect \{ f_b(0) ; b \in \Bs_{\intset{1,\pi} \setminus\{2\}} \setminus\{ M_{k-1}\} \}$.
    \end{itemize}
\end{lemma}

\begin{proof}
    By assumption, there exist $\beta_0,\dots,\beta_{k-2}\in\R$, and $P \in S_{\intset{3,\pi}}(X)$ such that $f_{B_1}(0)=0$ for $B_1=W_j + \beta_0 M_0 + \dots + \beta_{k-2} M_{k-2} + M_{\geq k} + P$.
    Let us prove that $\beta_0=\dots=\beta_{k-2}=0$.
    Working by contradiction, we consider $l_0 = \min\{ l \in \llbracket 0 , k-2 \rrbracket ; \beta_{l_0} \neq 0 \}$. Then 
    \begin{equation}
        B_1=\beta_{l_0} M_{l_0} + M_{>l_0} + W_j + P.
    \end{equation} 
    Then $f_{B_2}(0)=0$ where 
    \begin{equation}
        B_2=B_1 0^{2(k-1-l_0)}=W_j 0^{2(k-1-l_0)} + M_{\geq 2k-l_0-2}+\widetilde{P}  
    \end{equation}
    where $\widetilde{P}=P0^{2(k-1-l_0)} \in S_{\intset{3,\pi}}(X)$. Thus $f_{B_3}(0)=0$ where $B_3=\ad_{B_1}^2(B_2)$ i.e.
    \begin{equation}
        \begin{split}
        B_3 & = \ad_{\beta_{l_0} M_{l_0} + M_{>l_0} + W_j + P}^2 \Big( 
        W_j 0^{2(k-1-l_0)} + M_{\geq 2k-l_0-2} + \widetilde{P} \Big)
        \\ & = \beta_{l_0}^2 [M_{l_0},[M_{l_0},W_j 0^{2(k-1-l_0)} ]] + B_4 
        \end{split}
    \end{equation}
    where $B_4 \in \vect \mathcal{N}_{j,k}^{M}$.
    Indeed, $B_4$ is a sum of brackets that
    \begin{itemize}
        \item either belong to $S_3(X)$, $S_{\intset{5,\max\{3\pi,6\}}}(X)$ or $S_{4,\N^*\setminus\{ n_0(\q)\}}(X)$, thus belong to $\vect \mathcal{N}_{j,k}^{M}$,
    
        \item or are of the form $[W_j,[M_{i},M_{i'}]]$ and belong to  $\vect \mathcal{N}_{j,k}^{4}$  by \cref{p:2+2},
    
        \item or are of the form $[M_{i},[W_j,M_{i'}]]$ where $i' \geq 2k-l_0-2 \geq k$ and by \cref{p:1+1+2/1} of \cref{p:1+1+2}, $\q \notin \supp_{\Bs} [M_i,[W_j,M_{i'}]]$ i.e.\ such a bracket belongs to $\vect \mathcal{N}_{j,k}^{4}$.
    \end{itemize}
    Moreover, by \eqref{eq:jacobi.rtl} applied twice and \cref{p:mu=X0}, letting $\Lambda := k-1-l_0$,
    \begin{equation}
        \begin{split}
            [M_{l_0},[M_{l_0},W_j 0^{2\Lambda}]]
            & = \sum_{\mu+\mu'\leq 2\Lambda} (-1)^{\mu+\mu'} c_{\mu,\mu'} [M_{l_0+\mu'},[M_{l_0+\mu},W_j]]0^{2\Lambda-\mu-\mu'} 
            \\ & = \sum_{\mu+\mu' = 2\Lambda}  c_{\mu,\mu'} [M_{l_0+\mu'},[M_{l_0+\mu},W_j]] + B_5
        \end{split}
    \end{equation}
    where $c_{\mu,\mu'}=\binom{2\Lambda}{\mu} \binom{2\Lambda-\mu}{\mu'}$ and $B_5 \in \vect \{b0;b\in\Bs_4\} \subset \vect \mathcal{N}^4_{j,k}$.
    Thus, by \cref{p:1+1+2/2} of \cref{p:1+1+2},
    \begin{equation}
        [M_{l_0},[M_{l_0},W_j 0^{2\Lambda} ]]=c_{\Lambda,\Lambda} \q +B_6
    \end{equation} 
    where $B_6 \in \vect \mathcal{N}_{j,k}^4$. 
    Finally
    $B_3=\beta_{l_0}^2 c_{\Lambda,\Lambda} \q + B_7$ where $B_7=\beta_{l_0}^2 B_6+B_4\in \vect \mathcal{N}_{j,k}^M$ and the equality $f_{B_3}(0)=0$ gives $f_\q(0)=-\frac{1}{\beta_{l_0}^2 c_{\Lambda,\Lambda}} f_{B_7}(0) \in \mathcal{N}_{j,k}^{M}(f)(0)$, contradiction.
\end{proof}

\begin{proposition} \label{p:H0+H2}
    Assume that $f_\q(0)\notin \mathcal{N}_{j,k}^{M}(f)(0)$.
    Then \hyp{H2} holds if
    \begin{itemize}
        \item $j<k$, $M \geq 2N + \max \{\pi,2\}$, and $f_{W_j}(0) \in \vect \{ f_b(0) ; b \in \Bs_{\intset{1,\pi} \setminus\{2\}} \}$. 
    \end{itemize}
\end{proposition}

\begin{proof}
    \emph{General case $N \geq 2$.}
    Assume there exists $(\alpha_0,\dots,\alpha_{k-1})\in\R^k\setminus\{0\}$ and $B \in \vect \{\Bs_{\intset{1,N}} \setminus \{M_0,\dots,M_{k-1},W_{j',\nu};j' \geq j,\nu\in\N \} \}$ such that $f_{B_1}(0)=0$ for $B_1=\alpha_0 M_0+\dots+\alpha_{k-1}M_{k-1}+B$. 
    One may assume $\alpha_0=\dots=\alpha_{k-2}=0$ and $\alpha_{k-1}=1$ i.e.\ $B_1=M_{k-1}+M_{\geq k}+W_{< j}+P_1$ where $W_{< j} \in \vect\{ W_{j',\nu}; j'< j,\nu\in\N\}$ and $P_1 \in S_{\intset{3,N}}(X)$. 
    By assumption, there exists $\overline{M} \in S_1(X)$ and $P_2 \in S_{\intset{3,\pi}}(X)$ such that $f_{B_2}(0)=0$ for $B_2=W_{j}+\overline{M}+P_2$.
    Thus $f_{B_3}(0)=0$ for $B_3=\ad_{B_1}^2(B_2)$ i.e.
    \begin{equation}
        B_3=
        [M_{k-1}+M_{\geq k}+W_{< j}+P_1,
        [M_{k-1}+M_{\geq k}+W_{< j}+P_1,
        W_{j}+\overline{M}+P_2]]
         = \q + B_4
    \end{equation}
    where $B_4 \in \vect \mathcal{N}_{j,k}^M$ thanks to the assumption $M \geq 2 N + \max \{ \pi, 2\}$.
    Indeed, in the expansion of $B_4$, the terms with $n_1=4$ are of the form
    \begin{itemize}
        \item $[W_{j',\nu},[M_{k_1},M_{k_2}]]$, thus belong to $\vect \mathcal{N}_{j,k}^{4}$ by \cref{p:2+2},
         \item $[M_{k_1},[W_{j',\nu},M_{k_2}]]$, where $j'<j$, and by \cref{p:1+1+2/1} of \cref{p:1+1+2}, they belong to $\vect \mathcal{N}_{j,k}^4$,
        \item $[M_{k_1-1}, [M_{k_2-1}, W_j]]$, with $k_1+k_2 > 2k$, thus with $n_0 > n_0(\q)$.
    \end{itemize}
    As above, this contradicts $f_\q(0) \notin \mathcal{N}^M_{j,k}(f)(0)$.
    
    \medskip \noindent \emph{Particular case $N = 1$.}
    When $N = 1$, we follow the same lines.
    In this case, one can assume that $B_1 = M_{k-1} + M_{\geq k}$ ($W_{<j} = 0$ and $P_1 = 0$).
    Hence
    \begin{equation}
        B_3 = [M_{k-1}+M_{\geq k},
        [M_{k-1}+M_{\geq k},
        W_{j}+\overline{M}+P_2]]
         = \q + B_4,
    \end{equation}
    where $B_4 \in \vect \mathcal{N}^M_{j,k}$ thanks to the same arguments and the assumption $M \geq 2 + \max \{ \pi, 2 \}$.
\end{proof}

\begin{proposition} \label{p:H0+H3}
    Assume that $f_\q(0)\notin \mathcal{N}_{j,k}^{M}(f)(0)$.
    Then \hyp{H3} holds if either
    \begin{itemize}
        \item $j=k$ and $M \geq 4N$, or
        \item $j<k$, $M \geq 2 \max\{N,\pi,2\} + \max\{\pi,2\}$, and $f_{W_j}(0) \in \vect \{ f_b(0) ; b \in \Bs_{\intset{1,\pi} \setminus\{2\}} \setminus\{M_{k-1}\} \}$.
    \end{itemize}
\end{proposition}

\begin{proof}
    Assume there exists $(\alpha_0,\dots,\alpha_{k-1})\in\R^k\setminus\{0\}$ and $B \in \vect \{\Bs_{\intset{1,N}} \setminus\{M_0,\dots,M_{k-1} \} \}$ such that $f_{B_1}(0)=0$ for $B_1=\alpha_0 M_0+\dots+\alpha_{k-1}M_{k-1}+B$. One may assume $\alpha_0=\dots=\alpha_{k-2}=0$ and $\alpha_{k-1}=1$ i.e. $B_1=M_{k-1}+M_{\geq k}+W$ where $W \in S_{\llbracket 2 , N\rrbracket}(X)$.
    
    \medskip
    
    \noindent \emph{Case: $j=k$.} We have $f_{B_3}(0)=0$ for $B_3=\ad_{B_1}^4(X_0)$ i.e.
    \begin{equation}
        B_3 = \ad_{M_{k-1}+M_{\geq k}+W}^4(X_0)
        \in \q
        + S_{4,\llbracket n_0(\q)+1,\infty\llbracket} (X)
        + S_{\llbracket 5 , 4N \rrbracket} (X),
    \end{equation}
    which contradicts $f_\q(0) \notin \mathcal{N}^M_{j,k}(f)(0)$ since $M \geq 4N$.
    
    \medskip
    
    \noindent \emph{Case: $j<k$.} By \cref{p:H1}, \hyp{H1} holds and
    there exists $P \in S_{\intset{3,\pi}}(X)$ such that $f_B(0)=0$ for $B=W_{j}+M_{\geq k}+P$. 
    Thus $f_{B_2}(0)=0$ for
    \begin{equation}
        B_2=\ad_{B_1}^2(B)=[M_{k-1}+M_{\geq k}+W,[M_{k-1}+M_{\geq k}+W, W_{j}+M_{\geq k}+P ]] = \q + B_3 
    \end{equation}
    where $B_3 \in \vect \mathcal{N}_{j,k}^{M}$, thanks to the assumption $M \geq 2 \max \{ N, \pi, 2 \} + \max \{ \pi, 2\}$.
    Indeed, in the expansion of $B_2$, the terms with $n_1=4$ are of the form
    \begin{itemize}
        \item $[W_{j',\nu},[M_{k_2},M_{k_1}]]$, thus belong to $\vect \mathcal{N}_{j,k}^4$ by \cref{p:2+2},    
        \item $[M_{k_1-1},[W_{j',\nu},M_{k_2-1}]]$, 
        where $k_2 > k$, so by \cref{p:1+1+2/1} of \cref{p:1+1+2} they belong to $\vect \mathcal{N}_{j,k}^4$,
        \item $[M_{k_1-1}, [M_{k_2-1}, W_j]]$, with $k_1+k_2 > 2k$, thus with $n_0 > n_0(\q)$.
    \end{itemize}
    As above, this contradicts $f_\q(0) \notin \mathcal{N}^M_{j,k}(f)(0)$.
\end{proof}

\subsubsection{Closed-loop estimates}

We prove a higher-order ``closed-loop estimate'' on $U_k(t)$ from the representation formula \cref{thm:Magnus} and the vectorial relations \hyp{H2} and \hyp{H3}.
The proof uses the following strategy of \cite[Section 4.4]{BeauchardMarbach2026}.
Heuristically, it states that if one has bounds on all the coordinates of the second kind that can be involved in a considered component of the state, then one can estimate this component. 

\begin{proposition} \label{p:PZM-xibb-OXi}
    Let $M, L \in \N^*$.
    Let $\bb \in \Bs_{\intset{1,M}}$ and $\mathcal{N} \subset \Bs_{\intset{1,M}}$ with $\bb \notin \mathcal{N}$.
    Assume that there exist $c > 0$ and $\Xi : \R_+^* \times \lloc \to \R_+$ with $\Xi(t,u) = O(1)$ such that, for every $t > 0$ and $u \in \lone$,
    \begin{itemize}
        \item for all $b \in \Bs_{\intset{1,M}}$ such that $b \notin \mathcal{N} \cup \{ \bb \}$, there exists $\sigma \leq L$ such that $|b| \geq \sigma$ and
        \begin{equation} \label{eq:xib-XI-L}
            |\xi_b(t,u)| \leq \frac{(ct)^{|b|}}{|b|!} t^{-\sigma} \Xi(t,u),
        \end{equation}
        
        \item for all $q \geq 2$, $b_1 \geq \dotsb \geq b_q \in \Bs \setminus \{ X_0 \}$ such that $n_1(b_1) + \dotsb + n_1(b_q) \leq M$ and $\supp_{\Bs} \mathcal{F}(b_1, \dotsc, b_q) \not \subset \mathcal{N}$, there exists $\sigma_1,\dotsc,\sigma_q \leq L$ and $(\alpha_1, \dotsc, \alpha_q) \in [0,1]^q$ with $\alpha := \alpha_1 + \dotsb + \alpha_q \geq 1$ such that, for each $i \in \intset{1,q}$, $|b_i|\geq\sigma_i$ and
        \begin{equation} \label{eq:xib-othercross-XI-L}
            |\xi_{b_i}(t,u)| 
            \leq \frac{(ct)^{|b_i|}}{|b_i|!} t^{-\sigma_i} (\Xi(t,u))^{\alpha_i}.
        \end{equation}
    \end{itemize}
    Let $f_0$, $f_1$ be analytic vector fields on a neighborhood of $0$ with $f_0(0) = 0$.
    If $f_{\bb}(0) \notin \mathcal{N}(f)(0)$ and $\mathbb{P}$ is a component along $f_{\bb}(0)$ parallel to $\mathcal{N}(f)(0)$,
    \begin{equation} \label{eq:PZM-xibb-OXi}
        \mathbb{P} \mathcal{Z}_M(t,u)(0) 
        = \xi_{\bb}(t,u)
        + O\left(\Xi(t,u)\right).
    \end{equation}
\end{proposition}

\begin{proposition} \label{p:borduk}
    The following estimates hold as $(t,\|u_1\|_{L^\infty})\to 0$.
    \begin{enumerate}
        \item 
        \label{p:borduk/H0+H3}
        If \hyp{H3} holds, then
        \begin{equation} \label{u1uk=u1^N}
            U_k(t) = O\left( \|u_1\|_{L^{N+1}}^{N+1} + |x(t;u)| \right).
        \end{equation}
        
        \item 
        \label{p:borduk/H0+H2}
        If \hyp{H2} holds, then
        \begin{equation} \label{u1uk=uj^2+u1^N}
            U_k(t) = O\left( \|u_j\|_{L^2}^2 + \|u_1\|_{L^{N+1}}^{N+1}+ |x(t;u)| \right).
        \end{equation}
    \end{enumerate}
\end{proposition}

\begin{proof} 
    We apply \cref{thm:Magnus}.
    By \cref{eq:Magnus} with $M \gets N$,
    \begin{equation} \label{x=ZN+u^N+1}
        x(t;u)=\mathcal{Z}_N(t,u)(0)+O\left( \|u_1\|_{L^{N+1}}^{N+1} + |x(t;u)|^{1+\frac{1}{N}} \right). 
    \end{equation}

    \noindent \emph{Proof of \eqref{u1uk=u1^N}:} 
    By \eqref{eq:ZM=eta}, \cref{cor:eta-S1} and \eqref{xi_S1},
    \begin{equation}
        \mathcal{Z}_N(t,u)(0)=    
        \sum_{l=1}^{\infty} u_{l}(t) f_{M_{l-1}}(0)
        +\sum_{b\in\Bs_{\intset{2, N}}} \eta_b(t,u) f_b(0).
    \end{equation}
    Let $l \in \intset{1,k}$ and $\mathcal{N}:=\Bs_{\intset{1,N}} \setminus \{M_{l-1}\}$. 
    The assumption \hyp{H3} allows to consider $\mathbb{P}:\R^d \to \R$ giving a component along $f_{M_{l-1}}(0)$ parallel to $\mathcal{N}(f)(0)$. 
    Then
    \begin{equation} \label{PZ=ul}
        \mathbb{P} \mathcal{Z}_N(t,u)(0)=  u_{l}(t).
    \end{equation}
    By combining \eqref{x=ZN+u^N+1} and \eqref{PZ=ul}, we obtain $u_l(t)=O\left( \|u_1\|_{L^{N+1}}^{N+1} + |x(t;u)| \right)$ (using \cref{p:small-state} and $\|u\|_{W^{-1,\infty}} \to 0$ to absorb the $|x(t;u)|^{1+\frac{1}{N}}$ term).

    \bigskip

    \noindent \emph{Proof of \eqref{u1uk=uj^2+u1^N}:} 
    Let $l \in \intset{1,k}$ and $\mathcal{N}=\Bs_{\intset{1,N}} \setminus\{M_{l-1},W_{j',\nu};j' \geq j,\nu\in\N\}$.
    By assumption \hyp{H2}, we can consider $\mathbb{P}:\R^d \to \R$ a component along $f_{M_{l-1}}(0)$ parallel to $\mathcal{N}(f)(0)$.
    We intend to apply \cref{p:PZM-xibb-OXi} with $M \gets N$, $L \gets \max \{ k, 2j+1 \}$, $\bb \gets M_{l-1}$, so that \eqref{eq:PZM-xibb-OXi}, for the appropriate choice of $\Xi(t,u)$, will yield
    \begin{equation} \label{ZN=ul+uj2}
        \mathbb{P} \mathcal{Z}_N(t,u)(0) = u_{l}(t) + O\left( |U_k(t)|^2 +  \|u_j\|_{L^2}^2 \right).
    \end{equation}
    Combining \eqref{x=ZN+u^N+1} with \eqref{ZN=ul+uj2} concludes the proof of \eqref{u1uk=uj^2+u1^N} (using \cref{p:small-state} and $\|u\|_{W^{-1,\infty}} \to 0$ to absorb the $|x(t;u)|^{1+\frac{1}{N}}$ and $|U_k(t)|^2$ terms).

    \medskip \emph{Step 1: Estimates of other coordinates of the second kind.}
    Let $b \in \Bs_{\intset{1,N}}$ such that $b \notin \mathcal{N} \cup \{ \bb \}$. 
    By choice of $\mathcal{N}$, one necessarily has $b=W_{j',\nu}$ with $j' \geq j$. 
    By \eqref{bound-xiWjnu/j0} with $(p,j_0) \gets (1,j)$, $|b| \geq 2j+1$ and \eqref{eq:xib-XI-L} holds with $\sigma = 2j+1$ and $\Xi(t,u) := \|u_j\|_{L^2}^2$.    
   
    \medskip \emph{Step 2: Estimates of cross products.} 
    Let $q \geq 2$, $b_1 \geq \dotsb \geq b_q \in \Bs$ such that $n_1(b_1) + \dotsb + n_1(b_q) \leq N$ and $\supp_{\Bs} \mathcal{F}(b_1, \dotsc, b_q) \not \subset \mathcal{N}$. 
    Then necessarily $q=2$ and $n_1(b_1)=n_1(b_2)=1$. 
    \begin{itemize}
        \item If $b_i = M_{l'-1}$ for some $l' \in \intset{1,k}$, then $|b_i|=l'$ and, by \eqref{xi_S1},
        \begin{equation}
            |\xi_{b_i}(t,u)| = |u_{l'}(t)| = \frac{t^{|b_i|}}{|b_i|!} t^{-l'} l' ! |u_{l'}(t)|
        \end{equation}
        so \eqref{eq:xib-othercross-XI-L} holds with $\sigma_i = l'$, $\alpha_i = 1/2$ and $\Xi(t,u) = |(u_1,\dotsc,u_k)(t)|^2$.
        
        \item If $b_i = M_{l'-1}$ for some $l' > k$, 
        then $l'>j$ and by \eqref{bound-xiMj/J0} with $(p,j_0) \gets (2,j)$, \eqref{eq:xib-othercross-XI-L} holds with $\sigma_i = j+1$, $\alpha_i = 1/2$ and $\Xi(t,u) = t\|u_j\|_{L^2}^2$. \qedhere
    \end{itemize}
\end{proof}

\subsection{Interpolation inequalities}
\label{s:interpolation}

We prove interpolation estimates used in \cref{s:j<=k} to absorb remainders in the case $j < k$.
They are based on our nonlinear interpolation inequality of \cref{thm:FM-GN} of \cite{Marbach2023}.

When $j = k$, the usual Gagliardo--Nirenberg interpolation inequality of \cref{thm:GN} is sufficient, and the results of this section are not required (see \cref{s:j=k}).

\subsubsection{Interpolation inequalities for critical terms}

We start with a definition, which lightens proofs, by avoiding to keep track of the details.

\begin{definition}
    \label{def:om}
    Let $m \in \llbracket -1, \infty \llbracket$.
    Given two observables $A(T,u)$ and $B(T,u)$, we write
    \begin{equation}
        A(T,u) = \mathfrak{o}_m (B(T,u))
    \end{equation}
    when there exist $\alpha \in \R$ and $\beta > 0$ such that, for all $T \in (0,1]$ and $u \in W^{m,\infty}(0,T) \cap L^1(0,T)$ such that $\|u\|_{W^{m,\infty}} < 1$,
    \begin{equation}
        A(T,u) \leq T^\alpha \| u \|_{W^{m,\infty}}^\beta B(T,u).
    \end{equation}
    Even when $\alpha < 0$, for fixed times and small enough controls, $A$ is negligible with respect to $B$.
\end{definition}

Using \cref{thm:FM-GN}, we now prove that the principal remainder term $\|u_1\|_{L^{M+1}}^{M+1}$ will be absorbable by the coercive drift for small enough controls in $W^{m,\infty}$.

\begin{proposition} \label{p:interp-u1}
    Let $1 \leq j < k$ and $m \in \N$.
    Define $M, N \in \N$ by
    \begin{equation} \label{def:M0N0}
        M := 3 + \left\lceil \frac{2(k+j-2)}{m+1} \right\rceil,
        \qquad
        N := \left\lceil  \frac{k-1}{m+1} \right\rceil.
    \end{equation}
    For $T \in (0,1]$ and $u \in W^{m,\infty}(0,T)$,
    \begin{align}
        \label{GN:M0}
        \| u_1 \|_{L^{M+1}}^{M+1}
        & = \mathfrak{o}_m \big(\xi_\q(T,u)\big), \\
        \label{GN:N1}
        \| u_1 \|_{L^{N+1}}^{N+1}
        & = \mathfrak{o}_m \big( \xi_\q(T,u)^{\frac{k+m}{2(k+j+2m)}} \big).
    \end{align}
\end{proposition}

\begin{proof}
    Let $M_0:=3+\frac{2(k+j-2)}{m+1} > 3$ and $N_0:=\frac{k-1}{m+1}$.

    \medskip \noindent \emph{Proof of \eqref{GN:M0}:} 
    We apply \cref{thm:FM-GN} with $\phi \gets u_k$, $(\fk, \fj, \fl) \gets (k-j, k-1, k+m)$, so that $D^\fj \phi = u_1$ and $\phi D^\fk \phi = u_k u_j$. 
    We also take $(p,q,r) \gets (M_0+1,2,\infty)$ and $\theta = \theta^* \gets \frac{k+j-2}{k+j+2m}$ that satisfy \eqref{eq:critic-pqr-FM-GN}.
    Then \eqref{eq:GN-FM-estimate} entails that there exists $\alpha \in \R$ such that
    \begin{equation} \label{GN:M0_vrai}
        \|u_1\|_{L^{M_0+1}}^{M_0+1} 
        \leq C \|D^m u\|_{L^\infty}^{M_0-3} \xi_\q(T,u) + C T^\alpha \xi_\q^{\frac{M_0+1}{4}}
        = \mathfrak{o}_m (\xi_\q(T,u)),
    \end{equation}
    where we used \cref{lem:xi-naive-inf} to absorb the lower-order term. 
    Since $M \geq M_0$,
    \begin{equation}
        \|u_1\|_{L^{M+1}}^{M+1}
        \leq \|u_1\|_{L^\infty}^{M-M_0} \|u_1\|_{L^{M_0+1}}^{M_0+1} 
        = \mathfrak{o}_m (\xi_\q(T,u)). 
    \end{equation}
    
    \medskip \noindent \emph{Proof of \eqref{GN:N1}:} 
    By definition of $N_0$ and $M_0$ we have $N_0 \leq M_0$ thus, using Hölder's inequality with $T \leq 1$, \eqref{GN:M0_vrai} and the relation $\frac{N_0+1}{M_0+1}=\frac{k+m}{2(k+j+2m)}$, we obtain
    \begin{equation}
        \|u_1\|_{L^{N_0+1}}^{N_0+1}
        \leq \|u_1\|_{L^{M_0+1}}^{N_0+1}
        = \mathfrak{o}_m \big( \xi_\q(T,u)^{\frac{k+m}{2(k+j+2m)}} \big).
    \end{equation}
    Since $N \geq N_0$, we conclude as above using $\|u_1\|_{L^{N+1}}^{N+1} \leq \|u_1\|_{L^\infty}^{N-N_0} \|u_1\|_{L^{N_0+1}}^{N_0+1}$. 
\end{proof}

In the course of the proof, we will also need interpolation estimates on $\|u_j\|_{L^2}$.

\begin{proposition} \label{p:interp-uj}
    Let $j < k \in \N^*$ and $m \in \llbracket -1, \infty \llbracket$.
    For $T \in (0,1]$ and $u \in W^{m,\infty}(0,T)$,
    \begin{equation} \label{eq:interp-uj}
        \| u_j \|_{L^2} = \mathfrak{o}_m \big( \xi_\q(T,u)^{\frac{j+m}{2(k+j+2m)}} \big).
    \end{equation}
\end{proposition}

\begin{proof}
    When $(j,m) = (1,-1)$ (thus $k > 1$), by Hölder's inequality $\|u_1\|_{L^2} \leq T^{\frac 12} \|u_1\|_{L^\infty} = \mathfrak{o}_m(1)$.
    Assume that $(j,m) \neq (1,-1)$.
    Apply \cref{thm:FM-GN} with $\phi \gets u_k$, $(K, J, L) \gets (k-j, k-j, k+m)$ so that $D^\fj \phi = u_j$ and $\phi D^\fk \phi = u_k u_j$.
    In particular $J < L$ because $-j<m$.
    We also set $(p,q,r) \gets (2,1,2)$ and $\theta = \theta^* \gets \frac{k-j}{k+j+2m}$, that satisfy \eqref{eq:critic-pqr-FM-GN}.
    We obtain, for some $\alpha \in \R$,
    \begin{equation}
        \|u_j\|_{L^2} \leq C \|D^m u\|_{L^2}^{\theta^*}\| u_j u_k \|_{L^1}^{\frac{1-\theta^*}{2}} + C T^\alpha \| u_j u_k \|_{L^1}^{\frac{1}{2}}.
    \end{equation}
    Then, Hölder's inequality and $T \leq 1$ proves 
    \begin{equation}
        \| u_j u_k \|_{L^1} \leq T^{\frac 12} \xi_\q(T,u)^{\frac 12}.
    \end{equation}
    Using \cref{lem:xi-naive-inf} to absorb the lower-order term, we conclude that
    \begin{equation}
        \|u_j\|_{L^2} = \mathfrak{o}_m\big(\xi_\q^{\frac{1-\theta^*}{4}}(T,u)\big).
    \end{equation}
    In the above computations, if $m=-1$ then $D^m u$ denotes $u_1$ and $\|u_1\|_{L^\infty} = \|u\|_{W^{-1,\infty}}$ by our definition of the $W^{-1,\infty}$-norm (see \eqref{def:norm_W-1}).
\end{proof}

\subsubsection{Interpolation inequalities for sub-critical terms}

In the proof of \cref{thm:Qjk_intro} for $j < k$, we will need to estimate the sub-critical terms $\|u_k\|_{L^1}$ and $\| u_{j'}^2 u_{k'} \|$ for some $j' > j$.
The restriction $K \leq J$ in \cref{thm:FM-GN} prevents from applying this estimate for these quantities.
Instead, we rely on the following inequality.

\begin{lemma} \label{lem:interp-middle}
    Let $\fk \in \N$. 
    There exists $C>0$ such that, for every $T\in(0,1)$ and every $\phi \in W^{\fk,2}((0,T);\R)$ with $(\phi , \dots , D^{\fk-1} \phi)(0)=0$,
    \begin{equation}
        \| D^{\left\lfloor \frac{\fk}{2} \right\rfloor} \phi \|_{L^2}^2 \leq C \left( T^{\frac{1}{2}} \| \phi D^\fk \phi \|_{L^2} + |(\phi,\dots,D^{\fk-1} \phi)(T)|^2 \right).
    \end{equation}
\end{lemma}

\begin{proof}
    By density, it suffices to prove the estimate when $\phi \in \CC^\fk([0,T];\R)$.
    Using integrations by part, one obtains, when $\fk$ is even
    \begin{equation}
        \int_0^T (D^{\frac{\fk}{2}} \phi)^2
        = (-1)^{\frac{\fk}{2}} \int_0^T \phi D^\fk \phi + 
        \sum_{i=1}^{\fk/2} (-1)^{i-1} D^{\frac{\fk}{2}-i} \phi(T) D^{\frac{\fk}{2}+i-1} \phi(T)
    \end{equation}
    and when $\fk=2n-1$ with $n\in\N^*$
    \begin{equation}
        \begin{split}
            \int_0^T \tau \phi D^\fk \phi(\tau) \dd\tau
            & =  (-1)^{n-1} \int_0^T
            ( \tau D^{n-1} \phi(\tau) + (n-1) D^{n-2} \phi(\tau) ) 
            D^n \phi(\tau) \dd\tau   + B
            \\ & = (-1)^n  \left( n-\frac{1}{2} \right) \int_0^T (D^{n-1}\phi)^2  + B'
        \end{split}
    \end{equation}
    where $B$ and $B'$ are boundary terms bounded by 
    $C |(\phi,\dots,D^{\fk-1}\phi)(T)|^2$.
    In both cases, Hölder's inequality gives the conclusion.
    (These computations also hold when $K = 1$).
\end{proof}

\begin{corollary}
    \label{cor:interp-un}
    Let $n := \lceil \frac{j+k}{2} \rceil$.
    There exists $C>0$ such that, for all $t \in (0,1]$ and $u \in \lone$,
    \begin{equation}
        \| u_n \|_{L^2} \leq C \left(t \xi_\q(t,u)\right)^{\frac 14} + C |U_k(t)|.
    \end{equation}
\end{corollary}

\begin{proof}
    Apply \cref{lem:interp-middle} to $\phi \gets u_k$ and $K \gets k-j$, so that $n = k - \lfloor K / 2 \rfloor$.
\end{proof}

\begin{lemma}
    \label{lem:bound-uj'k'}
    Let $1 \leq j < j' \leq k' \leq k$ be such that $2j+2k+1 \le 3j'+k'$. 
    Let $n := \lceil \frac{j+k}{2} \rceil$.
    There exists $C > 0$ such that, for $t \in (0,1]$ and $u \in \lone$,
    \begin{equation}
        \| u_{j'}^2 u_{k'} \|_{L^1} \leq C \| u_j \|_{L^2} \| u_n \|_{L^2}^2.
    \end{equation}
\end{lemma}

\begin{proof}
    First, since $k' \in \N$ and $j' \le k'$, the assumption $2j + 2k + 1 \le 3j' + k'$ entails that $n \le k'$.
    Second, combining $k' \le k$, $j + 1 \le j'$, $2j+2k+1 \le 3j' + k'$ and $2n \le j + k + 1$, we obtain
    \begin{equation} \label{eq:direct-index}
        n-j'+\frac14 \leq \frac{n-j}{2}.
    \end{equation}
    By Hölder's inequality and Poincaré's inequality for $k' \ge n$,
    \begin{equation}
        \| u_{j'}^2 u_{k'} \|_{L^1} 
        \le \| u_{j'} \|_{L^4}^2 \| u_{k'} \|_{L^2}
        \le \| u_{j'} \|_{L^4}^2 \| u_n \|_{L^2}.
    \end{equation}
    If $j' > n$, by Poincaré's inequality
    \begin{equation}
        \| u_{j'}^2 u_{k'} \|_{L^1} 
        \le \| u_n \|_{L^2}^3 
        \le \| u_j \|_{L^2} \| u_n \|_{L^2}^2.
    \end{equation}
    If $j' \le n$, by \cref{cor:GN}  with $\phi \gets u_n$, $(J,L) \gets (n-j',n-j)$ and $(p,q,r,s) \gets (4,2,2,2)$,
    \begin{equation}
        \|u_{j'}\|_{L^4} \le C \|u_j\|_{L^2}^\theta \|u_n\|_{L^2}^{1-\theta} 
        \quad \text{for} \quad
        \theta := \frac{n-j'+\frac14}{n-j} > \theta^* := \frac{n-j'}{n-j}.
    \end{equation}
    By \eqref{eq:direct-index}, $2 \theta \le 1$.
    Since $\|u_n\|_{L^2} \le \|u_j\|_{L^2}$, the result follows.
\end{proof}

\begin{lemma}
    \label{lem:bound-uk}
    Let $1 \leq j < k$.
    Let $n := \lceil \frac{j+k}{2} \rceil$.
    There exists $C > 0$ such that, for $t \in (0,1]$ and $u \in \lone$,
    \begin{equation}
        \| u_k \|_{L^1}^3 \leq C \| u_j \|_{L^2} \| u_n \|_{L^2}^2.
    \end{equation}
\end{lemma}

\begin{proof}
    For $t \leq 1$, by Hölder's inequality $\| u_k \|_{L^1} \leq \| u_k \|_{L^3}$.
    Thus the conclusion follows from \cref{lem:bound-uj'k'} with $(j',k') \gets (k,k)$ which satisfy $3j'+k' = 4k > 2j+2k+1$ since $k > j$.
\end{proof}

\subsection{Proof of the presence of the drift for \texorpdfstring{$m \geq 0$}{m >= 0}}
\label{s:obs-m}

We prove \cref{thm:Qjk_intro} when $m \in \N$.
The case $m = -1$ is postponed to \cref{s:obs-1}.

\subsubsection{The simplest case $j=k$}
\label{s:j=k}

We prove \cref{thm:Qjk_intro} when $j=k \geq 2$ and $m \in \N$, as a consequence of the following statement (the case $j = k = 1$ being easier and already known, see \cref{rk:stefani}).
We isolate it here for pedagogical reasons.
In particular, it relies on the coercivity of the functional $\int u_j^4$, and therefore only requires usual Gagliardo--Nirenberg interpolation inequalities.

\begin{theorem} \label{Thm:Qjj_m_derive}
    Let $m \in \N$, $j\geq 2$ and $M=3+\left\lceil \frac{4(j-1)}{m+1} \right\rceil$. 
    We assume $f_{Q_{j,j,j}}(0) \notin \mathcal{N}_{j,j}^M(f)(0)$. 
    Then, system \eqref{syst} has a drift along $f_{Q_{j,j,j}}(0)$,
    parallel to $\mathcal{N}_{j,j}^M(f)(0)$, as $T \to 0$ and $\|u\|_{W^{m,\infty}} \to 0$.
\end{theorem}

\begin{proof}
    \step{Dominant part of the logarithm}
    Using \eqref{xi_Qljknu}, \eqref{eq:eta_Qjk-new} of \cref{p:eta_Qjk-new} and Hölder and Young inequalities, one has, as $T\to 0$ and uniformly with respect to $u\in L^1(0,T)$:
    \begin{equation} \label{eq:eta-Qjj}
        \eta_{Q_{j,j,j}}(T,u) = 
        \int_0^T \frac{u_j^4}{4!}
         + O\left( T^{\frac{1}{3}} \|u_j\|_{L^4}^4 + |U_j(T)|^4 \right).
    \end{equation}
    
    \step{Vectorial relations and closed-loop estimate} 
    Let $N=\left\lfloor \frac{M}{4} \right\rfloor$. 
    By \cref{p:H0+H3}, \hyp{H3} holds because $M \geq 4N$.
    Thus \cref{p:borduk/H0+H3} of \cref{p:borduk} gives
    \begin{equation}
        |U_j(T)| = O\left( \|u_1\|_{L^{N+1}}^{N+1} + |x(T;u)| \right).
    \end{equation}
    Using that $4 (N+1) \geq (M+1)$, Hölder's inequality, $T \leq 1$ and $\|u_1\|_{L^\infty} \leq 1$, we obtain
    \begin{equation}
        \label{u1uj^4}
        |U_j(T)|^4 = O\left( \|u_1\|_{L^{M+1}}^{M+1} + |x(T;u)|^4 \right).
    \end{equation}

    \step{Interpolation inequality}
    Let $M_0 := 3+\frac{4(j-1)}{m+1}$.
    We prove the existence of $C>0$ such that, for all $T \in (0,1)$ and $u \in W^{m,\infty}(0,T)$,
    \begin{equation} \label{Qj:u1M}
        \|u_1\|_{L^{M+1}}^{M+1} \leq C (1+T^{M_0-4j+1}) \|u\|_{W^{m,\infty}}^{M-3} \|u_j\|_{L^4}^4.
    \end{equation}
    Applying \cref{thm:GN} with $\phi \gets u_j$, 
    $(J,L) \gets (j-1,j+m)$, 
    $(p,r,q,s) \gets (M_0+1,\infty,4,4)$, 
    $\theta=\frac{j-1}{j+m}=\frac{M_0-3}{M_0+1}$,  we obtain a constant $C>0$ such that, for every $T>0$ and $u \in W^{m,\infty}(0,T)$,
    \begin{equation} \label{u1M0}
        \begin{split}
            \|u_1\|_{L^{M_0+1}}^{M_0+1}
            & \leq C \left( \|u^{(m)}\|_{L^\infty}^{M_0-3} +  T^{1-(j-\frac{3}{4})(M_0+1)} \|u_j\|_{L^4}^{M_0-3}\right) \|u_j\|_{L^4}^4 \\
            & \leq C (1+T^{M_0-4j+1}) \|u\|_{W^{m,\infty}}^{M_0-3} \|u_j\|_{L^4}^4
        \end{split}
    \end{equation}
    where the last inequality results from
    $\|u_j\|_{L^4}^{M_0-3} \leq \left( T (T^j\|u\|_{L^\infty})^4 \right)^{\frac{M_0-3}{4}} = T^{(j+\frac{1}{4})(M_0-3)} \|u\|_{L^\infty}^{M_0-3}$.
    Then \eqref{Qj:u1M} is a consequence of \eqref{u1M0} because $M \geq M_0$ and $m \ge 0$. 

    \step{Drift} 
    Let $\mathbb{P}:\R^d \to \R$ be a component along $f_{Q_{j,j,j}}(0)$ parallel to $\mathcal{N}_{j,j}^M(f)(0)$.
    By \cref{thm:Magnus}, as  $(T,\|u_1\|_{L^\infty}) \to 0$,
    \begin{equation} \label{Px=eta+O(u1^N+1)}
        \mathbb{P} x(T;u) =  \eta_{Q_{j,j,j}}(T,u) + O\left( \|u_1\|_{L^{M+1}}^{M+1} + |x(T;u)|^{1+\frac{1}{M}} \right) 
    \end{equation}
    and by \eqref{eq:eta-Qjj} and \eqref{u1uj^4},
    \begin{equation}
        \mathbb{P}  x(T;u)  = \int_0^T \frac{u_j^4}{4!} +  O\left(
        T^{\frac{1}{3}} \|u_j\|_{L^4}^4
        +  \|u_1\|_{L^{M+1}}^{M+1}
        + |x(T;u)|^{1+\frac{1}{M}} \right).
    \end{equation}
    Then \eqref{Qj:u1M} proves 
    \begin{equation}
        \mathbb{P}  x(T;u)  = \int_0^T \frac{u_j^4}{4!} +  O\left( 
    \left( T^{\frac{1}{3}} +  
    (1+T^{M_0-4j+1}) \|u\|_{W^{m,\infty}}^{M_0-3} 
    \right) \|u_j\|_{L^4}^4
    + |x(T;u)|^{1+\frac{1}{M}} \right), 
    \end{equation}
    which entails \eqref{eq:def-drift} for small enough $T$, then small enough $u$ (with a smallness condition depending on $T$; except when $m=0$, for which $M = M_0=4j-1$).
\end{proof}

\subsubsection{The general case $j < k$}
\label{s:j<=k}

We prove \cref{thm:Qjk_intro} when $j < k$ and $m \in \N$, as a consequence of the following statement.

\begin{theorem} \label{Thm:Qjk_m_derive}
    Let $m\in\N$, $1 \leq j < k$,
    \begin{equation} \label{def:Pi&M_jkm}
    \pi=\pi(j,m):=1+\left\lceil \frac{2j-2}{m+1} \right\rceil
    \quad  \text{ and } \quad 
    M=M(j,k,m):=3+\left\lceil \frac{2(k+j-2)}{m+1} \right\rceil.
    \end{equation}
    We assume that $f_\q(0) \notin \mathcal{N}_{j,k}^M(f)(0)$ and
    \begin{itemize}
        \item either $k \leq (2j+m)$ and\footnote{The assumption $f_{W_j}(0) \in \vect\{ f_b(0) ; b \in \Bs_{\intset{1,\pi}\setminus\{2\}} \}$, not included in \cref{thm:Qjk_intro}, is legitimate because it is a necessary condition for $W^{m,\infty}$-STLC, see \cite[Theorem 1.11]{BeauchardMarbach2026}.} $f_{W_j}(0) \in \vect\{ f_b(0) ; b \in \Bs_{\intset{1,\pi}\setminus\{2\}} \}$;
        \item or $(2j+m)<k$ and $f_{W_j}(0) \in \vect\{ f_b(0) ; b \in \Bs_{\intset{1,\pi}\setminus\{2\}} \setminus \{M_{k-1}\} \}$.
    \end{itemize}
    Then, system \eqref{syst} has a drift along $f_\q(0)$,
    parallel to $\mathcal{N}_{j,k}^M(f)(0)$, as $T \to 0$ and $\|u\|_{W^{m,\infty}} \to 0$.
\end{theorem}

\begin{proof}
    Let $\mathbb{P} :\R^d \to \R$ be a component along $f_\q(0)$ parallel to $\mathcal{N}_{j,k}^M(f)(0)$.
    By \cref{thm:Magnus},
    \begin{equation} \label{Magnus_Qjk}
        \mathbb{P} x(T;u) = \xi_\q(T,u) 
        + O\left( |\eta_\q(T,u) - \xi_\q(T,u)| +  
        \|u_1\|_{L^{M+1}}^{M+1} + |x(T;u)|^{1+\frac{1}{M}} \right). 
    \end{equation}
    Using \eqref{eq:Qjk-eta-xi-closed} of \cref{lem:Qjk-eta-xi-closed} below and \eqref{GN:M0} of \cref{p:interp-u1}, there exists $\gamma > 0$ such that
    \begin{equation}
        \label{eq:Px-xi=om}
        | \mathbb{P} x(T;u) - \xi_\q(T,u) | 
        = \mathfrak{o}_m(\xi_\q(T,u)) +  O(|x(T;u)|^{1+\gamma}),
    \end{equation}
    which proves the drift in the sense of \cref{def:drift}.
    Due to the \cref{def:om} of $\mathfrak{o}_m$, the smallness assumption on $u$ might depend on $T$ (when $\alpha < 0$). 
    This still denies STLC.
\end{proof}

The key argument is therefore the following estimate for $\eta_\q - \xi_\q$.

\begin{lemma}
    \label{lem:Qjk-eta-xi-closed}
    Under the assumptions of \cref{Thm:Qjk_m_derive}, there exists $\gamma > 0$ such that, for $T \in (0,1]$ and $u \in W^{m,\infty}(0,T)$ with $\|u\|_{W^{m,\infty}} \leq 1$,
    \begin{equation}
        \label{eq:Qjk-eta-xi-closed}
        | \eta_\q(T,u) - \xi_\q(T,u) |
        = \mathfrak{o}_m( \xi_\q(T,u) ) + O(| x(T;u) |^{1+\gamma}).
    \end{equation}
\end{lemma}

\begin{proof}
    We write $\eta_\q$ and $\xi_\q$ instead of $\eta_\q(T,u)$ and $\xi_\q(T,u)$ to lighten the computations.
    By \eqref{eq:eta_Qjk-new} of \cref{p:eta_Qjk-new} and \cref{lem:bound-uj'k',lem:bound-uk}, we obtain
    \begin{equation}
        \eta_\q = \xi_\q  + O \Big( |U_k(T)| \Big( 
        |U_k(T)|^3 
        + \| u_j \|_{L^2} \| u_n \|_{L^2}^2
        + |U_k(T)| \| u_j \|_{L^2}^2 
        + \| u_j^2 u_k \|_{L^1} \Big) \Big),
    \end{equation}
    where $n := \lceil \frac{j+k}{2} \rceil$.
    Using \cref{cor:interp-un} for $\|u_n\|_{L^2}^2$ and Cauchy--Schwarz for $\|u_j^2 u_k\|_{L^1}$,
    \begin{equation} \label{eta_Qjk_1}
        \eta_\q = \xi_\q + O \Big(  |U_k(T)|^4 
        + |U_k(T)|^2 \| u_{j} \|_{L^2}^2
        + |U_k(T)| \| u_{j} \|_{L^2} \xi_\q^{\frac 12}\Big).
    \end{equation}
    The main task is to use the closed-loop estimates to handle the boundary term $U_k(T)$.

    \medskip \noindent \emph{Step1: Preliminary remarks.}
    Let $N$ be defined by \cref{def:M0N0}. 
    Then $\pi+2N \leq M$ because
    \begin{equation} \label{eq:Pi+2N<=M}
        \pi+2N = 1+ \left\lceil \frac{2j-2}{m+1}\right\rceil 
        +2 \left\lceil \frac{k-1}{m+1} \right\rceil < 4 + \frac{2(j-1)}{m+1} + \frac{2(k-1)}{m+1} \leq 1 + M.
    \end{equation}
    By \eqref{eq:interp-uj} of \cref{p:interp-uj} and \eqref{GN:N1} of \cref{p:interp-u1},
    \begin{align} 
        \label{eq:interp-uj-bis}
        \|u_j\|_{L^2}^2 & = \mathfrak{o}_m\big(\xi_\q^{\frac{j+m}{k+j+2m}}\big), \\
        \label{GN:N1-bis}
        \| u_1 \|_{L^{N+1}}^{N+1}
        & = \mathfrak{o}_m\big(\xi_\q^{\frac{k+m}{2(k+j+2m)}}\big).
    \end{align}

    \medskip \noindent \emph{Step 2: We prove the following estimate on the boundary term:}
    \begin{equation} \label{u1uk_closed_xi}
        |U_k(T)| = \mathfrak{o}_m \big( \xi_\q^{\frac{k+m}{2(k+j+2m)}}\big) + O(|x(T;u)|).
    \end{equation}
    We consider the two cases of \cref{Thm:Qjk_m_derive}.
    \begin{itemize}
        \item \emph{Case $j<k\leq 2j+m$.}
        We assume that $f_{W_j}(0) \in \vect \{f_b(0);b\in\Bs_{\intset{1,\pi}\setminus\{2\}} \}$ and $f_\q(0) \notin \mathcal{N}_{j,k}^M(f)(0)$ where $\pi=\pi(j,m)$ and $M=M(j,k,m)$ are defined in \cref{def:Pi&M_jkm}.

         If $\pi = 1$, $j = 1$, so $k \leq 2j+m =m+2$ so $N = 1$. 
        Moreover, since $k \geq 2$, $M \geq 4$ so $M \geq 2N+2$.
        When $\pi > 1$, by \eqref{eq:Pi+2N<=M}, $M \geq 2N+\pi$.
        In both cases, we can apply \cref{p:H0+H2}.
        Thus \hyp{H2} holds and by \cref{p:borduk/H0+H2} of \cref{p:borduk}
        \begin{equation} \label{u1uk<uj^2+u1^N+1}
            |U_k(T)| = O\left( \|u_j\|_{L^2}^2 + \|u_1\|_{L^{N+1}}^{N+1} + |x(T;u)| \right).
        \end{equation}
        The assumption $k \leq 2j+m$ gives
        \begin{equation} \label{eq:indices_Uj}
            \frac{j+m}{k+j+2m} \geq \frac{k+m}{2(k+j+2m)}.
        \end{equation}
        Thus, using \eqref{eq:interp-uj-bis} and \eqref{GN:N1-bis} in \eqref{u1uk<uj^2+u1^N+1} proves \eqref{u1uk_closed_xi}.

        \item \emph{Case $2j+m < k$.} 
        We assume that $f_{W_j}(0) \in \vect \{f_b(0);b\in\Bs_{\intset{1,\pi} \setminus\{2\}} \setminus \{M_{k-1}\} \}$ and $f_\q(0) \notin \mathcal{N}_{j,k}^M(f)(0)$ where $\pi=\pi(j,m)$ and $M=M(j,k,m)$ are defined in \cref{def:Pi&M_jkm}.
    
        Recalling \eqref{eq:Pi+2N<=M}, $\pi+2N \leq M$.
        Since $k \geq 2j+(m+1)$, we also have $N \geq \pi$ and $N \geq 2$.
        Moreover, we have $2N+2 \leq M$ because $2N+2$ is an integer and
        \begin{equation}
            2N+2 = 2 \left\lceil \frac{k-1}{m+1} \right\rceil 
            < 4 + \frac{2k-2}{m+1}
            \leq 4 + \frac{2(k+j-2)}{m+1}
            \leq 1 + M.
        \end{equation}
        This proves that $M \geq 2 \max \{ N, \pi, 2 \} + \max \{ \pi, 2 \}$.
        By \cref{p:H0+H3}, \hyp{H3} holds.
        Thus, by \cref{p:borduk/H0+H3} of \cref{p:borduk},
        \begin{equation} \label{u1uk<u1^N+1}
            U_k(T)= O\left( \|u_1\|_{L^{N+1}}^{N+1} + |x(T;u)| \right).
        \end{equation}
        Thus, using \eqref{GN:N1-bis} proves \eqref{u1uk_closed_xi}.
    \end{itemize}
    
    \medskip \noindent \emph{Step 3: Counting powers of the drift.} 
    The assumption $j < k$ implies
    \begin{equation} \label{exposants-1/4}
        \frac{k+m}{2(k+j+2m)} > \frac{1}{4}.
    \end{equation}
    Substituting \eqref{eq:interp-uj-bis} and \eqref{u1uk_closed_xi} into \eqref{eta_Qjk_1} proves \eqref{eq:Qjk-eta-xi-closed}.
\end{proof}

\subsection{Proof of the presence of a weak drift for \texorpdfstring{$m=-1$}{m=-1}}
\label{s:obs-1}

We now investigate the case $m = -1$.
We follow here a methodology which we introduced in \cite[Section 10]{BeauchardMarbach2026} to handle this very low-regularity case where the remainders cannot be absorbed by interpolation.
\emph{We encourage the interested reader to first consult this reference.}

\begin{definition}[Weak drift] 
	\label{def:weak-drift}
	Let $\bb \in \Bs$ and $\mathcal{N} \subset \Br(X)$.
	We say that system \eqref{syst} has a \emph{weak drift along $f_{\bb}(0)$, parallel to $\mathcal{N}(f)(0)$, as $(t, \|u_1\|_{L^\infty}) \to 0$} when there exists $C>0$, $\beta>0$ such that, for every $\varepsilon>0$, there exists $\rho > 0$ such that for all $t \in (0,\rho)$, for all $u \in W^{-1,\infty}(0,t)$ with $\|u_1\|_{L^\infty}<\rho$,
	\begin{equation} \label{eq:def-drift-weak}
	   \mathbb{P} x(t;u) \geq (1-\varepsilon) \xi_{\bb}(t,u) - C \|u_1\|_{L^\infty}^{\beta} |x(t;u)|
	\end{equation}
	where $\mathbb{P}$ gives a component along $f_{\bb}(0)$ parallel to $\mathcal{N}(f)(0)$ and $(\xi_{b})_{b\in\Bs}$ are the coordinates of the second kind associated with $\Bs$ (see \cref{Prop:Coord_Bstar}).
\end{definition}

We prove \cref{thm:Qjk_intro} for $m=-1$ as a consequence of the following more precise result.

\begin{theorem} \label{Thm:m=-1}
    Let $j \leq k \in \N^*$. 
    Assume that\footnote{The assumption that $\mathcal{L}(f)(0) = \R^d$, not included in \cref{thm:Qjk_intro}, is legitimate because it is a necessary condition for $W^{-1,\infty}$-STLC.} $\mathcal{L}(f)(0) = \R^d$, $f_\q(0) \notin \mathcal{N}_{j,k}^{\infty}(f)(0)$ and that
    \begin{itemize}
    \item either $j=k$,
    \item or $j<k \leq 2j-1$ and\footnote{The assumption $f_{W_j}(0)\in \vect\{ f_b(0) ; b \in \Bs_{\N^* \setminus \{2\}} \}$, not included in \cref{thm:Qjk_intro}, is legitimate because it is a necessary condition for $W^{-1,\infty}$-STLC (see \cite[Theorem 1.11]{BeauchardMarbach2026}).} $f_{W_j}(0)\in \vect\{ f_b(0) ; b \in \Bs_{\N^* \setminus \{2\}} \}$,
    \item or $j<k$, $2j-1<k$ and $f_{W_j}(0)\in \vect\{ f_b(0) ; b \in \Bs_{\N^* \setminus \{2\}} \setminus \{M_{k-1}\} \}$.
    \end{itemize}
    Then system \eqref{syst} has a weak drift along $f_\q(0)$, parallel to $\mathcal{N}_{j,k}^{\infty}(f)(0)$, as $(t,\|u_1\|_{L^\infty}) \to 0$.
\end{theorem}

\subsubsection{The semi-nilpotent case}

We say that $f_1$ is \emph{semi-nilpotent at $0$ with respect to $f_0$} when there exists $R \ge 0$ such that
\begin{equation}
    \label{eq:semi-nilpotent}
    \forall b \in \Br(X), \quad 
    n_1(b) > R \quad \Rightarrow \quad f_b(0) = 0.
\end{equation}
We prove \cref{Thm:m=-1} under this assumption.
The proof follows the same steps as the one of \cref{Thm:Qjk_m_derive}. 
A key point is that the approximation formula \eqref{eq:Magnus} of \cref{thm:Magnus} is replaced with the following one (see \cite[Proposition 10.6]{BeauchardMarbach2026}).

\begin{proposition}
    \label{thm:Key_2}
    Assume that $\mathcal{L}(f)(0)=\R^d$ and that \eqref{eq:semi-nilpotent} holds.
    Then
    \begin{equation} \label{eq:x=ZM+O+Nilp}
        x(t;u) = \mathcal{Z}_R(t,u)(0) + O\left(\|u_1\|_{L^\infty} |x(t;u)| \right).
    \end{equation}
\end{proposition}

One must substitute \cref{thm:Key_2} for \cref{thm:Magnus} both in the proofs of the presence of the drift, and in the closed-loops estimates.

\begin{proposition} \label{p:nilp_u1uk}
    Assume that $\mathcal{L}(f)(0)=\R^d$ and that \eqref{eq:semi-nilpotent} holds for some $R \ge 0$.
    \begin{enumerate}
        \item \label{p:nilp_u1uk-1}
        If $j<k$, $f_{W_j}(0) \in \vect \{ f_b(0) ; b \in \Bs_{\llbracket 1 , \infty \llbracket \setminus\{2\}} \}$ and $f_\q(0)\notin \mathcal{N}_{j,k}^{\infty}(f)(0)$, then
        \begin{equation} \label{u1uk=uj^2+x}
            U_k(t)=O\left( \|u_j\|_{L^2}^2  + |x(t;u)| \right).
        \end{equation}
    
        \item \label{p:nilp_u1uk-2}
        If one of the following properties holds
        \begin{itemize}
            \item $j=k$, and $f_\q(0)\notin \mathcal{N}_{j,j}^{\infty}(f)(0)$,
            \item $j<k$, $f_{W_j}(0) \in \vect \{ f_b(0) ; b \in \Bs_{\llbracket 1 , \infty \llbracket \setminus\{2\}} \setminus\{M_{k-1}\} \}$ and $f_\q(0)\notin \mathcal{N}_{j,k}^{\infty}(f)(0)$, 
        \end{itemize}
        then
        \begin{equation} \label{u1uk=x}
            U_k(t) = O\left( |x(t;u)| \right).
        \end{equation}      
    \end{enumerate}
\end{proposition}

\begin{proof}
    Let $N \geq R$. 
    By \cref{thm:Key_2},
    \begin{equation} \label{x=ZN+}
        x(t;u) = \mathcal{Z}_N(t,u)(0) + O\left(\|u_1\|_{L^\infty} |x(t;u)| \right).
    \end{equation}

    \medskip \noindent \emph{Proof of \cref{p:nilp_u1uk-1}.} 
    There exists $\pi \geq 2$ such that $f_{W_j}(0) \in \vect \{ f_b(0) ; b \in \Bs_{\intset{1,\pi} \setminus\{2\}} \}$. 
    Let $M \in \N$ such that $M \geq 2N + \pi$. 
    Then $f_\q(0)\notin \mathcal{N}_{j,k}^{M}(f)(0)$. 
    By \cref{p:H0+H2}, \hyp{H2} holds. 
    Thus, for $l \in \intset{1,k}$ and $\mathcal{N}=\Bs_{\intset{1,N}} \setminus\{M_{l-1},W_{j',\nu};j'\geq j,\nu\in\N\}$, one may consider $\mathbb{P}:\R^d \to \R$ giving a component along $f_{M_{l-1}}(0)$ parallel to  $\mathcal{N}(f)(0)$. 
    We deduce from \cref{x=ZN+} that
    \begin{equation}
        \mathbb{P} x(t;u) = \mathbb{P} \mathcal{Z}_N(t,u)(0) + O\left(\|u_1\|_{L^\infty} |x(t;u)| \right).
    \end{equation}
    The estimate \eqref{ZN=ul+uj2} holds (it is proved in the proof of \eqref{u1uk=uj^2+u1^N}) and gives the conclusion.

    \medskip \noindent \emph{Proof of \cref{p:nilp_u1uk-2} when $j=k$.} Let $M\in \N$ such that $M \geq 4N$. 
    Since $f_\q (0)\notin \mathcal{N}_{j,k}^{M}(f)(0)$, \cref{p:H0+H3} applies, thus \hyp{H3} holds.
    Hence, for $l \in \intset{1,k}$ and $\mathcal{N}=\Bs_{\intset{1,N}} \setminus\{M_{l-1}\}$, one may consider $\mathbb{P}:\R^d \to \R$ giving a component along $f_{M_{l-1}}(0)$ parallel to $\mathcal{N}(f)(0)$. 
    Then $\mathbb{P} \mathcal{Z}_N(t,u)(0)=u_{l}(t)$ and we deduce from \cref{x=ZN+} that
    \begin{equation}
        \mathbb{P} x(t;u) = u_{l}(t) + O\left(\|u_1\|_{L^\infty} |x(t;u)| \right),
    \end{equation}
    which gives the conclusion. 
    
    \medskip \noindent \emph{Proof of \cref{p:nilp_u1uk-2} when $j<k$.} 
    The proof is similar to the two previous ones, fixing first $\pi$ large enough, then $M$ large enough to apply \cref{p:H0+H3}, and concluding thanks to \cref{x=ZN+}.
\end{proof}

\begin{proof}[Proof of \cref{Thm:m=-1} for semi-nilpotent systems] 
    We assume $f_\q(0)\notin \mathcal{N}_{j,k}^{\infty}(f)(0)$. 
    Let $\mathbb{P}:\R^d \to \R$ a component along $f_\q(0)$ parallel to $\mathcal{N}_{j,k}^{\infty}(f)(0)$. 
    By \cref{thm:Key_2},
    \begin{equation} \label{Px=eta+O(ux)}
        \mathbb{P} x(T;u) = \eta_\q(T,u) + O\left( \|u_1\|_{L^\infty} |x(T;u)| \right).
    \end{equation}
    
    \medskip \noindent \textbf{Case $j=k$.} 
    By \cref{p:nilp_u1uk-2} of \cref{p:nilp_u1uk}, \eqref{u1uk=x} holds.
    Under \eqref{u1uk=x}, \eqref{eq:eta_Qjk-new} of \cref{p:eta_Qjk-new} yields 
    \begin{equation}
        \label{eq:eta-xi-u13-x}
        |\eta_\q(T,u) - \xi_\q(T,u)| \leq C \| u_1 \|_{L^\infty}^3 |x(T;u)|.
    \end{equation}
    Combining \eqref{Px=eta+O(ux)} and \eqref{eq:eta-xi-u13-x}, we obtain
    \begin{equation}
        \label{eq:px-xi-u13-x}
        \mathbb{P} x(T;u) = \xi_\q(T,u) + O\left(\|u_1\|_{L^\infty} |x(T;u)|\right),
    \end{equation}
    which gives the conclusion.

    \medskip \noindent \textbf{Case $j < k \leq 2j-1$.} 
    We assume that $f_{W_j}(0) \in \vect\{ f_b(0) ; b \in \Bs_{\N^* \setminus \{2\}} \}$.
    By \cref{p:nilp_u1uk-1} of \cref{p:nilp_u1uk}, estimate \eqref{u1uk=uj^2+x} holds. 
    Using \cref{eq:interp-uj-bis} and \eqref{eq:indices_Uj}, we obtain \cref{u1uk_closed_xi}.
    Working as in the proof of \cref{Thm:Qjk_m_derive} (with fewer terms to estimate), we obtain \cref{eq:Qjk-eta-xi-closed}.
    Incorporating this estimate in \cref{Px=eta+O(ux)}, we obtain finally \eqref{eq:Px-xi=om}, which gives the conclusion.

    \medskip \noindent \textbf{Case $j<k$ and $2j-1<k$.} 
    We assume that $f_{W_j}(0) \in \vect\{ f_b(0) ; b \in \Bs_{\N^* \setminus \{2\}} \setminus \{M_{k-1}\} \}$.
    By \cref{p:nilp_u1uk-2} of \cref{p:nilp_u1uk}, \eqref{u1uk=x} holds.
    As in the $j = k$ case, we obtain \eqref{eq:eta-xi-u13-x} and thus \eqref{eq:px-xi-u13-x} which gives the conclusion.
\end{proof}

\subsubsection{From semi-nilpotent to all systems}
\label{subsc:Drift=seminilp->all}

We now extend \cref{Thm:m=-1} from semi-nilpotent systems to all systems.
Thanks to the framework of \cite[Section 10.5]{BeauchardMarbach2026} (see in particular Lemma 10.17 and Proposition 10.18), it suffices to check that the conditions under which we proved the weak drift in \cref{Thm:m=-1} can be expressed as a finite boolean combination of conditions of the form $(E_i+S_{\llbracket q_i,\infty \llbracket}) \cap (\ker V) = \emptyset$, where $V : \mathcal{L}(X) \to \R^d$ is the linear map $b \mapsto f_b(0)$, for some $E_i \subset \mathcal{L}(X)$ and $q_i \in \N^*$.

And indeed, they can be rephrased as follows:
\begin{itemize}
    \item $f_\q(0) \notin \mathcal{N}^\infty_{j,k}(f)(0)$ is equivalent to $((\q + \vect \mathcal{N}^4_{j,k}) + S_{\llbracket 5,\infty \llbracket}) \cap (\ker V) = \emptyset$,
    \item $f_{W_j}(0)\in \vect\{ f_b(0) ; b \in \Bs_{\N^* \setminus \{2\}} \}$ is equivalent to $((W_j+S_1) + S_{\llbracket 3,\infty \llbracket}) \cap (\ker V) \neq \emptyset$,
    \item $f_{W_j}(0)\in \vect\{ f_b(0) ; b \in \Bs_{\N^* \setminus \{2\}} \setminus \{M_{k-1}\} \}$ is equivalent to $((W_j+H) + S_{\llbracket 3,\infty \llbracket}) \cap (\ker V) \neq \emptyset$, where $H := \vect \{ M_l ; l \neq k-1 \}$.
\end{itemize}

\subsection{Limiting examples}
\label{s:limiting}

We construct a family of examples of controllable systems, which prove the last sentence of \cref{thm:Qjk_intro} concerning the optimality of the threshold $M(j,k,m)$.
Let $1 \le j \leq k$ with $(j,k) \neq (1,1)$ and $R \in \N^*$. 
We consider the system 
\begin{equation} \label{limiting_ex}
    \begin{cases}
        \dot{x}_1=u \\ \dot{x}_2=x_1 \\ \dots \\ \dot{x}_k=x_{k-1} \\ \dot{x}_{k+1}=x_j^2 x_k^2 - x_1^R
    \end{cases}
\end{equation}
on $\R^{k+1}$ which corresponds to $f_0(x) = (0, x_1, \dotsc, x_{k-1},x_j^2 x_k^2 - x_1^R)$ and $f_1(x) = (1, 0, \dotsc, 0)$.

\begin{lemma} \label{p:brackets-limiting}
    In $\Bs$, the only non-vanishing Lie brackets at $0$ of $f_0$ and $f_1$ are: $f_{M_{i-1}}(0)=e_{i}$ for $i \in \intset{1,k}$, $f_\q(0) = C e_{k+1}$ ($C = 4!$ if $j = k$ and $C = 4$ if $j < k$), and $f_{\ad_{X_1}^R(X_0)}(0)=-R! e_{k+1}$.
\end{lemma}

\begin{proof}[Proof of \cref{p:brackets-limiting}]
    Let $B := \{ M_0, \dotsc, M_{k-1}, Q_{j,k,k}, \ad_{X_1}^R(X_0) \}$ which is a finite factor-stable subset of $\Bs$ of size $k+2$.
    Let $\dot{y} = g_0(y) + u g_1(y)$ be the associated canonical system of coordinates of the second kind on $\R^{k+2}$, as defined in \cref{p:canonical}.
    In $\Bs$, the only non-vanishing Lie brackets of $g_0$ and $g_1$ at $0$ are $g_{M_{i-1}}(0) = e_i$ for $i \in \intset{1,k}$, $g_\q(0) = e_{k+1}$ and $g_{\ad_{X_1}^R(X_0)}(0) = e_{k+2}$.

    By integrating \eqref{limiting_ex}, for any $u \in \lone$, one has $x(t;u) = \Psi(y(t;u))$ where,
    \begin{equation} \label{eq:def-theta-limiting}
        \Psi(y) := \left(y_1, \dotsc, y_k, C y_{k+1} - R! y_{k+2} \right).
    \end{equation}
    By an easy argument (see e.g.\ \cite[Theorem 1]{Krener1973} based on \cite[Corollary 8.31]{Lee2013}), this implies that, for every $b \in \Bs$, $f_b(0) = (D\Psi_{\rvert 0}) g_b(0)$.
    By \eqref{eq:def-theta-limiting}, $(D\Psi_{\rvert 0}) e_i = e_i$ for $i \in \intset{1,k}$, $(D\Psi_{\rvert 0}) e_{k+1} = C e_{k+1}$ and $(D\Psi_{\rvert 0}) e_{k+2} = -R! e_{k+1}$.
    This concludes the proof.
\end{proof}

\begin{lemma} \label{p:limiting-R5}
    Assume that $R \geq 5$.
    Then system \eqref{limiting_ex} is $W^{m,\infty}$-STLC for every $m \ge -1$ such that $m < \frac{2(j+k-2)}{R-4} - 1$.
    In particular, since $(j,k) \neq (1,1)$, it is $W^{-1,\infty}$-STLC.
\end{lemma}

\begin{proof}
    The system exhibits a competition between $Q_{j,k,k}$ which is of type $(4,2k+2j-3)$ so (even, odd) and $\ad_{X_1}^R(X_0)$ of type $(R,1)$.
    Hence, when $R$ is odd, one can apply Sussmann's $\mathcal{S}(\theta)$ condition of \cite[Theorem~7.3]{Sussmann1987} if and only if there is a $\theta \in [0,\infty)$ such that
    \begin{equation}
        R+\theta < 4+\theta(2k+2j-3)
        \quad \Leftrightarrow \quad 
        \theta>\frac{R-4}{2(j+k-2)}.
    \end{equation}
    In this case, it is known that the system is $W^{m,\infty}$-STLC for $m \leq \frac{1}{\theta}-1$, as claimed.

    Unfortunately, when $R$ is even, the bracket $\ad_{X_1}^R(X_0)$ is of type (even, odd) so required to be compensated by Sussmann's condition, which thus cannot be applied for this system.
    Let us give a proof of the controllability of \eqref{limiting_ex}, which works for any $R \geq 5$.
    By classical arguments, it suffices to prove that one can move infinitesimally in the directions $\pm e_{k+1}$ starting from $0$.
    More precisely, we use the notion of ``control-tangent direction'' of \cref{def:tangent-direction}.
    
    First, using $u(t) := a \chi^{(k)}(t)$ where $\chi \in \CC^\infty_c((0,T);\R)$ is a non-zero function normalized such that $\int_0^T (\chi \chi^{(k-j)})^2 = 1$ and $0 \leq a \ll 1$, we obtain, since $R \geq 5$, $x(T;u) = a^4 e_{k+1} + O(a^5)$, which proves that $+e_{k+1}$ is a control-tangent direction of order 4 at any regularity.

    Second, to move along $-e_{k+1}$, we use the same dilation as in Sussmann's condition.
    Let $\chi \in \CC^\infty_c((0,T);\R)$ be a non-zero function normalized such that $\int_0^T (\chi^{(k-1)})^R = 1$.
    Using $u(t) := \varepsilon^{1-\theta} \chi^{(k)}(t/\varepsilon^\theta)$ with $0 < \varepsilon \ll 1$ and $\theta \in (0,\infty)$, we obtain on the one hand
    \begin{equation}
        \| u \|_{W^{m,\infty}} = O(\varepsilon^{1-(m+1)\theta})
    \end{equation}
    and on the other hand
    \begin{equation}
        x(\varepsilon^\theta T;u)=\left( \varepsilon^{4+\theta(2k+2j-3)} \int_0^T (\chi \chi^{(k-j)})^2 - \varepsilon^{R+\theta} \right) e_{k+1}.
    \end{equation}
    Hence, $\|u\|_{W^{m,\infty}} = O(\varepsilon^{1-(m+1)\theta}) = o(1)$ when $m+1 < \frac{1}{\theta}$ and $x(\varepsilon^\theta T;u) = - \varepsilon^{R+\theta} e_{k+1} + o(\varepsilon^{R+\theta})$ when $\frac{1}{\theta} < \frac{2(k+j-2)}{R-4}$.
    Hence, when $m+1 < \frac{2(k+j-2)}{R-4}$, there is a choice of $\theta$ such that $-e_{k+1}$ is a control-tangent direction of order $\frac{R+\theta}{1-(m+1)\theta}$ at regularity $W^{m,\infty}$.

    By \cref{p:tgt}, under this condition, \eqref{limiting_ex} is $W^{m,\infty}$-STLC.
\end{proof}

We cannot use $R = 4$ in system \eqref{limiting_ex} because such a system would not be controllable, violating the condition $f_{Q_{1,1,1}}(0) \in S_{\intset{1,3}}(f)(0)$ (see \cref{rk:stefani}).
Instead, we consider the system 
\begin{equation} \label{limiting_ex_4}
    \begin{cases}
        \dot{x}_1=u \\ \dot{x}_2=x_1 \\ \dots \\ \dot{x}_k=x_{k-1} \\ \dot{x}_{k+1}=x_j^2 x_k^2 - x_1^3 x_2
    \end{cases}
\end{equation}
on $\R^{k+1}$ which corresponds to $f_0(x) = (0, x_1, \dotsc, x_{k-1},x_j^2 x_k^2 - x_1^3 x_2)$ and $f_1(x) = (1, 0, \dotsc, 0)$.

\begin{lemma} \label{p:brackets-limiting_4}
    In $\Bs$, the only non-vanishing Lie brackets at $0$ of $f_0$ and $f_1$ are: $f_{M_{i-1}}(0)=e_{i}$ for $i \in \intset{1,k}$, $f_\q(0) = C e_{k+1}$ ($C = 4!$ if $j = k$ and $C = 4$ if $j < k$), and $f_{Q_{1,1,2}}(0)= -6 e_{k+1}$.
\end{lemma}

\begin{proof}
    As in the proof of \cref{p:brackets-limiting}, the proof follows from \cref{p:canonical}.
\end{proof}

\begin{lemma} \label{p:limiting-4}
    System \eqref{limiting_ex_4} is smoothly-STLC.
\end{lemma}

\begin{proof}
    The system exhibits a competition between $Q_{j,k,k}$ which is of type $(4,2k+2j-3)$ so (even, odd) and $Q_{1,1,2}$ of type $(4,2)$ so (even, even).
    Hence, one can apply Sussmann's $\mathcal{S}(\theta)$ condition of \cite[Theorem~7.3]{Sussmann1987} if and only if there is a $\theta \in [0,\infty)$ such that
    \begin{equation}
        4+2\theta < 4+\theta(2k+2j-3)
        \quad \Leftrightarrow \quad 
        0<\theta (2k+2j-5).
    \end{equation}
    Hence, since $(j,k) \neq (1,1)$, $2k+2j-5 \geq 1$ and the condition holds for any $\theta > 0$. 
    In this case, it is known that the system is $W^{m,\infty}$-STLC for $m \leq \frac{1}{\theta}-1$.
    Letting $\theta \to 0$ proves the result.

    Equivalently, let $u^\pm \in \CC^\infty_c((0,1);\R)$ be given by \cref{cor:germ-dual} such that $u^\pm_1(1) = \dotsb = u^\pm_k(1) = 0$ and $\int_0^1 (u^\pm_1)^3 u^\pm_2 = \pm 1$.
    Given $\theta > 0$ and $\varepsilon > 0$, consider $u^{\pm,\varepsilon}(t) := \varepsilon^{1-\theta} u^\pm(t/\varepsilon^\theta)$.
    Then $x(\varepsilon^\theta;u^{\pm,\varepsilon}) = \pm \varepsilon^{4+2\theta} e_{k+1} + O(\varepsilon^{4+(2j+2k-3)\theta})$ and $\|u\|_{W^{m,\infty}} = O(\varepsilon^{1-(m+1)\theta})$.
    Hence, when $(m+1)\theta < 1$, $\pm e_{k+1}$ are control-tangent directions of order $\frac{4+2\theta}{1-(m+1)\theta}$ at regularity $W^{m,\infty}$.
    So the result follows from \cref{p:tgt}.
\end{proof}

Gathering these examples and statements proves the last sentence of \cref{thm:Qjk_intro}.

\begin{proof}[Proof of the optimality of $M(j,k,m)$ in \cref{thm:Qjk_intro}]
    We start with the case $j = k = 1$.
    Then $M(j,k,m) = 3$ for all $m \ge -1$.
    The system $\dot{x}_1 = u$ and $\dot{x}_2 = x_1^3 + x_1^4$ satisfies $f_{Q_{1,1,1}}(0) \notin S_2(f)(0)$ but is smoothly-STLC by \cref{thm:S0-B4}.
    Let $1 \leq j \leq k$ with $(j,k) \neq (1,1)$.
    \begin{itemize}
        \item Case $m = -1$, so $M(j,k,m) = +\infty$.
        And indeed, for any finite $R$ (even large), system \eqref{limiting_ex} is an example which satisfies $f_\q(0) \notin \mathcal{N}^{R-1}_{j,k}(f)(0)$ but is $W^{-1,\infty}$-STLC (by \cref{p:limiting-R5}).

        \item Case $m \in \N$.
        By definition, one has $M(j,k,m) = 3 + \lceil \frac{2(k+j-2)}{m+1} \rceil$ so $M(j,k,m) \geq 4$.
        \begin{itemize}
            \item Case $M(j,k,m) = 4$.
            System \eqref{limiting_ex_4} satisfies $f_\q(0) \notin \mathcal{N}^{3}_{j,k}(f)(0)$ but is smoothly-STLC by \cref{p:limiting-4}.
            So $M = 4$ was optimal.

            \item Case $M(j,k,m) \geq 5$.
            Then $M-4 < \frac{2(k+j-2)}{m+1}$ so $m+1 < \frac{2(j+k-2)}{M-4}$.
            Hence, system \eqref{limiting_ex} with $R = M$  satisfies $f_\q(0) \notin \mathcal{N}^{M-1}_{j,k}(f)(0)$ and is $W^{m,\infty}$-STLC by \cref{p:limiting-R5}.\qedhere
        \end{itemize}
    \end{itemize}
\end{proof}

\newpage

\section{Low-regularity STLC using a lonely bad bracket}
\label{s:low-stlc}

Let $1 \leq j < k$ with $k \geq 2j$.
This section proves \cref{thm:cex}, which provides an example where controllability, at low regularity, is obtained by means of the ``bad'' bracket $Q_{j,k,k}$. 
We start with short heuristic discussions in \cref{s:low-paradox,s:low-unmet}.
In \cref{s:cex-lie}, we prove that the evaluated Lie brackets of system \eqref{syst:cex} are exactly those claimed in \cref{thm:cex}, so that, in particular, $Q_{j,k,k}$ is indeed ``alone on its line''.
Eventually, in \cref{s:cex-easy-directions,s:cex-hard,s:cex-proof}, we prove the main claim of \cref{thm:cex}, i.e.\ that system \eqref{syst:cex} is $W^{k-2j-1,\infty}$-STLC.

\subsection{Paradoxical results?}
\label{s:low-paradox}

The example of system \eqref{syst:cex} shows that, when $k \geq 2j$, the ``bad'' Lie bracket $Q_{j,k,k}$ can actually be used to obtain controllability at low regularity, thanks to the presence of $f_{W_j}(0)$ at a very particular position within the system.
Let us try to explain this apparent paradox.

As explained in \cref{s:approach-Wm}, the quartic obstructions we prove rely on the coercivity of the coordinate of the second kind $\xi_{\bb}(t,u)$ of the concerned ``bad'' Lie bracket $\bb \in \Bs$.
As stated in \cref{thm:Magnus}, the coordinate involved in the expansion of the state $x(t;u)$ is not $\xi_{\bb}(t,u)$, but the one of the pseudo-first kind $\eta_\bb(t,u)$.
In many situations, we are able to bound the difference between $\xi_\bb(t,u)$ and $\eta_\bb(t,u)$ thanks to closed-loop estimates (see e.g.\ \cref{p:eta_Qjk-new} in combination with \cref{p:borduk}) and this difference turns out to be negligible for sufficiently regular controls, using interpolation inequalities.

Unfortunately, it is known that there are situations where one cannot ignore the difference between these two coordinates (see \cite[Section 4.5]{BeauchardMarbach2026}).
The examples studied in this section fall in this context.
More precisely, the presence of $f_{W_j}(0)$ saturates the mentioned closed-loop estimates, and the lack of regularity of the control prevents being able to use interpolation theory to neglect the difference.
This somehow justifies why, despite the coercivity of $\xi_\bb(t,u)$, one can obtain controllability in the directions $\pm f_\bb(0)$, therefore relying on the fact that $\eta_\bb(t,u)$ is onto.

\subsection{An unmet coercivity inequality}
\label{s:low-unmet}

We present the key point concerning the low-regularity controllability of \eqref{syst:cex}.
As already noted, the controllability of the subsystem $(x_1,\dotsc,x_{k+1})$ is guaranteed by Sussmann's $\mathcal{S}(0)$ condition (see also \cref{s:cex-easy-directions}) and the main difficulty is to move in the direction $-e_{k+2}$ (see \cref{s:cex-hard}).

Let $\lambda \in \R^*$.
Given $u \in L^1(0,t)$, straightforward explicit integration of \eqref{syst:cex} proves that
\begin{align}
    \label{eq:cex-xi}
    x_i(t;u) & = u_i(t), \qquad \qquad \qquad \qquad  \text{for } i \in \intset{1,k-1}, \\
    \label{eq:cex-xk}
    x_k(t;u) & = u_k(t) + \lambda \int_0^t u_j^2, \\
    \label{eq:cex-xk1}
    x_{k+1}(t;u) & = 2 \int_0^t u_j^2 u_k + \lambda \left(\int_0^t u_j^2\right)^2, \\
    \label{eq:cex-xk2}
    x_{k+2}(t;u) & = \int_0^t u_j^2 u_k^2 + 2 \lambda \int_0^t u_j^2 \int_0^t u_j^2 u_k + \frac{2}{3} \lambda^2 \left(\int_0^t u_j^2\right)^3.
\end{align}
The first coercive term in \eqref{eq:cex-xk2} is exactly $4 \xi_\q(t,u)$ (see \eqref{xi_Qljknu}).
The second and third terms are of higher order with respect to $u$.
Thus, for sufficiently small controls in sufficiently strong norms, they are negligible with respect to the first term and one concludes that $x_{k+2}(t;u) \geq 0$.
This mechanism is at the heart of the obstruction results obtained in \cref{s:j<=k} concerning $Q_{j,k,k}$.
In particular, \cref{thm:Qjk_intro} implies that system \eqref{syst:cex} is not $W^{k-2j,\infty}$-STLC.

At lower regularity, one cannot derive from the Gagliardo--Nirenberg interpolation inequality that $x_{k+2}(t;u) \geq 0$ anymore.
Nevertheless, by Cauchy--Schwarz and Young inequalities, one can bound the second term of \eqref{eq:cex-xk2} as
\begin{equation} \label{eq:cex-unmet}
    \left| 2 \lambda \int_0^t u_j^2 \int_0^t u_j^2 u_k \right|
    \leq \int_0^t u_j^2 u_k^2 + \lambda^2 \left(\int_0^t u_j^2\right)^3.
\end{equation}
Due to the factor $\frac 2 3$ in \eqref{eq:cex-xk2}, the bound \eqref{eq:cex-unmet} falls just short of proving that $x_{k+2}(t;u) \geq 0$.

Therefore, the key step in the proof of \cref{thm:cex} consists in constructing controls, small in $W^{k-2j-1,\infty}$ (but large in $W^{k-2j,\infty}$) which are close to saturating \eqref{eq:cex-unmet} so that $x_{k+2}(t;u) < 0$, and satisfying moreover $x_1(t;u) = \dotsb = x_{k+1}(t;u) = 0$, to ensure a movement in the direction $-e_{k+2}$.
This requires an almost explicit construction of the controls in \cref{s:cex-hard}.

\subsection{Computation of the Lie brackets}
\label{s:cex-lie}

The system described in \eqref{syst:cex} stems from the vector fields
\begin{equation} \label{eq:cex-f0f1}
    f_0(x) := (0, x_1, x_2, \dotsc, x_{k-2}, x_{k-1} + \lambda x_j^2, 2 x_j^2 x_k, x_j^2 x_k^2 + \lambda x_j^2 x_{k+1}) 
    \quad \text{and} \quad
    f_1(x) := e_1.
\end{equation}
As claimed in \cref{thm:cex}, this system involves the following Lie brackets.

\begin{lemma}
    Let $f_0, f_1$ be given by \eqref{eq:cex-f0f1}.
    In $\Bs$, the only non-vanishing Lie brackets at $0$ of $f_0$ and $f_1$ are: $f_{M_{i-1}}(0) = e_i$ for $i \in \intset{1,k}$, $f_{W_j}(0) = 2 \lambda e_k$, $f_{P_{j,k}}(0) = 4 e_{k+1}$ and $f_\q(0) = 4 e_{k+2}$.
\end{lemma}

\begin{proof}
    Let $B := \{ M_0, \dotsc, M_{k-1}, P_{j,k}, Q_{j,k,k}, W_j \}$ with $|B| = k + 3$.
    Consider the associated canonical system, as defined in \cref{p:canonical}.
    Up to a choice of indices for the coordinates (corresponding to the order used in the list describing $B$), it is given on $\R^{k+3}$ by
    \begin{equation} \label{eq:syst-cex-y}
        \begin{cases}
            \dot{y}_1 = u \\
            \dot{y}_2 = y_1 \\
            \dotsc \\
            \dot{y}_k = y_{k-1} \\
            \dot{y}_{k+1} = \frac 12 y_j^2 y_k \\
            \dot{y}_{k+2} = \frac 14 y_j^2 y_k^2 \\
            \dot{y}_{k+3} = \frac 12 y_j^2
        \end{cases}
    \end{equation}
    which is a scalar-input system stemming from the vector fields
    \begin{equation} \label{eq:cex-g0g1}
        g_0(y) := \left(0, y_1, y_2, \dotsc, y_{k-2}, y_{k-1}, \frac 12 y_j^2 y_k, \frac 14 y_j^2 y_k^2, \frac 12 y_j^2\right) 
        \quad \text{and} \quad
        g_1(y) := e_1.
    \end{equation}
    By \cref{p:canonical}, in $\Bs$, the only non-vanishing Lie brackets of $g_0$ and $g_1$ at $0$ are $g_{M_{i-1}}(0) = e_i$ for $i \in \intset{1,k}$, $g_{P_{j,k}}(0) = e_{k+1}$, $g_\q(0) = e_{k+2}$ and $g_{W_j}(0) = e_{k+3}$.

    Given $u \in L^1(0,t)$, by a straightforward integration of \eqref{eq:syst-cex-y} and recalling \eqref{eq:cex-xk}, \eqref{eq:cex-xk1} and \eqref{eq:cex-xk2}, one observes that $x_i(t;u) = y_i(t;u)$ for $i \in \intset{1,k-1}$ and
    \begin{align}
        x_k(t;u) & = y_k(t;u) + 2 \lambda y_{k+3}(t;u), \\
        x_{k+1}(t;u) & = 4 y_{k+1}(t;u) + 4 \lambda y_{k+3}^2(t;u), \\
        x_{k+2}(t;u) & = 4 y_{k+2}(t;u) + 8 \lambda y_{k+3}(t;u) y_{k+1}(t;u) + \frac{16 \lambda^2}{3} y_{k+3}^3(t;u).
    \end{align}
    Thus, there exists $\Psi : \R^{k+3} \to \R^{k+2}$ such that $x(t;u) = \Psi(y(t;u))$.
    More precisely,
    \begin{equation} \label{eq:def-theta}
        \Psi(y) := \left(y_1, \dotsc, y_{k-1}, y_k + 2 \lambda y_{k+3}, 4 y_{k+1} + 4 \lambda y_{k+3}^2, 4 y_{k+2} + 8 \lambda y_{k+3} y_{k+1} + \frac {16\lambda^2}{3} y_{k+3}^3 \right).
    \end{equation}
    By an easy argument (see e.g.\ \cite[Theorem 1]{Krener1973} based on \cite[Corollary 8.31]{Lee2013}), this implies that, for every $b \in \Bs$, $f_b(0) = (D\Psi_{\rvert 0}) g_b(0)$.
    Moreover, by \eqref{eq:def-theta}, $(D\Psi_{\rvert 0}) e_i = e_i$ for $i \in \intset{1,k}$, $(D\Psi_{\rvert 0}) e_{k+1} = 4 e_{k+1}$, $(D\Psi_{\rvert 0}) e_{k+2} = 4 e_{k+2}$ and $(D\Psi_{\rvert 0}) e_{k+3} = 2 \lambda e_k$.
    This concludes the proof thanks to the brackets of $g_0$ and $g_1$ computed above using \cref{p:canonical}.
\end{proof}

\subsection{Motions in the easy directions}
\label{s:cex-easy-directions}

We prove that one can easily move infinitesimally in the directions $\pm e_i$ for $i \in \intset{1,k}$, $\pm e_{k+1}$ and $+ e_{k+2}$, respectively at the linear order, cubic order, and quartic order.
The Hölder exponents of the target-to-control maps scale accordingly, as illustrated below.
First, classical linear theory (see \cref{lem:S1-tgt}) implies the following result.

\begin{lemma} \label{p:cex-easy-lin}
    Let $T > 0$, $\lambda \in \R$ and $i \in \intset{1,k}$.
    There exists $\bar{u}^i \in \CC^\infty_c((0,T);\R)$ such that the solution to \eqref{syst:cex} satisfies $x(T;a \bar{u}^i) = a e_i + O(a^2)$ for $a \in [-1,1]$.
\end{lemma}

\begin{lemma} \label{p:cex-easy-cub}
    Let $T > 0$, $\lambda \in \R$ and $m \in \N$.
    There exists $C, a^* > 0$ and a map $\mathcal{V}: [-a^*,a^*] \to  \CC^\infty_c((0,T);\R)$, continuous for the $W^{m,\infty}(0,T)$-topology, such that, for every $a \in [-a^*,a^*]$, one has $\|\mathcal{V}(a)\|_{W^{m,\infty}} \leq C|a|$ and $x(T;\mathcal{V}(a))=a^3 e_{k+1} + O(a^4)$, where $x$ is the solution to \eqref{syst:cex}.
\end{lemma}

\begin{proof}
    Let $\chi \in \CC^\infty_c((0,T/2),\R_+)$, scaled such that
    \begin{equation}
        \int_0^{T/2} (\chi^{(k-j)})^2 \chi = \frac{1}{2}.
    \end{equation}
    Let $a^* := 1$ and $a \in [-1,1]$.
    We set, on $(0,T/2)$, 
    \begin{equation}
        \mathcal{V}(a)(t) := a \chi^{(k)}(t)
    \end{equation}
    This corresponds to a control $u$ such that $u_k(t) = a \chi(t)$.
    By \cref{eq:cex-xi,eq:cex-xk,eq:cex-xk1,eq:cex-xk2}, this yields $x_i(T/2) = 0$ for $i \in \intset{1,k-1}$, $x_k(T/2) = c \lambda a^2$, where $c := \int_0^{T/2} (\chi^{(k-j)})^2$, $x_{k+1}(T/2) = a^3 + O(a^4)$ and $x_{k+2}(T/2) = O(a^4)$.
    This control therefore almost achieves the desired goal, except for the remainder term of size $a^2$ on $e_k$.
    
    Then, on $(T/2,T)$, we use the control
    \begin{equation}
        \mathcal{V}(a)(t) := -c\lambda a^2 \bar{u}^k\left(t-\frac{T}{2}\right),
    \end{equation}
    where $\bar{u}^k$ is given by \cref{p:cex-easy-lin} for a time $T/2$.
    Explicit integration of \eqref{syst:cex} then proves that $x(T) = a^3 e_{k+1} + O(a^4)$, where the remainder comes both from the remainder at the end of the first phase, and from the interactions between the non-null initial data at time $T/2$ (of amplitude~$a^2$) and the control (of amplitude~$a^2$).
\end{proof}

\begin{lemma} \label{p:cex-easy-quad}
    Let $T > 0$, $\lambda \in \R$ and $m \in \N$.
    There exists $C, a^* > 0$ and a map $\mathcal{U}^+: [0,a^*] \to  \CC^\infty_c((0,T);\R)$, continuous for the $W^{m,\infty}(0,T)$-topology, such that, for every $a \in [0,a^*]$, $\|\mathcal{U}^+(a)\|_{W^{m,\infty}} \leq C a$ and $x(T;\mathcal{U}^+(a))=a^4 e_{k+2} + O(a^5)$, where $x$ is the solution to \eqref{syst:cex}.
\end{lemma}

\begin{proof}
    Let $\chi \in \CC^\infty_c((0,T/6);\R)$ be a non-null function, rescaled such that
    \begin{equation} \label{eq:T/3-chiscale}
        \int_0^{T/6} (\chi^{(k-j)})^2 \chi^2 = \frac 12.
    \end{equation}
    Let $a^* := 1$ and $a \in [0,1]$. 
    Set
    \begin{equation}
        \mathcal{U}^+(a)(t) := 
        \begin{cases}
            a \chi^{(k)}(t) & \text{for $t \in (0,T/6)$,} \\
            - a \chi^{(k)}(t-T/6) & \text{for $t \in (T/6,T/3)$}.
        \end{cases}
    \end{equation}
    This choice guarantees in particular that, for $u := \mathcal{U}^+(a)$,
    \begin{equation}
        \int_0^{T/3} u_j^2 u_k = 0.
    \end{equation}
    By \cref{eq:cex-xi,eq:cex-xk,eq:cex-xk1,eq:cex-xk2}, one obtains $x_i(T/3) = 0$ for $i \in \intset{1,k-1}$, 
    $x_k(T/3) = c \lambda a^2$, where $c := 2 \int_0^{T/6} (\chi^{(k-j)})^2$, $x_{k+1}(T/3) = c^2 \lambda a^4$ and $x_{k+2}(T/3) = a^4 + 2/3 c^3 \lambda^2 a^6$, using \eqref{eq:T/3-chiscale}.
    
    We now correct the movement along $e_k$ by setting, for $t \in (T/3,2T/3)$,
    \begin{equation}
        \mathcal{U}^+(a)(t) := \left( - c \lambda a^2 - c^2 \lambda^3 a^4 \int_0^{T/3} (\bar{u}^k_j)^2 \right) \bar{u}^k(t-T/3),
    \end{equation}
    where $\bar{u}^k$ is the control of \cref{p:cex-easy-lin} for a time $T/3$.
    Explicit integration of \eqref{syst:cex} proves that $x_i(2T/3) = 0$ for $i \in \intset{1,k-1}$, $x_k(2T/3) = O(a^6)$, $x_{k+1}(2T/3) = c^2 \lambda a^4 + O(a^6)$ and $x_{k+2}(2T/3) = a^4 + O(a^6)$.
    
    Eventually, we correct the movement along $e_{k+1}$ by setting $\mathcal{U}^+(a)(t) := \mathcal{V}(-\sqrt[3]{c^2\lambda a^4})(t-2T/3)$ where $\mathcal{V}$ is the map of \cref{p:cex-easy-cub} with time $T/3$.
    
    Explicit integration of \eqref{syst:cex} concludes that $x(T) = a^4 e_{k+2} + O(a^5)$.
\end{proof}

\subsection{Motion in the difficult direction}
\label{s:cex-hard}

Assuming $\lambda = T = 1$ (without loss of generality, by the scaling arguments used in \cref{s:cex-proof}), we prove that one can use oscillating controls, small in $W^{k-2j-1,\infty}$, to move in the direction $-e_{k+2}$.
We start with a basic Riemann--Lebesgue-type result, proved in \cite[Lemma A.12, Appendix A.7]{BeauchardMarbach2026}.

\begin{lemma} \label{p:RL}
    Let $\theta \in \CC^\infty(\R;\R)$ be a $1$-periodic function, and $h \in \CC^\infty_c((0,1];\R)$.
    The map
    \begin{equation}
        F : \eta \mapsto \int_0^1 h(t) \theta(\eta^{-2}(t-1)) \dd t
    \end{equation}
    admits a $\CC^1$ extension to $\R$ and which satisfies $F(0) = \int_0^1 h \int_0^1 \theta$.
\end{lemma}


\begin{proposition}
    \label{p:cex-hard}
    Let $1 \leq j < k$ with $k \geq 2j$.
    There exists $C, a^* > 0$ and a map $\mathcal{U}^-: [0,a^*] \to  \CC^\infty_c((0,1);\R)$, continuous for the $W^{k-2j-1,\infty}(0,1)$-topology, such that, for $a \in [0,a^*]$, $\| \mathcal{U}^-(a) \|_{W^{k-2j-1,\infty}} \leq Ca$ and $x(1;\mathcal{U}^-(a))=-a^{12(k-j)-6} e_{k+2} + o(a^{12(k-j)-6})$, where $x$ is the solution to \eqref{syst:cex} with $\lambda = 1$.
\end{proposition}

\begin{proof}
    Let $\chi \in \CC^\infty_c((0,1];\R)$ to be chosen later (heuristically, $\chi \approx -1$), with $\chi'$ compactly supported in $(0,1)$.
    Let $\phi \in \CC^\infty(\R;\R)$ be a fixed non-zero $1$-periodic function, vanishing identically in a neighborhood of $0$, and scaled such that 
    \begin{equation} \label{eq:phikj-average}
        \int_0^1 (\phi^{(k-j)})^2 = 1.
    \end{equation}
    For $\eta > 0$ small enough and $\mu \approx \sqrt{2}$ to be chosen later, we consider controls such that
    \begin{equation} \label{eq:uk-eta}
        u_k(t) = \eta^{4(k-j)-2} \chi(t) \left(1 + \mu \eta \phi(\eta^{-2}(t-1))\right).
    \end{equation}
    The oscillation of \eqref{eq:uk-eta} is tuned precisely to ensure that the three terms of \eqref{eq:cex-xk2} have the same amplitude as $\eta \to 0$, in order to saturate \eqref{eq:cex-unmet} and achieve $x_{k+2}(1;u) < 0$ by an appropriate choice of $\chi$ and $\mu$.
    
    By the general Leibniz rule, for every $n \geq 0$,
    \begin{equation} \label{eq:uk-leibniz}
        u_k^{(n)}(t) = \eta^{4(k-j)-2} \chi^{(n)}(t) + \mu \eta^{4(k-j)-1} \sum_{l = 0}^{n} \binom{n}{l} \eta^{-2l} \chi^{(n-l)}(t) \phi^{(l)}(\eta^{-2}(t-1)).
    \end{equation}
    In particular, with $n = k + (k-2j-1)$, \eqref{eq:uk-leibniz} proves that $\|u\|_{W^{k-2j-1,\infty}} \lesssim \eta$.
    Moreover, since $\chi$ and $\phi$ vanish identically near $0$, and since $\chi'$ vanishes identically near $1$, \eqref{eq:uk-leibniz} with $n = 1$ proves that $u_{k-1}$ is compactly supported in $(0,1)$, so $u$ is too.
    
    Heuristically, using \eqref{eq:uk-leibniz} with $n = k-j$,
    \begin{align}
        u_k(t) & = \eta^{4(k-j)-2} \big(\chi(t) + O(\eta)\big), \\
        u_j(t) & = \eta^{2(k-j)-1} \big(\mu \chi(t) \phi^{(k-j)}(\eta^{-2}(t-1)) + O(\eta)\big).
    \end{align}
    Thus, using \cref{p:RL} and \eqref{eq:phikj-average},
    \begin{align}
        \int_0^1 u_j^2 & = \eta^{4(k-j)-2} \left( \mu^2 \int_0^1 \chi^2 + O(\eta) \right), \\
        \int_0^1 u_j^2 u_k & = \eta^{8(k-j)-4} \left( \mu^2 \int_0^1 \chi^3 + O(\eta)\right), \\
        \int_0^1 u_j^2 u_k^2 & = \eta^{12(k-j)-6} \left( \mu^2 \int_0^1 \chi^4 + O(\eta) \right).
    \end{align}
    From \eqref{eq:cex-xi} and the fact that $u_{k-1}$ is compactly supported in $(0,1)$, we conclude that $x_i(1;u) = 0$ for every $i \in \intset{1,k-1}$.
    From \cref{p:RL} and \eqref{eq:phikj-average}, we conclude that
    \begin{align}
        x_k(1;u) & = \eta^{4(k-j)-2} \left( \chi(1) + \mu^2 \int_0^1 \chi^2 + \eta b_0(\mu,\eta) \right), \\
        x_{k+1}(1;u) & = \eta^{8(k-j)-4} \left( 2 \mu^2 \int_0^1 \chi^3 + \mu^4 \left(\int_0^1 \chi^2\right)^2 + \eta b_1(\mu,\eta) \right), \\
        x_{k+2}(1;u) & = \eta^{12(k-j)-6} \left( \mu^2 P_\chi(\mu^2) + \eta b_2(\mu,\eta) \right),
    \end{align}
    where $b_0, b_1, b_2 \in \CC^1(\R^2;\R)$ by \cref{p:RL} and
    \begin{equation} \label{eq:P_chi}
        P_{\chi}(z) := \frac{2}{3} \left( \int_0^1 \chi^2 \right)^3 z^2 + 2 \left( \int_0^1 \chi^2 \right) \left( \int_0^1 \chi^3 \right) z + \int_0^1 \chi^4.
    \end{equation}
    Hence, omitting temporarily the remainder terms $b_0, b_1, b_2$, we wish to find $\chi$ and $\mu$ such that
    \begin{gather}
        \label{eq:cond1}
        \chi(1) + \mu^2 \int_0^1 \chi^2 = 0, \\
        \label{eq:cond2}
        2 \mu^2 \int_0^1 \chi^3 + \mu^4 \left(\int_0^1 \chi^2\right)^2 = 0, \\
        \label{eq:cond3}
        \mu^2 P_\chi(\mu^2) < 0.
    \end{gather}
    When $\chi \equiv -1$, $P_{-1}(z) = \frac{2}{3} z^2 - 2 z +1$ is negative within $\frac 32 \pm \frac {\sqrt{3}}{2}$.
    For $\mu = \sqrt{2}$ and $\chi = -1$, \eqref{eq:cond2} holds.
    Moreover $P_{-1}(2) < 0$ so \eqref{eq:cond3} holds with $\mu = \sqrt{2}$.
    Eventually, \eqref{eq:cond1} could be ensured by choosing $\chi(1) = -2$, which is a ``pointwise'' modification.
    
    More precisely, let $\overline{\chi} \in \CC^\infty_c((0,1);\R)$ with $\|\overline{\chi}+1\|_{L^4} \ll 1$ so that, in particular, $P_{\overline{\chi}}$ is still negative in a neighborhood of $2$.
    Let $\rho \in \CC^\infty([0,1];\R)$, with $\rho'$ compactly supported in $(0,1)$, $\rho(1) = 1$, such that the support of $\rho$ does not intersect the support of $\overline{\chi}$ and such that $\|\rho\|_{L^4} \ll 1$.
    We set $\chi_s(t) := \overline{\chi}(t) + s \rho(t)$, where $s \in \R$ is a real parameter to be chosen.
    The conditions become
    \begin{gather}
        s + \mu^2 \int_0^1 \overline{\chi}^2 + \mu^2 s^2 \int_0^1 \rho^2 = 0, \\
        2 \mu^2 \int_0^1 \overline{\chi}^3 + 2 \mu^2 s^3 \int_0^1 \rho^3 + \mu^4 \left(\int_0^1 \overline{\chi}^2 + s^2 \int_0^1 \rho^2 \right)^2 = 0.
    \end{gather}
    These two relations take the form $F(s,\mu^2,\overline{\chi},\rho)=0$, where $F: (s,z,\overline{\chi},\rho) \in \mathbb{R}^2 \times (L^4(0,1))^2 \to \mathbb{R}^2$ is of class $\CC^1$, satisfies $F(-2,2,-1,0)=0$ and
    \begin{equation}
        \frac{\partial F}{\partial (s,z)}(-2,2,-1,0)=
        \begin{pmatrix}
        1 & 1 \\ 0 & 2
        \end{pmatrix}
    \end{equation}
    is invertible, thus, by the implicit function theorem, if $\| \overline{\chi}+1\|_{L^4}$ and $\|\rho\|_{L^4}$ are small enough, the system above admits a solution $(\bar{s},\bar{\mu})$, close to $(-2,\sqrt{2})$, thus satisfying ${\bar \mu}^2 P_{\chi_{\bar s}}({\bar \mu}^2) < 0$.

    Restoring the perturbations $\eta b_0(\mu,\eta)$ and $\eta b_1(\mu,\eta)$, which are now of the form $\eta c_0(\mu,\eta,s)$ and $\eta c_1(\mu,\eta,s)$ with $c_0, c_1$ being $\CC^1$ functions, the implicit function theorem yields $\CC^1$ functions $s, \mu$ of~$\eta$, for $\eta$ sufficiently small, such that $x_k(1;u) = x_{k+1}(1;u) = 0$ and
    \begin{equation}
        \begin{split}
            x_{k+2}(1;u) & = \eta^{12(k-j)-6} \left( \mu(\eta)^2 P_{\chi_{s(\eta)}}(\mu^2(\eta)) + \eta c_2(\mu(\eta),\eta,s(\eta)) \right) \\
            & = \eta^{12(k-j)-6} \bar{c} + o(\eta^{12(k-j)-6}),
        \end{split}
    \end{equation}
    where $\bar{c} = {\bar \mu}^2 P_{\chi_{\bar s}}({\bar \mu}^2) < 0$.
    
    Therefore, for $a$ small enough, setting $\mathcal{U}^-(a) := u$ where $\eta$ is chosen such that $\eta^{12(k-j)-6} \bar{c} = - a^{12(k-j)-6}$, defines a map which satisfies the claimed properties.
\end{proof}

\begin{remark}
    Even in the most favorable case $(j,k) = (1,2)$, \cref{p:cex-hard} only proves that $-e_{k+2}$ is a control-tangent direction\footnote{Equivalently, it yields a Hölder estimate of order $1/6$ for the target-to-control map.} of order $6$ at regularity $W^{-1,\infty}$ (see \cref{def:tangent-direction}).
    This is weaker than the natural order $4$ obtained in \cref{p:cex-easy-quad} to move in the direction $+e_{k+2}$.
    Along with the fact that we are unable to construct controls which are small in arbitrary strong norms, it is another illustration that the movement in the direction $-e_{k+2}$ is harder.
    
    Moreover, it is easy to prove that this exponent is optimal.
    Indeed, assume by contradiction the existence of $\sigma < 6$ and a map $\mathcal{U}^- : [0,a^*] \to L^1(0,1)$ such that $x(1;\mathcal{U}^-(a)) = - a^\sigma e_{k+2}$ and with an estimate of the form $\| \mathcal{U}^-(a) \|_{W^{-1,\infty}} \leq C a$.
    Let $a \in (0,a^*]$ and $u := \mathcal{U}^-(a)$.
    From the equality $x_{k+1}(1;u) = 0$ and \eqref{eq:cex-xk1}, we derive
    \begin{equation}
        \int_0^1 u_j^2 u_k = O\left(a^4\right).
    \end{equation}
    Thus, by \eqref{eq:cex-xk2},
    \begin{equation}
        -a^\sigma = x_{k+2}(1;u) = \int_0^1 u_j^2 u_k^2 + O\left(a^6\right),
    \end{equation}
    which yields a contradiction as $a \to 0$ when $\sigma < 6$ since $\int_0^1 u_j^2 u_k^2 \geq 0$.
\end{remark}

\subsection{Proof of controllability}
\label{s:cex-proof}

\begin{proof}[Proof of \cref{thm:cex}]
    \step{Scaling argument}
    First, the result obtained in \cref{p:cex-hard} under the assumption $\lambda = T = 1$ easily extends to arbitrary values of these parameters thanks to the following straightforward scaling argument.
    Let $\lambda \in \R^*$ and $T > 0$.
    For $u \in L^1(0,T)$ and $t \in (0,1)$, define 
    \begin{equation}
        v(t) := \lambda T^{1+2j-k} u(Tt).
    \end{equation}
    Let $y$ denote the solution to \eqref{syst:cex} with $\lambda = 1$ and control $v$ and $x$ denote the solution to \eqref{syst:cex} with the arbitrary value of $\lambda$ fixed above and control $u$.
    One easily checks that, for $t \in (0,T)$, $x_i(t;u) = \lambda^{-1} T^{k+i-2j-1} y_i(t/T;v)$ for $i \in \intset{1,k}$, $x_{k+1}(t;u) = \lambda^{-3} T^{4(k-j)-2} y_{k+1}(t/T;v)$ and $x_{k+2}(t;u) = \lambda^{-4} T^{6(k-j)-3} y_{k+2}(t/T;v)$.
    Hence, the motion in the direction $-e_{k+2}$ obtained in \cref{p:cex-hard} for the $y$-system with $\lambda = T = 1$ rescales to a motion in the direction $-e_{k+2}$ for the $x$-system.
    
    \step{Conclusion}
    Gathering \cref{p:cex-easy-lin,p:cex-easy-cub,p:cex-easy-quad} and \cref{p:cex-hard}, we have proved that $\pm e_1, \dotsc, \pm e_{k+2}$ are control-tangent directions (of orders 1, 3, 4 and $12(k-j)-6$) at regularity $W^{k-2j-1,\infty}_0$ as in \cref{def:tangent-direction}.
    \cref{p:tgt} entails that \cref{syst:cex} is $W^{k-2j-1,\infty}_0$-STLC.
\end{proof}

\newpage

\cleardoublepage
\part*{Appendices}
\addcontentsline{toc}{part}{Appendices}

\appendix 

\section{Some folklore results from control theory}

\subsection{Control-tangent directions}
\label{s:tangent}

In control theory, the idea to combine elementary motions by concatenating controls is ubiquitous.
The historical methods for small-time local controllability use time-based tangent vectors (see the recent course \cite[Section 4]{Marbach2026}).
In this paper, we use a control-based approach, where the displacements are not parameterized by time but by the control size.
This viewpoint is rooted in the power-series expansion method (see \cite[Section 8.1]{Coron2007}) and is closer to the ``small-time continuously approximately reachable vectors'' used for PDEs in \cite{Bournissou2024,Gherdaoui2025}.

\bigskip

Our definition is the following.

\begin{definition} 
    \label{def:tangent-direction}
	Let $\mathbf{e} \in \R^d \setminus \{0\}$, $r \geq 1$ and $m \in \{-1\} \cup \N$.
	We say that $\mathbf{e}$ is a \emph{control-tangent direction of order $r$ at regularity $W^{m,\infty}_0$} for \eqref{syst} when, for every $T > 0$, there exists a map $\mathcal{U}^T_\mathbf{e} \in \CC^0([0,1];W^{m,\infty}_0((0,T);\R))$ such that, for $a \in [0,1]$,
	\begin{gather}
        \label{eq:tangent-control-size}
		\|\mathcal{U}^T_\mathbf{e}(a)\|_{W^{m,\infty}} = O(a), \\
        \label{eq:tangent-direction}
		x(T; \mathcal{U}^T_\mathbf{e}(a)) = a^r \mathbf{e} + o(a^{r}).
	\end{gather}
\end{definition}

A first classical example is that linear directions are control-tangent of order $1$ at any regularity (see e.g.\ \cite{Kalman1960,Lasalle1960} for the historical works, and \cite[Theorem 1]{BeauchardMarbach2018} for the smooth regularity).

\begin{lemma} \label{lem:S1-tgt}
    Let $n := \dim S_1(f)(0)$, $i \in \intset{0,n-1}$ and $m \in \{-1\} \cup \N$.
    Then $f_{M_i}(0)$ is a control-tangent direction of order $1$ at regularity $W^{m,\infty}_0$.
    More precisely, for every $T > 0$ there exists $\bar{u}^i \in \CC^\infty_c((0,T);\R)$ such that $x(T;a\bar{u}^i) = a f_{M_i}(0) + O(a^2)$ as $a \to 0$, uniformly for $a \in [-1,1]$.
\end{lemma}

\begin{proof}
    Since $\CC^\infty_c((0,T);\R) \subset W^{m,\infty}_0((0,T);\R)$ for every
    $m \ge 1$, it suffices to construct such a control $\bar{u}^i$.
    The conclusion then holds at every regularity simultaneously.

    Set $A := Df_0(0)$.
    Thus $f_{M_j}(0) = A^j f_1(0)$ for all $j \in \N$ and $S_1(f)(0) = \vect \{A^j f_1(0) : j \in \N\}$.
    Moreover, $(f_{M_0}(0), \dots, f_{M_{n-1}}(0))$ is a basis of $S_1(f)(0)$. 

    Fix $T > 0$.
    For any $t \in [0,T]$, $e^{(T-t)A} f_1(0) \in S_1(f)(0)$.
    Define $K_0, \dots, K_{n-1} : [0,T] \to \R$ as its coordinates in this basis:
    \begin{equation} \label{eq:def-Ki}
        e^{(T-t)A} f_1(0) = \sum_{j=0}^{n-1} K_j(t) f_{M_j}(0).
    \end{equation}
    As $t \to T$,
    \begin{equation}
        e^{(T-t)A} f_1(0) = \sum_{j=0}^{n-1} \frac{(T-t)^j}{j!} f_{M_j}(0) + O((T-t)^n).
    \end{equation}
    Thus $K_j(t) \sim (T-t)^j / j!$ as $t \to T$.
    Hence these functions are independent.
    By density of $\CC^\infty_c((0,T);\R)$ in $L^2((0,T);\R)$, there exists $\bar{u}^i \in \CC^\infty_c((0,T);\R)$ such that $\int_0^T K_j(t) \bar{u}^i(t) \dd t = \delta_{i = j}$.

    For $a \in [-1,1]$, a first-order Taylor expansion yields 
    \begin{equation}
        x(T;a\bar u^i)
        = a \int_0^T e^{(T-t)A} f_1(0) \bar u^i(t) \dd t + O(a^2)
        = a \sum_{j=0}^{n-1} \left( \int_0^T K_j \bar u^i \right) f_{M_j}(0) + O(a^2).
    \end{equation}
    Hence $x(T;a \bar{u}^i) = a f_{M_i}(0) + O(a^2)$, which concludes the proof.
\end{proof}

The main result \cref{p:tgt} requires the following concatenation estimates.

\begin{lemma}
    \label{lem:Gronwall}
    Let $\delta > 0$ and $B_\delta$ be the closed ball of $\R^d$ of radius $\delta$.
    Let $G \in \CC^0([0,\delta] \times B_\delta ; \R^d)$, continuously differentiable with respect to $x$ so that $D_x G \in \CC^0([0,\delta] \times B_\delta; \mathcal{L}(\R^d;\R^d))$.
    
    There exist $C, \rho > 0$ such that, for all $p, q \in B_\rho$, the trajectories to $\dot{x}(t) = G(t,x(t))$, $x(0)=p$ and $\dot{y}(t) = G(t,y(t))$, $y(0)=q$ are well defined on $[0,\rho]$ and satisfy, for all $t \in [0,\rho]$,
    \begin{equation}
        \label{eq:Gronwall}
        | (x(t) - p) - (y(t) - q) | \leq C t |p-q|.
    \end{equation}
\end{lemma}

\begin{proof}
    For $\rho$ small, the trajectories remain in $B_\delta$, and \eqref{eq:Gronwall} follows from Grönwall's lemma.
\end{proof}

\begin{lemma} 
    \label{lem:concatenate-2}
    There exists $C, \rho > 0$ such that, for all $p \in \R^d$ with $|p| \leq \rho$, all $T \in (0,\rho]$ and all $u \in L^1(0,T)$ with $\| u_1 \|_{L^\infty} \leq \rho$,
    \begin{equation}
        | x(T;u,p) - x(T;u,0) - p | \leq C (T + \| u_1 \|_{L^\infty}) |p|.
    \end{equation}
\end{lemma}

\begin{proof}
    The function $x_1(t;u,p):=e^{-u_1(t)f_1} x(t;u,p)$ solves
    $\dot{x}_1=F(x_1,u_1)$ and $x_1(0)=p$, where $F$ is a real-analytic function near the origin (see \cite[Proposition 145]{BeauchardLeBorgneMarbach2023}). 
    Set $G(t,x) := F(x, u_1(t))$.
    By \cref{lem:Gronwall}, there exist $C,\rho > 0$ such that, for all $T \in (0,\rho]$, $u \in L^1(0,T)$ with $\|u_1\|_{L^\infty} \leq \rho$ and $p \in B_\rho$, for all $t \in [0,T]$,
    \begin{equation}
        | x_1(t;u,p) - x_1(t;u,0) - p | \leq C t |p|.
    \end{equation}
    Thus, for all $t \in [0,T]$,
    \begin{equation}
        \begin{split}
            | x(t;u,p)& -x(t;u,0)-p |
              =
            | e^{u_1(t) f_1} x_1(t;u,p) - e^{u_1(t) f_1} x_1(t;u,0) - p |
            \\ 
            & \leq | (e^{u_1(t) f_1} - \operatorname{Id}) (x_1(t;u,p) - x_1(t;u,0)) | + | x_1(t;u,p) - x_1(t;u,0) - p | \\ 
            & \leq M |u_1(t)| | x_1(t;u,p) - x_1(t;u,0)| + C t |p| \\
            & \leq M |u_1(t)| ( C t |p| + |p| ) + C t |p| \\
            & \leq C' ( t + |u_1(t)| ) |p|,
        \end{split}
    \end{equation}
    which concludes the proof.
\end{proof}

\begin{lemma} 
    \label{lem:concatenate-n}
    Let $n \geq 2$.
    There exists $C, \rho > 0$ such that, for all $T_i \in (0,\rho]$ and $u^i \in L^1(0,T_i)$ such that $\| u^i_1 \|_{L^\infty} \leq \rho$,
    \begin{equation}
        \Big| x(T;u,0) - \sum_{i=1}^n x(T_i;u^i,0) \Big|
        \leq C \Big( T + \sum_{i=1}^n \| u^i_1 \|_{L^\infty} \Big) \sum_{i=1}^n |x(T_i;u^i,0) |
    \end{equation}
    where $T := T_1+\dotsb+T_n$ and $u= u^1 \diamond \dotsb \diamond u^n$.
\end{lemma}

\begin{proof}
    Let $\tilde{T} := T_1+\dots+T_{n-1}$ and  $\tilde{u}=u^1 \diamond \dotsb \diamond u^{n-1}$.
    By \cref{lem:concatenate-2},
    \begin{equation}
        \begin{split}
            x(T;u,0) 
            & = x(T_n;u^n, x(\tilde{T};\tilde{u},0)) \\
            & = x(T_n;u^n,0) + x(\tilde{T};\tilde{u},0) + O((T_n + \|u^n_1\|_{L^\infty}) |x(\tilde{T};\tilde{u},0)|)
        \end{split}
    \end{equation}
    and the conclusion follows by induction on $n$.
\end{proof}

The main result is the following consequence of Brouwer's theorem and concatenation.

\begin{proposition} 
    \label{p:tgt}
	Let $m \in \{-1\} \cup \N$.
	Assume that $\pm e_1, \dotsc, \pm e_d$ are control-tangent directions of order $r_1^\pm, \dotsc, r_d^\pm \geq 1$ at regularity $W^{m,\infty}_0$ for \eqref{syst}.
	Then \eqref{syst} is $W^{m,\infty}_0$-STLC. 
	
	More precisely, for every $T > 0$, there exists $C, \delta > 0$ and a map $\mathcal{U} : B_\delta \to W^{m,\infty}_0((0,T);\R)$ such that $x(T;\mathcal{U}(z),0) = z$ for every $z \in B_\delta$, with moreover
	\begin{equation}
        \label{eq:tangent-U-size}
		\| \mathcal{U}(z) \|_{W^{m,\infty}} \leq C |z|^{\frac 1 r}
	\end{equation}
	where $r := \max r_i^\pm$. Moreover, if the maps $\mathcal{U}_{\pm e_i}$ take values in $\CC^\infty$, then so does $\mathcal{U}$.
\end{proposition}

\begin{proof}
    Let $T > 0$.
    For $z = z_1 e_1 + \dotsb + z_d e_d \in \R^d$ set $u^z := u^{z,1} \diamond \dotsb \diamond u^{z,d} \in W^{m,\infty}_0(0,T)$ where
    \begin{equation}
        u^{z,i} := \mathcal{U}^{T/d}_{\sign(z_i) e_i}(a_{z,i})
        \quad \text{where} \quad
        a_{z,i} := |z_i|^{\frac{1}{r_i^{\sign(z_i)}}}.
    \end{equation}
    By \eqref{eq:tangent-direction},
    \begin{equation}
        \label{eq:x-eq-ziei}
        x(T/d;u^{z,i},0) = z_i e_i + o(|z_i|).
    \end{equation}
    Moreover, $u^{z,i} \in W^{m,\infty}_0(0,T/d)$ and $\|u^z\|_{W^{m,\infty}} = O(|z|^{\frac{1}{r}} )$ by \eqref{eq:tangent-control-size}.

    Thus, by \cref{lem:concatenate-2} and \eqref{eq:x-eq-ziei}, there exists $\rho > 0$ such that, when $T \leq \rho$ and $|z| \leq \rho$,
    \begin{equation}
        | x(T;u^z,0) - z | \leq \frac 12 |z|.
    \end{equation}
    Let $z^* \in B_{\rho/2}$.
    The map $F : z \in B_\rho \mapsto z - x(T;u^z,0) + z^*$ is continuous and takes values in $B_\rho$.
    By Brouwer's fixed point theorem, there exists $\bar{z} \in B_\rho$ such that $F(\bar{z}) = \bar{z}$, i.e.\ $x(T;u^{\bar{z}},0) = z^*$.
    Then $|z^* - \bar{z}| = |\bar{z} - x(T;u^{\bar{z}},0)| \leq \frac 12|\bar{z}|$.
    Hence $|\bar{z}| \leq 2 |z^*|$.
    Setting $\mathcal{U}(z^*) := u^{\bar{z}}$, we obtain \eqref{eq:tangent-U-size}.
\end{proof}

\subsection{Dual families of controls}
\label{s:dual}

In this section, we study the notion of dual family introduced in \cref{def:dual}.

\subsubsection{Introduction}

Throughout this paragraph, $\mathcal{B}$ is a Hall set of $\Br(X)$ in which $X_0$ is maximal, for instance $\Bs$.
For $T_1, T_2 > 0$, $u \in L^1(0,T_1)$ and $v \in L^1(0,T_2)$, we denote by $u \diamond v \in L^1(0,T_1+T_2)$ the concatenation of $u$ and $v$, equal to $u$ on $(0,T_1)$ and to $v(\cdot - T_1)$ on $(T_1,T_1+T_2)$.
If $u$ and $v$ are smooth and compactly supported in the corresponding open intervals, so is $u \diamond v$.

We say that a finite subset $B$ of $\mathcal{B} \setminus \{X_0\}$ is \emph{factor-stable} when the factors (see \cref{def:factorization}) of every $b \in B \setminus X$ belong to $B$.
We say that $a \in B$ is \emph{terminal} in $B$ when $a$ is not a factor of any element of $B$.
In particular, any element of $B$ of maximal length is terminal in $B$.

\begin{lemma}[Homogeneity]
    \label{lem:homog}
    Let $T > 0$, $u \in L^1(0,T)$, $\lambda \in \R$ and $\tau > 0$.
    For $t \in (0,\tau T)$, let $u^{\lambda,\tau}(t) := \frac{\lambda}{\tau} u(t/\tau)$.
    Then $\xi_b(\tau T, u^{\lambda,\tau}) = \lambda^{n_1(b)} \tau^{n_0(b)}
    \xi_b(T,u)$ for every $b \in \mathcal{B} \setminus \{X_0\}$.
    
    Consequently, the existence of a dual family for $B \subset \mathcal{B} \setminus \{ X_0 \}$ does not depend on $T$.
\end{lemma}

\begin{proof}
    We proceed by induction on the length $|b| \geq 1$.
    First, from \eqref{xi_X0} and \eqref{xi_S1}, $\xi_{X_0}(t,u) = t$ and $\xi_{X_1}(t,u) = u_1(t)$ satisfy the estimate.
    Now, given $b \in \mathcal{B} \setminus X$, write its maximal factorization as in \cref{def:factorization}.
    By definition of the coordinates of the second kind (see \cite[Definition 62]{BeauchardLeBorgneMarbach2023})
    \begin{equation}
        \label{eq:xib-factorization'}
        \xi_b(t,u) = \int_0^t \frac{\xi_{b_r}^{m_r}}{m_r!} \dotsb \frac{\xi_{b_1}^{m_1}}{m_1!} (s,u) \dd s.
    \end{equation}
    From \eqref{eq:factorization}, we get $|b| = 1 + m_1 |b_1| + \dotsb + m_r |b_r|$.
    Thus $|b_i| < |b|$ and we can apply the induction assumption to $b_i$.
    From \eqref{eq:factorization}, we also get the relations $n_0(b) = 1 + m_1 n_0(b_1) + \dotsb + m_r n_0(b_r)$ and $n_1(b) = m_1 n_1(b_1) + \dotsb + m_r n_1(b_r)$.
    Using the induction hypothesis for the $b_i$ in \eqref{eq:xib-factorization'}, proves the claimed formula for $b$.
    
    If $(u^\sigma_b)$ is a dual family on $[0,T]$ and $T' > 0$, it
    suffices to apply the above with $\tau := T'/T$ and $\lambda := \tau^{-n_0(b) / n_1(b)}$ for each $b \in B$.
\end{proof}

\subsubsection{Enlargement of dual families}

The following statement, which is used repeatedly in the sequel, expresses the additivity under concatenation of the coordinates of the second kind associated with a bracket $c$, provided that the coordinates associated with the iterated factors of $c$ vanish at the intermediate time (a constraint which only involves the first control).

\begin{lemma}[Concatenation]
    \label{lem:terminal}
    Let $S$ be a finite factor-stable subset of $\mathcal{B} \setminus \{X_0\}$ and $b \in \mathcal{B} \setminus (S \cup \{ X_0 \})$ such that the factors of $b$ are in $S$.
    Let $T_1 > 0$ and $u \in L^1(0,T_1)$ be such that $\xi_c(T_1,u) = 0$ for all $c \in S$.
    Then, for all $T_2 > 0$, $v \in L^1(0,T_2)$ and $c \in S \cup \{ b \}$,
    \begin{equation}
        \label{eq:terminal}
        \xi_c(T_1 + T_2, u \diamond v) = \xi_c(T_1,u) + \xi_c(T_2, v).
    \end{equation}
\end{lemma}

\begin{proof}
	Start by proving that, for all $t \in [0,T_2]$, and $c \in S$, $\xi_c(T_1 + t, u \diamond v) = \xi_c(t, v)$.
	One can proceed by induction on $|c|$.
	Then the conclusion follows by \eqref{eq:xib-factorization}.
\end{proof}

\begin{lemma}[Adjoining new brackets]
    \label{lem:dual-new}
    Let $B$ be a finite factor-stable subset of $\mathcal{B} \setminus \{X_0\}$ having a dual family, and let $H$ be a finite subset of $\mathcal{B} \setminus (B \cup \{X_0\})$ such that the factors of every $h \in H$ belong to $B$.
    Assume that, for some $T>0$, there exists a family $(v^\pm_h)_{h \in H, \sigma
    \in \{\pm1\}}$ of $\CC^\infty_c((0,T);\R)$ such that
    \begin{equation}
        \label{eq:dual-new}
        \xi_{h'}(T, v^\pm_h) = \pm \delta_{h = h'}
        \quad (h,h' \in H),
        \qquad
        \xi_b(T, v^\pm_h) = 0
        \quad (b \in B).
    \end{equation}
    Then $B \cup H$ is factor-stable and has a dual family.
\end{lemma}

\begin{proof}
    Factor-stability of $B \cup H$ is clear.
    
    We may assume that $B$ has a dual family $(u^\pm_b)_{b \in B}$ on $[0,T]$, by
    \cref{lem:homog}.

    Let $n := |H|$.
    Given $z \in \R^H$, let $V(z)$ be the concatenation, in any order, of the $n$ controls $|z_h|^{1/n_1(h)} v^{\sign z_h}_h$ for $h \in H$.
    Applying \cref{lem:terminal}, $n-1$ times, and using \eqref{eq:dual-new} together with \cref{lem:homog}, we get
    \begin{align}
    	\forall c \in B, \quad &\xi_c(nT, V(z)) = 0, \\
    	\forall h \in H, \quad &\xi_h(nT, V(z)) = z_h.
    \end{align}
   	Let $b \in B$, $\sigma \in \{\pm1\}$ and $z^\sigma_b := (\xi_h(T,
    u^\sigma_b))_{h \in H}$, and set $\widetilde u^\sigma_b := V(-z^\sigma_b) \diamond
    u^\sigma_b$. 
    By \cref{lem:terminal} again,
    \begin{align}
       	\forall c \in B, \quad & \xi_c((n+1)T, \widetilde u^\sigma_b) = \sigma \delta_{b = c}, \\
       	\forall h \in H, \quad & \xi_h((n+1)T, \widetilde u^\sigma_b) = - z^\sigma_{b,h} + z^\sigma_{b,h} = 0.
    \end{align}
    Hence $(\widetilde u^{\pm}_b)_{b \in B}$, completed by the family $(v^{\pm}_h)_{h \in H}$
    (extended by $0$ on $(T,(n+1)T)$), is a dual family for $B \cup H$ on $[0,(n+1)T]$.
\end{proof}

\begin{lemma}[Adjoining one integration by $X_0$]
    \label{lem:dual-b0}
    Let $B$ be a finite factor-stable subset of $\mathcal{B} \setminus \{X_0\}$ having a dual
    family and let $a$ be terminal in $B$.
    Then $B \cup \{a0\}$ is factor-stable and has a dual family.
\end{lemma}

\begin{proof}
    Since $X_0$ is maximal, $a0 = (a,X_0) \in \mathcal{B}$ and its unique factor is $a$.
    By \cref{lem:dual-new} applied with $H := \{a0\}$, it suffices to construct, for some
    $T' > 0$, controls $\mathfrak{u}^\pm \in \CC^\infty_c((0,T');\R)$ such that
    $\xi_b(T',\mathfrak{u}^\pm) = 0$ for $b \in B$ and
    $\xi_{a0}(T',\mathfrak{u}^\pm) = \pm 1$.

    Let $(u^{\pm}_b)_{b \in B}$ be a dual family for $B$ on $[0,T]$, let $L > 0$, let $0_L$
    denote the null control on $(0,L)$ and let $w^\pm_L := u^\pm_a \diamond 0_L \diamond
    u^\mp_a$, a control on $(0,2T+L)$.
    Applying \cref{lem:terminal} twice, we get $\xi_b(t,w^\pm_L) = \pm \delta_{a=b}$ for $b \in B$ and $t \in [T,T+L]$, and, for $t \in [T+L,2T+L]$,
    \begin{equation*}
        \xi_a(t,w^\pm_L) = \pm 1 + \xi_a(t-T-L,u^\mp_a),
    \end{equation*}
    so that in particular $\xi_b(2T+L,w^\pm_L) = 0$ for every $b \in B$.
    Since $\dot{\xi}_{a0} = \xi_a$, integrating the above over the three segments gives
    \begin{equation*}
        \xi_{a0}(2T+L, w^\pm_L) = \pm L + C_\pm,
        \qquad
        C_\pm := \int_0^T \xi_a(t,u^\pm_a) \dd t \pm T + \int_0^T \xi_a(t,u^\mp_a) \dd
        t,
    \end{equation*}
    where $C_\pm$ does not depend on $L$.
    For $L$ large enough, $\pm \xi_{a0}(2T+L,w^\pm_L) > 0$, while all the coordinates indexed by $B$ vanish.
    It remains to multiply $w^\pm_L$ by a suitable positive amplitude, using \cref{lem:homog}.
\end{proof}

As an example, we recover the following classical result.

\begin{lemma}
    \label{lem:linear}
    Let $d \in \N^*$.
    The set $\{ M_\nu \mid \nu < d \}$ is factor-stable and has a dual family.
\end{lemma}

\begin{proof}
    Since $M_0 = X_1 \in X$ and $M_\nu = \ad_{M_{\nu-1}}(X_0) = M_{\nu-1} 0$ for $\nu \geq 1$, this set is factor-stable.
    For $d = 1$, and $T > 0$, let $u^\pm := \pm \frac{1}{T} \chi(\cdot / T)$ where $\chi \in \CC^\infty_c((0,1);\R)$ with $\int \chi = 1$.
    Then $(u^\pm)$ is a dual family for the singleton $\{ X_1 \}$.
    Applying \cref{lem:dual-b0} proves the claim by induction on $d$.
\end{proof}

\subsubsection{Continuous right inverses}

\begin{definition}
    \label{def:C0-inverse}
    Let $G$ be a finite subset of $\Bs$.
    We say that $G$ has a \emph{continuous right inverse} when, for all $T > 0$, there exists a map $R_G : \R^G \to \CC^\infty_c((0,T);\R)$ such that
    \begin{equation}
        \forall z \in \R^G, \quad
        \forall b \in G, \quad
        \xi_b(T, R_G(z)) = z_b
    \end{equation}
    and such that $R_G(0) = 0$, and for each $k\in\N$, $R_G \in \CC^0(\R^G;W^{k,\infty}_0((0,T);\R))$.
\end{definition}

\begin{proposition}
    \label{p:dual-inverse}
    Let $B$ be a finite factor-stable subset of $\mathcal{B} \setminus \{X_0\}$ having a dual
    family.
    Then~$B$ has a continuous right inverse in the sense of \cref{def:C0-inverse}.
\end{proposition}

\begin{proof}
    We proceed by induction on $|B|$, the case $B = \varnothing$ being trivial.
    Let $T > 0$, let $a \in B$ be of maximal length, so that $a$ is terminal in $B$,
    and let $B' := B \setminus \{a\}$, which is factor-stable and inherits a dual family.
    By the induction hypothesis, there exists $R' : \R^{B'} \to
    \CC^\infty_c((0,T);\R)$ such that $R'(0) = 0$, $\xi_b(T,R'(z')) = z'_b$ for $b
    \in B'$, and which is continuous with values in every $W^{k,\infty}$.

	Let $(u^\sigma_b)$ be a dual family for $B$ on $[0,T]$ and,
    for $s \in \R$, let $U(s) := |s|^{1/n_1(a)} u_a^{\sign s}$, so that $\xi_b(T,U(s)) = s \delta_{a = b}$ for $b \in B$.
    Finally, for $z \in \R^B$, set
    \begin{equation}
        R_B(z) := U \big(z_a - \gamma(z')\big) \diamond R'(z')
        \quad \text{where} \quad 
        z' := (z_b)_{b \in B'} 
        \quad \text{and} \quad 
        \gamma(z') := \xi_a(T, R'(z')).
    \end{equation}
    Then $R_B(z) \in \CC^\infty_c((0,2T);\R)$.
    By \cref{lem:terminal} with $S := B'$, one has, for all $c \in B$,
    $\xi_c(2T,R_B(z)) = (z_a - \gamma(z')) \delta_{a=c} + \xi_c(T,R'(z')) = z_c$.

	All the steps of the construction are continuous with values in every $W^{k,\infty}_0$ and $R_B(0) = 0$.
\end{proof}

\subsection{Sussmann dilations}

We present a general theorem based on control-dilations introduced by Sussmann in \cite{Sussmann1987}, of the form $u^\varepsilon(t) := \varepsilon^{1-\theta} u(t / \varepsilon^\theta)$ for a well-chosen $\theta > 0$, reference control $u$ and $0 < \varepsilon \ll 1$.

\begin{theorem}
	\label{thm:Stheta-Holder}
    Let $G$ be a finite factor-stable subset of $\Bs$ with a continuous right inverse (see \cref{def:C0-inverse}).
    Assume that $\mathcal{L}(f)(0) = \R^d$ and that there exists $\theta \in (0,\infty)$ such that
    \begin{equation}
        \label{eq:omega-compensation}
        \forall b \in \Bs \setminus G, \quad 
        f_b(0) \in \vect \{ f_a(0) \mid \omega(a) < \omega(b) \}
    \end{equation}
    where $\omega(b) := n_1(b) + \theta n_0(b)$.
    Then \eqref{syst} is $W^{m,\infty}_0$-STLC for any $-1 \leq m < \frac{1}{\theta}-1$.
\end{theorem}

We also use our expansion \cref{thm:Magnus} of the state.
Let us underline two subtleties:
\begin{itemize}
    \item \cref{def:C0-inverse} concerns the coordinates of the second kind $\xi$, whereas \cref{thm:Magnus} is based on the coordinates of the pseudo-first kind $\eta$.
    We use \cref{p:etab-xib-XI} to estimate the difference between both.
    The assumption that $G$ is factor-stable is used to bound $\eta_b - \xi_b$.

    \item Typical proofs rely either on Brouwer's fixed point theorem (see e.g.\ \cite[Section 8.1]{Coron2007}) or on the $\CC^1$ inverse function theorem using the notion of \emph{normal control} (see \cite{Sussmann1987}).
    Here, we directly use a topological degree argument (see \cref{lem:anisotropic-perturbation}).
\end{itemize}

\begin{proof}
    Let $\theta$ be given by the assumption and $m \in [-1, \frac{1}{\theta} - 1)$.
    We prove that \eqref{syst} is $W^{m,\infty}_0$-STLC.
	
	\medskip \noindent \emph{Step 1: An adapted basis.}
	The rank condition $\mathcal{L}(f_0,f_1)(0) = \R^d$ and \eqref{eq:omega-compensation} yield $G(f)(0) = \R^d$.
	Enumerating $G$ by nondecreasing $\omega$ and greedily extracting a linearly independent family, we obtain $b_1, \dotsc, b_d \in G$ such that the vectors $e_i := f_{b_i}(0)$ form a basis of $\R^d$, $\omega_1 \le \dotsb \le \omega_d$ where $\omega_i := \omega(b_i)$, and writing $f_b(0) = \sum_i \alpha_{b,i} e_i$ for $b \in \Bs$,
	\begin{equation} 
        \label{eq:triangular}
		\alpha_{b,i} \neq 0
		\quad \Rightarrow \quad
		\begin{cases}
			\omega_i \le \omega(b) & \text{if } b \in G, \\
			\omega_i < \omega(b) & \text{otherwise}.
		\end{cases}
	\end{equation}
	Indeed, \eqref{eq:triangular} holds by construction when $b \in G$ and follows from \eqref{eq:omega-compensation} when $b \in \Bs \setminus G$.

    Let $M \in \N$ such that $M > \max_{b \in G} \omega(b) \geq \max_{b \in G} n_1(b)$.
	The set $A := \{ b \in \Bs ; \omega(b) < M \}$ is finite.
    There exists $\kappa > 0$ such that, for all $b \in \Bs$ and $i \in \intset{1,d}$,
    \begin{equation} \label{eq:weights}
        \alpha_{b,i} \neq 0 \quad \text{and} \quad b \notin G
        \quad \Rightarrow \quad
        \omega(b) \geq \omega_i + \kappa.
    \end{equation}
    Indeed, $\min \{ \omega(b) - \omega_i \mid b \in A \setminus G, \alpha_{b,i} \neq 0 \} > 0$ since $A$ is finite.
    When $b \notin A$, $\omega(b) - \omega_i > M - \omega_d$.
    
	\medskip \noindent \emph{Step 2: A family of controls.}
	For $z \in \R^d$ with $|z| \leq 1$, let $u_z \in \CC^\infty_c((0,1);\R)$ be the control given by \cref{def:C0-inverse} for $T = 1$, prescribing $\xi_{b_i}(1,u_z) = z_i$ for $i \in \intset{1,d}$ and $\xi_b(1,u_z) = 0$ for the other $b \in G$.
	The map $z \mapsto u_z$ is continuous with values in $W^{m,\infty}_0((0,1);\R)$.
	In particular, there exists $\varrho \in \CC^0(\R^d;\R_+)$ with $\varrho(z) \to 0$ as $z \to 0$ such that $\| u_z \|_{W^{m,\infty}} \leq \varrho(z)$.
	By \cref{lem:xi-naive-inf},
    \begin{equation} \label{eq:xi-uz-crude}
        \forall b \in \Bs, \quad \forall |z| \leq 1, \quad
        |\xi_b(1,u_z)| \leq \varrho(z)^{n_1(b)}.
    \end{equation}
    We claim that there exists $C \ge 1$ such that
    \begin{align}
        \label{eq:eta-crude}
        \forall b \in A, \quad \forall |z| \leq 1, \quad
        & |\eta_b(1,u_z)| \leq C \varrho(z)^{n_1(b)}, \\
        \label{eq:eta-xi-good}
        \forall b \in A \cap G, \quad \forall |z| \leq 1, \quad
        & |\eta_b(1,u_z) - \xi_b(1,u_z)| \leq C |z| \varrho(z).
    \end{align}
    Both follow from \cref{p:etab-xib-XI}, applied to the finitely many $b \in A$.
    Let $b \in A$, $q \geq 2$ and $c_1, \dotsc, c_q \in \Bs \setminus \{ X_0 \}$ with $b \in \supp_{\Bs} \mathcal{F}(c_1,\dotsc,c_q)$.
    Hence $n_1(c_1) + \dotsb + n_1(c_q) = n_1(b)$.
    Estimating each factor by \eqref{eq:xi-uz-crude} gives
    $|\xi_{c_1} \dotsm \xi_{c_q}|(1,u_z) \leq \varrho(z)^{n_1(c_1) + \dotsb + n_1(c_q)}$, whence \eqref{eq:eta-crude}.
    
    Since $G$ is factor-stable, by \cite[Lemma 5.32]{BeauchardLeBorgneMarbach2026}, $\vect (\Bs \setminus G)$ is a Lie subalgebra.
    Thus, if $b \in G \cap \supp_{\Bs} \Lie \{ c_1, \dotsc, c_q \}$, there exists $l$ such that $c_l \in G$.
    For such an $l$, $|\xi_{c_l}(1,u_z)| \leq |z|$ by construction of $u_z$.
    Estimating the other factors by \eqref{eq:xi-uz-crude} yields $|\xi_{c_1} \dotsm \xi_{c_q}|(1,u_z) \leq C |z| \varrho(z)$, whence \eqref{eq:eta-xi-good}.
	
	\medskip \noindent \emph{Step 3: Dilation and expansion of the state.}
	Let $T > 0$. For $0 < \varepsilon \ll 1$ and $|z| \leq 1$, let $u_z^\varepsilon$ be supported in $(T-\varepsilon^\theta,T)$ and defined there by
	\begin{equation}
		u_z^{\varepsilon}(t) := \varepsilon^{1-\theta} u_z \left( (t-T+\varepsilon^\theta)/\varepsilon^\theta \right).
	\end{equation}
	Since $f_0(0) = 0$, the state stays at $0$ on $(0,T-\varepsilon^\theta)$, so it suffices to study the last subinterval. 
    Moreover
	\begin{equation} \label{eq:uz-eps-size}
		\|u_z^\varepsilon\|_{W^{m,\infty}}
		\leq C \varepsilon^{1-(m+1)\theta} \varrho(z),
	\end{equation}
	which tends to $0$ as $\varepsilon \to 0$, uniformly in $|z| \le 1$, because $\theta < 1/(m+1)$.

	By \cref{thm:Magnus} and homogeneity of the coordinates, for $|z| \le 1$ and $\varepsilon \in (0,1]$,
    \begin{equation}
        x(T;u^\varepsilon_z) = \sum_{b \in \Bs_{\intset{1,M}}} \varepsilon^{\omega(b)} \eta_b(1,u_z) f_b(0) + O\left( \varepsilon^M + |x(T;u^\varepsilon_z)|^{1+\frac{1}{M}} \right),
    \end{equation}
    where the absolutely convergent expansion of $\mathcal{Z}_M$ splits according to whether $b \in A$ or not, the latter terms being $O(\varepsilon^M)$.
	Let $F_\varepsilon(z) := x(T;u^\varepsilon_z)$, let $L_\varepsilon$ be the linear map
	$e_i \mapsto \varepsilon^{\omega_i} e_i$ and let
	\begin{equation} \label{eq:Gamma-eps}
		(\Gamma_\varepsilon(z))_i
		:= \sum_{b \in A} \varepsilon^{\omega(b)-\omega_i} \alpha_{b,i} \eta_b(1,u_z).
	\end{equation}
    By \cref{thm:Magnus}, the sum \eqref{eq:ZM=eta} defining $\mathcal{Z}_M$ is absolutely convergent.
    Thus, uniformly for $|z| \le 1$ and $\varepsilon \in (0,1]$, we have
	\begin{equation} \label{eq:x-eps-Magnus}
		| F_\varepsilon(z) - L_\varepsilon \Gamma_\varepsilon(z) |
        = O\left(\varepsilon^M + | F_\varepsilon(z) |^{1+\frac{1}{M}}\right).
	\end{equation}
	
	\medskip \noindent \emph{Step 4: $\Gamma_\varepsilon$ is close to the identity.}
	We claim that there exists $C_\Gamma > 0$ such that
	\begin{equation} \label{eq:Gamma-eps-z}
		\forall \varepsilon \in [0,1], \quad \forall |z| \leq 1, \quad
		|\Gamma_\varepsilon(z) - z| \leq C_\Gamma \left( |z| \varrho(z) + \varepsilon^\kappa \right).
	\end{equation}
	Indeed, fix $i$ and split the finite sum \eqref{eq:Gamma-eps} depending on whether $b \in G$.
	By \eqref{eq:weights}, terms with $b \notin G$ carry a factor $\varepsilon^{\omega(b)-\omega_i} \leq \varepsilon^\kappa$ and bounded coordinates by Step 2.
	By~\eqref{eq:eta-xi-good}, when $b \in G$, $\eta_b(1,u_z) = \xi_b(1,u_z) + O(|z| \varrho(z))$, which equals $z_i + O(|z| \varrho(z))$ if $b = b_i$ and $O(|z| \varrho(z))$ otherwise.

    \medskip \noindent \emph{Step 5: Conclusion.}
    Fix $r \in (0,1]$ such that $4 C_\Gamma \sup_{[0,r]} \varrho \leq 1$ and let $\varepsilon$ be such that $4 C_\Gamma \varepsilon^\kappa \leq r$.
    By \eqref{eq:Gamma-eps-z}, for $|z| = r$,
    \begin{equation} \label{eq:gamma-z-r2}
        |\Gamma_\varepsilon(z) - z| \leq \frac r2,
        \quad \text{hence} \quad
        |L_\varepsilon \Gamma_\varepsilon(z)| \geq \frac r2 \varepsilon^{\omega_d}.
    \end{equation}
    By \eqref{eq:uz-eps-size} and \cref{p:small-state}, $F_\varepsilon(z) = O(\varepsilon)$ uniformly for $|z| \leq r$, so \eqref{eq:x-eps-Magnus} gives
    \begin{equation}
		| F_\varepsilon(z) - L_\varepsilon \Gamma_\varepsilon(z) |
        = O\left(\varepsilon^M + \varepsilon^{\frac1M} | F_\varepsilon(z) |\right)
        = O\left(\varepsilon^M + \varepsilon^{\frac1M} | L_\varepsilon \Gamma_\varepsilon (z) | + \varepsilon^{\frac1M} | F_\varepsilon(z) - L_\varepsilon \Gamma_\varepsilon (z) | \right).
    \end{equation}
    Absorbing the third term in the left-hand side and using \eqref{eq:gamma-z-r2} gives
    \begin{equation}
        |F_\varepsilon(z) - L_\varepsilon \Gamma_\varepsilon(z)|
        \leq C \varepsilon^M + C \varepsilon^{\frac1M} | L_\varepsilon \Gamma_\varepsilon (z) |
        \leq \frac{2C}{r} \varepsilon^{M-\omega_d} |L_\varepsilon \Gamma_\varepsilon(z)| + C \varepsilon^{\frac1M} | L_\varepsilon \Gamma_\varepsilon (z) |
        \leq \frac12 |L_\varepsilon \Gamma_\varepsilon(z)|
    \end{equation}
    for $|z| = r$ and $\varepsilon$ small enough, since $\omega_d < M$.
    
    A topological degree argument (see \cref{lem:anisotropic-perturbation}) proves that $0 \in \operatorname{int} F_\varepsilon(B(0,r))$.
    By \eqref{eq:uz-eps-size} and $m < \frac 1 \theta - 1$, $\|u^\varepsilon_z\|_{W^{m,\infty}} \to 0$ as $\varepsilon \to 0$ uniformly for $|z| \leq r$.
    
    This proves that \eqref{syst} is $W^{m,\infty}_0$-STLC.
\end{proof}

The presence of the dilation $L_\varepsilon$ above makes the application of Brouwer's fixed-point theorem difficult. 
As a replacement, we used the following consequence of topological degree theory (see e.g.\ \cite[Appendix B]{Coron2007} for an introduction).

\begin{lemma}[Anisotropic perturbation]
    \label{lem:anisotropic-perturbation}
    Let $L \in \GL(\R^d)$, $r > 0$ and $\Gamma, F : \overline{B(0,r)} \to \R^d$.
    Assume that $\Gamma, F$ are continuous and that, for all $z \in \partial B(0,r)$,
    \begin{align}
        \label{eq:lem-hyp-1}
        |\Gamma(z)-z|& \leq\frac12|z|, \\
        \label{eq:lem-hyp-2}
        \bigl|F(z)-L\Gamma(z)\bigr|&\leq\frac12\bigl|L\Gamma(z)\bigr| .
    \end{align}
    Then $B(0,\rho) \subset F(B(0,r))$, where $\rho := r / (4 \| L^{-1} \|)$.
\end{lemma}

\begin{proof}
    For $t \in [0,1]$, consider the homotopies:
    \begin{align}
        H_t(z) & := L \Gamma(z) + t (F(z) - L \Gamma(z)), \\
        K_t(z) & := L z + t (L \Gamma(z) - L z).
    \end{align}
    For all $z \in \partial B(0,r)$, $|K_t(z)| \geq 2 \rho$ by \eqref{eq:lem-hyp-1} and $|H_t(z)| \geq \rho$ by \eqref{eq:lem-hyp-1} and \eqref{eq:lem-hyp-2}.

    Fix $y \in B(0,\rho)$.
    Since $y \notin H_t(\partial B(0,r))$ and $y \notin K_t(\partial B(0,r))$, 
    \begin{equation}
        \deg (F, B(0,r), y) = \deg (L \Gamma, B(0,r), y) = \deg (L, B(0,r), y).
    \end{equation}
    Since $y \in L(B(0,r))$ and $L$ is invertible, $\deg (L, B(0,r), y) = \sign (\det L) \neq 0$.
    
    Therefore $\deg (F, B(0,r), y) \neq 0$ and $y \in F(B(0,r))$.
\end{proof}

\newpage 

\section{Some explicit quartic elements and their nature}
\label{s:tables}

\begin{table}[!ht]
    \centering
    \setlength{\extrarowheight}{10pt}
    \setlength{\tabcolsep}{10pt}
    \begin{tabular}{|m{2cm}|m{6cm}|m{2cm}|}
        \hline
        $b$ & $\xi_{b}(t,u)$ & Nature \\[10pt]
        \hline
        $Q_{1,1,1,0}$ & $\displaystyle\int_0^t \frac{u_1(s)^4}{4!} \dd s$ & bad
        \\[10pt]
        \hline
    \end{tabular}
    \caption{Elements of $\Bs_{4,1}$}
    \label{tab:S41}
\end{table}

\begin{table}[!ht]
    \centering
    \setlength{\extrarowheight}{10pt}
    \setlength{\tabcolsep}{10pt}
    \begin{tabular}{|m{2cm}|m{6cm}|m{2cm}|}
        \hline
        $b$ & $\xi_{b}(t,u)$ & Nature \\[10pt]
        \hline
        $Q_{1,1,2,0}$ & $\displaystyle\int_0^t \frac{u_1(s)^3}{3!} u_2(s) \dd s$ & good
        \\[10pt]
        \hline
        $Q_{1,1,1,1}$ & $\displaystyle\int_0^t (t-s) \frac{u_1(s)^4}{4!} \dd s$ & minor
        \\[10pt]
        \hline
    \end{tabular}
    \caption{Elements of $\Bs_{4,2}$}
    \label{tab:S42}
\end{table}

\begin{table}[!ht]
    \centering
    \setlength{\extrarowheight}{10pt}
    \setlength{\tabcolsep}{10pt}
    \begin{tabular}{|m{2cm}|m{6cm}|m{2cm}|}
        \hline
        $b$ & $\xi_{b}(t,u)$ & Nature \\[10pt]
        \hline
        $Q_{1,1,2,1}$ & $\displaystyle\int_0^t (t-s) \frac{u_1(s)^3}{3!} u_2(s) \dd s$ & good
        \\[10pt]
        \hline
        $Q_{1,1,3,0}$ & $\displaystyle\int_0^t \frac{u_1(s)^3}{3!} u_3(s) \dd s$ & good
        \\[10pt]
        \hline
        $Q_{1,2,2,0}$ & $\displaystyle\int_0^t \frac{u_1(s)^2}{2} \frac{u_2(s)^2}{2} \dd s$ & bad
        \\[10pt]
        \hline 
        $Q_{1,1,1,2}$ & $\displaystyle \int_0^t \frac{(t-s)^2}{2!} \frac{u_1(s)^4}{4!} \dd s$ & minor \\[10pt]
        \hline
        $Q^\flat_{1,0,0}$ & 
        $\displaystyle \int_0^t \frac 1 2 \left(\int_0^s \frac{u_1^2}{2} \right)^2 \dd s$ & minor 
        \\[10pt]
        \hline
    \end{tabular}
    \caption{Elements of $\Bs_{4,3}$}
    \label{tab:S43}
\end{table}

\begin{table}[!ht]
    \centering
    \setlength{\extrarowheight}{10pt}
    \setlength{\tabcolsep}{10pt}
    \begin{tabular}{|m{2cm}|m{6cm}|m{2cm}|}
        \hline
        $b$ & $\xi_{b}(t,u)$ & Nature \\[10pt]
        \hline
        $Q_{1,1,2,2}$ & $\displaystyle \int_0^t \frac{(t-s)^2}{2!} \frac{u_1^3(s)}{3!} u_2(s) \dd s$ & good
        \\[10pt]
        \hline
        $Q_{1,1,3,1} $ & $\displaystyle \int_0^t (t-s) \frac{u_1^3(s)}{3!} u_3(s) \dd s$ & good
        \\[10pt]
        \hline
        $Q_{1,1,4,0} $ & $\displaystyle \int_0^t \frac{u_1^3(s)}{3!} u_4(s) \dd s$ & good
        \\[10pt]
        \hline
        $Q_{1,2,3,0}$ & $\displaystyle \int_0^t \frac{u_1^2(s)}{2} u_2(s) u_3(s) \dd s$ & good
        \\[10pt]
        \hline
        $Q_{1,1,1,3}$ & $\displaystyle \int_0^t \frac{(t-s)^3}{3!} \frac{u_1^4(s)}{4!} \dd s$ & minor
        \\[10pt]
        \hline
        $Q_{1,2,2,1}$ & $\displaystyle \int_0^t (t-s) \frac{u_1^2(s)}{2} \frac{u_2^2(s)}{2} \dd s$ & minor
        \\[10pt]
        \hline
        $Q^\flat_{1,0,1}$ & $\displaystyle \int_0^t (t-s) \frac 1 2 \left( \int_0^s \frac{u_1^2}{2} \right)^2 \dd s$ & minor
        \\[10pt]
        \hline
        $Q^\sharp_{1,0,2,0}$ & $\displaystyle \int_0^t \left(\int_0^s \frac{u_1^2(s')}{2}\right) \frac{u_2^2(s)}{2} \dd s$ & minor
        \\[10pt]
        \hline
    \end{tabular}
    \caption{Elements of $\Bs_{4,4}$}
    \label{tab:S44}
\end{table}

\begin{table}[!ht]
    \centering
    \setlength{\extrarowheight}{10pt}
    \setlength{\tabcolsep}{10pt}
    \begin{tabular}{|m{2cm}|m{6cm}|m{2cm}|}
        \hline
        $b$ & $\xi_{b}(t,u)$ & Nature \\[10pt]
        \hline
        $Q_{1,1,2,3}$ & $\displaystyle \int_0^t \frac{(t-s)^3}{3!} \frac{u_1^3(s)}{3!} u_2(s) \dd s$ & good
        \\[10pt]
        \hline
        $Q_{1,1,3,2} $ & $\displaystyle \int_0^t \frac{(t-s)^2}{2!} \frac{u_1^3(s)}{3!} u_3(s) \dd s$ & good
        \\[10pt]
        \hline
        $Q_{1,1,4,1} $ & $\displaystyle \int_0^t (t-s) \frac{u_1^3(s)}{3!} u_4(s) \dd s$ & good
        \\[10pt]
        \hline
        $Q_{1,1,5,0} $ & $\displaystyle \int_0^t \frac{u_1^3(s)}{3!} u_5(s) \dd s$ & good
        \\[10pt]
        \hline
        $Q_{1,2,3,1}$ & $\displaystyle \int_0^t (t-s) \frac{u_1^2(s)}{2} u_2(s) u_3(s) \dd s$ & good
        \\[10pt]
        \hline
        $Q_{1,2,4,0}$ & $\displaystyle \int_0^t \frac{u_1^2(s)}{2} u_2(s) u_4(s) \dd s$ & good
        \\[10pt]
        \hline
        $Q_{1,3,3,0}$ & $\displaystyle \int_0^t \frac{u_1^2(s)}{2} \frac{u_3^2(s)}{2} \dd s$ & bad
        \\[10pt]
        \hline
        $Q_{2,2,2,0}$ & $\displaystyle \int_0^t \frac{u_2^4(s)}{4!} \dd s$ & bad
        \\[10pt] 
        \hline
        $Q_{1,1,1,4}$ & $\displaystyle \int_0^t \frac{(t-s)^4}{4!} \frac{u_1^4(s)}{4!} \dd s$ & minor
        \\[10pt]
        \hline
        $Q_{1,2,2,2}$ & $\displaystyle \int_0^t \frac{(t-s)^2}{2} \frac{u_1^2(s)}{2} \frac{u_2^2(s)}{2} \dd s$ & minor
        \\[10pt]
        \hline
        $Q^\flat_{1,0,2}$ & $\displaystyle \int_0^t \frac{(t-s)^2}{2} \frac 1 2 \left( \int_0^s \frac{u_1^2}{2} \right)^2 \dd s$ & minor
        \\[10pt]
        \hline
        $Q^\flat_{1,1,0}$ & $\displaystyle \int_0^t \frac 1 2 \left(\int_0^s (s-s') \frac{u_1^2(s')}{2}\right)^2 \dd s $ & minor
        \\[10pt]
        \hline
        $Q^\sharp_{1,0,2,1}$ & $\displaystyle \int_0^t (t-s) \left(\int_0^s \frac{u_1^2(s')}{2}\right) \frac{u_2^2(s)}{2} \dd s$ & minor
        \\[10pt]
        \hline
        $Q^\sharp_{1,1,2,0}$ & $\displaystyle \int_0^t \left(\int_0^s (s-s') \frac{u_1^2(s')}{2}\right) \frac{u_2^2(s)}{2} \dd s$ & minor
        \\[10pt]
        \hline
    \end{tabular}
    \caption{Elements of $\Bs_{4,5}$}
    \label{tab:S45}
\end{table}

\clearpage

\section*{Acknowledgments}

The authors acknowledge support from grants ANR-20-CE40-0009 and ANR-11-LABX-0020, as well as from the Fondation Simone et Cino Del Duca -- Institut de France.

\bibliographystyle{plain}
\bibliography{control}

@article{AgrachevGamkrelidze1993_Semigroups,
	title = {Local controllability and semigroups of diffeomorphisms},
	author = {Agrachev, Andrei and Gamkrelidze, Revaz},
	journal = {Acta Applicandae Mathematica},
	volume = {32},
	pages = {1--57},
	year = {1993},
	publisher = {Kluwer Academic Publishers}
}

@article{BeauchardLeBorgneMarbach2022,
	author = {Karine Beauchard and Jérémy {Le Borgne} and Frédéric Marbach},
	title = {{Growth of structure constants of free Lie algebras relative to Hall bases}},
	journal = {Journal of Algebra},
	volume = {612},
	pages = {281-378},
	year = {2022},
	issn = {0021-8693},
	doi = {https://doi.org/10.1016/j.jalgebra.2022.08.030},
	url = {https://www.sciencedirect.com/science/article/pii/S0021869322004264}
}

@article{BeauchardLeBorgneMarbach2023,
	author = {Karine Beauchard and Jérémy {Le Borgne} and Frédéric Marbach},
	title = {On expansions for nonlinear systems, error estimates and convergence issues},
	journal = {Comptes Rendus. Math\'ematique},
	pages = {97--189},
	publisher = {Acad\'emie des sciences, Paris},
	volume = {361},
	year = {2023},
	doi = {10.5802/crmath.395},
	language = {en},
	url = {https://comptes-rendus.academie-sciences.fr/mathematique/articles/10.5802/crmath.395/}
}

@article{BeauchardMarbach2018,
	author = {Beauchard, Karine and Marbach, Fr\'{e}d\'{e}ric},
	title = {Quadratic obstructions to small-time local controllability for scalar-input systems},
	journal = {J. Differential Equations},
	fjournal = {Journal of Differential Equations},
	volume = {264},
	year = {2018},
	number = {5},
	pages = {3704--3774},
	issn = {0022-0396},
	mrclass = {93B05 (93C10 93C15)},
	mrnumber = {3741402},
	mrreviewer = {Matthias Kawski},
	doi = {10.1016/j.jde.2017.11.028},
	url = {https://doi.org/10.1016/j.jde.2017.11.028}
}

@article{BeauchardMarbach2020,
	title = {Unexpected quadratic behaviors for the small-time local null controllability of scalar-input parabolic equations},
	author = {Beauchard, Karine and Marbach, Fr{\'e}d{\'e}ric},
	journal = {Journal de Math{\'e}matiques Pures et Appliqu{\'e}es},
	volume = {136},
	pages = {22--91},
	year = {2020},
	publisher = {Elsevier},
	doi = {10.1016/j.matpur.2020.02.001},
	url = {https://doi.org/10.1016/j.matpur.2020.02.001}
}

@article{BeauchardMarbach2026,
	author = {Beauchard, Karine and Marbach, Fr{\'e}d{\'e}ric},
	title = {A unified approach of obstructions to small-time local controllability for scalar-input systems},
	journal = {Journal of Dynamical and Control Systems},
	issn = {1079-2724},
	volume = {32},
	number = {1},
	pages = {95},
	year = {2026},
	language = {English},
	doi = {10.1007/s10883-025-09752-1},
	zbmath = {8144822}
}

@misc{BeauchardMarbachPerrin2025,
	title = {Small-time local control of a Schr{\"o}dinger equation: a negative and a positive quadratic result},
	author = {Beauchard, Karine and Marbach, Fr{\'e}d{\'e}ric and Perrin, Thomas},
	year = {2025},
	howpublished = {{arXiv:2501.03882}}
}

@inproceedings{BianchiniStefani1986,
	title = {Sufficient conditions of local controllability},
	author = {Bianchini, Rosa Maria and Stefani, Gianna},
	booktitle = {1986 25th IEEE Conference on Decision and Control},
	pages = {967--970},
	year = {1986},
	organization = {IEEE}
}

@article{BoscainCannarsaFranceschiSigalotti2023,
	author = {Boscain, Ugo and Cannarsa, Daniele and Franceschi, Valentina and Sigalotti, Mario},
	title = {Local controllability does imply global controllability},
	fjournal = {Comptes Rendus. Math{\'e}matique. Acad{\'e}mie des Sciences, Paris},
	journal = {C. R., Math., Acad. Sci. Paris},
	issn = {1631-073X},
	volume = {361},
	pages = {1813--1822},
	year = {2023},
	language = {English},
	doi = {10.5802/crmath.538},
	zbmath = {7811841},
	zbl = {1533.57065}
}

@article{Bournissou2024,
	author = {Bournissou, M{\'e}gane},
	title = {Small-time local controllability of the bilinear {Schr{\"o}dinger} equation with a nonlinear competition},
	fjournal = {European Series in Applied and Industrial Mathematics (ESAIM): Control, Optimization and Calculus of Variations},
	journal = {ESAIM, Control Optim. Calc. Var.},
	issn = {1292-8119},
	volume = {30},
	pages = {38},
	note = {Id/No 2},
	year = {2024},
	language = {English},
	doi = {10.1051/cocv/2023077},
	zbmath = {7798860},
	zbl = {1530.93024},
	url = {https://doi.org/10.1051/cocv/2023077}
}

@misc{Casselman2020_Free,
	author = {Casselmann, Bill},
	year = {2020},
	title = {{Free Lie algebras}},
	howpublished = {"\url{https://secure.math.ubc.ca/~cass/research/pdf/Free.pdf}"},
	note = {"Last accessed 2021-07-17"}
}

@article{Coron2006,
	author = {Coron, Jean-Michel},
	title = {On the small-time local controllability of a quantum particle in a moving one-dimensional infinite square potential well},
	journal = {C. R. Math. Acad. Sci. Paris},
	fjournal = {Comptes Rendus Math\'ematique. Acad\'emie des Sciences. Paris},
	volume = {342},
	year = {2006},
	number = {2},
	pages = {103--108},
	issn = {1631-073X},
	mrclass = {93B05 (35Q40 81Q99)},
	mrnumber = {2193655},
	doi = {10.1016/j.crma.2005.11.004},
	url = {https://doi.org/10.1016/j.crma.2005.11.004},
	zbmath = {2248350},
	zbl = {1082.93002}
}

@book{Coron2007,
	author = {Coron, Jean-Michel},
	title = {Control and nonlinearity},
	series = {Mathematical Surveys and Monographs},
	volume = {136},
	publisher = {American Mathematical Society, Providence, RI},
	year = {2007},
	pages = {xiv+426},
	isbn = {978-0-8218-3668-2; 0-8218-3668-4},
	mrclass = {93-02 (35Q30 35Q53 35Q55 93C20)},
	mrnumber = {2302744},
	mrreviewer = {Vilmos Komornik},
	zbl = {1140.93002}
}

@article{CoronKoenigNguyen2024,
	author = {Coron, Jean-Michel and Koenig, Armand and Nguyen, Hoai-Minh},
	title = {Lack of local controllability for a water-tank system when the time is not large enough},
	fjournal = {Annales de l'Institut Henri Poincar{\'e} C. Analyse Non Lin{\'e}aire},
	journal = {Ann. Inst. Henri Poincar{\'e} C, Anal. Non Lin{\'e}aire},
	issn = {0294-1449},
	volume = {41},
	number = {6},
	pages = {1327--1365},
	year = {2024},
	language = {English},
	doi = {10.4171/AIHPC/123},
	zbmath = {7930628},
	zbl = {1553.93018},
	url = {https://doi.org/10.4171/AIHPC/123}
}

@article{Gagliardo1959,
	author = {Gagliardo, Emilio},
	title = {Ulteriori proprieta di alcune classi di funzioni in piu variabili},
	fjournal = {Ricerche di Matematica},
	journal = {Ric. Mat.},
	issn = {0035-5038},
	volume = {8},
	pages = {24--51},
	year = {1959},
	language = {Italian},
	zbmath = {3318089},
	zbl = {0199.44701}
}

@article{Gherdaoui2025,
	author = {Gherdaoui, Th{\'e}o},
	title = {Small-time local controllability of the multi-input bilinear {Schr{\"o}dinger} equation thanks to a quadratic term},
	fjournal = {European Series in Applied and Industrial Mathematics (ESAIM): Control, Optimization and Calculus of Variations},
	journal = {ESAIM, Control Optim. Calc. Var.},
	issn = {1292-8119},
	volume = {31},
	pages = {49},
	note = {Id/No 44},
	year = {2025},
	language = {English},
	doi = {10.1051/cocv/2024091},
	zbmath = {8056732},
	url = {https://doi.org/10.1051/cocv/2024091}
}

@article{Hermes1982,
	title = {{Control systems which generate decomposable Lie algebras}},
	author = {Hermes, Henry},
	journal = {Journal of Differential Equations},
	volume = {44},
	number = {2},
	pages = {166--187},
	year = {1982},
	publisher = {Elsevier}
}

@article{Kalman1960,
	title = {Contributions to the theory of optimal control},
	author = {Kalman, Rudolf},
	journal = {Boletin de la Sociedad Matematica Mexicana},
	volume = {5},
	number = {2},
	pages = {102--119},
	year = {1960}
}

@book{Kawski1986,
	author = {Kawski, Matthias},
	title = {Nilpotent {L}ie algebras of vectorfields and local controllability of nonlinear systems},
	note = {Thesis (Ph.D.)--University of Colorado at Boulder},
	publisher = {ProQuest LLC, Ann Arbor, MI},
	year = {1986},
	pages = {104},
	mrclass = {Thesis},
	mrnumber = {2635388},
	url = {http://gateway.proquest.com/openurl?url_ver=Z39.88-2004&rft_val_fmt=info:ofi/fmt:kev:mtx:dissertation&res_dat=xri:pqdiss&rft_dat=xri:pqdiss:8706431}
}

@article{Kawski1987_Necessary,
	title = {A necessary condition for local controllability},
	author = {Kawski, Matthias},
	journal = {Contemporary Mathematics},
	volume = {68},
	pages = {143--155},
	year = {1987}
}

@incollection{Kawski1987_Survey,
	title = {High-order small-time local controllability},
	author = {Kawski, Matthias},
	booktitle = {Nonlinear controllability and optimal control},
	pages = {431--467},
	year = {1987},
	publisher = {Routledge}
}

@article{Kawski2002_Coordinates,
	title = {Controllability and coordinates of the first kind},
	author = {Kawski, Matthias},
	journal = {Contemporary Trends In Nonlinear Geometric Control Theory And Its Applications},
	pages = {381--403},
	year = {2002},
	publisher = {World Scientific}
}

@article{Krastanov2009,
	title = {A sufficient condition for small-time local controllability},
	author = {Krastanov, Mikhail},
	journal = {SIAM Journal on Control and Optimization},
	volume = {48},
	number = {4},
	pages = {2296--2322},
	year = {2009},
	publisher = {SIAM}
}

@article{Krener1973,
	author = {Arthur {Krener}},
	title = {{On the equivalence of control systems and the linearization of nonlinear systems}},
	fjournal = {{SIAM Journal on Control}},
	journal = {{SIAM J. Control}},
	issn = {0036-1402},
	volume = {11},
	pages = {670--676},
	year = {1973},
	publisher = {Society for Industrial \& Applied Mathematics, Philadelphia},
	language = {English},
	msc2010 = {93C10},
	zbl = {0243.93009}
}

@incollection{Lasalle1960,
	author = {LaSalle, Joseph},
	title = {The time optimal control problem},
	booktitle = {Contributions to the theory of nonlinear oscillations},
	volume = {5},
	pages = {1--24},
	publisher = {Princeton Univ. Press, Princeton, N.J.},
	year = {1960},
	mrclass = {34.85},
	mrnumber = {0145169},
	mrreviewer = {H. Kaufman}
}

@misc{Lazard1960,
	author = {Michel {Lazard}},
	title = {{Groupes, anneaux de Lie et probl\`eme de Burnside}},
	year = {1960},
	language = {French},
	howpublished = {{C.I.M.E., Gruppi, Anelli di Lie e Teoria della Coomologia}},
	zbl = {0134.26003}
}

@book{Lee2013,
	author = {John {Lee}},
	title = {{Introduction to smooth manifolds}},
	fjournal = {{Graduate Texts in Mathematics}},
	journal = {{Grad. Texts Math.}},
	issn = {0072-5285},
	volume = {218},
	edition = {2nd revised},
	isbn = {978-1-4419-9981-8/hbk; 978-1-4419-9982-5/ebook},
	pages = {xvi + 708},
	year = {2013},
	publisher = {New York, NY: Springer},
	language = {English},
	msc2010 = {53-01 53-02 58-02 57-02 53Cxx 57Rxx 58Axx},
	zbl = {1258.53002}
}

@article{Marbach2023,
	title = {A family of interpolation inequalities involving products of low-order derivatives},
	author = {Marbach, Fr{\'e}d{\'e}ric},
	journal = {arXiv:2306.06668},
	year = {2023}
}

@article{Moree2005,
	author = {Moree, Pieter},
	title = {The formal series {W}itt transform},
	journal = {Discrete Math.},
	fjournal = {Discrete Mathematics},
	volume = {295},
	year = {2005},
	number = {1-3},
	pages = {143--160},
	issn = {0012-365X,1872-681X},
	mrclass = {05A19 (05A15 11B75)},
	mrnumber = {2143453},
	mrreviewer = {Francesco\ Pappalardi},
	doi = {10.1016/j.disc.2005.03.004},
	url = {https://doi.org/10.1016/j.disc.2005.03.004}
}

@article{Nirenberg1959,
	author = {Nirenberg, Louis},
	title = {On elliptic partial differential equations},
	journal = {Ann. Scuola Norm. Sup. Pisa (3)},
	volume = {13},
	year = {1959},
	pages = {115--162},
	mrclass = {35.00},
	mrnumber = {0109940},
	mrreviewer = {L. Garding}
}

@misc{NiuXiang2025,
	title = {{Small-time local controllability of a KdV system for all critical lengths}},
	author = {Jingrui Niu and Shengquan Xiang},
	year = {2025},
	eprint = {2501.13640},
	archiveprefix = {arXiv},
	primaryclass = {math.AP},
	url = {https://arxiv.org/abs/2501.13640}
}

@book{Reutenauer1993,
	author = {Reutenauer, Christophe},
	title = {Free {L}ie algebras},
	series = {London Mathematical Society Monographs. New Series},
	volume = {7},
	publisher = {The Clarendon Press, Oxford University Press, New York},
	year = {1993},
	pages = {xviii+269},
	isbn = {0-19-853679-8},
	mrclass = {17-02 (05-02 17B05)},
	mrnumber = {1231799},
	mrreviewer = {Hartmut Laue}
}

@incollection{Reutenauer2003,
	author = {Reutenauer, Christophe},
	title = {{Free Lie algebras}},
	booktitle = {{Handbook of algebra. Volume 3}},
	isbn = {0-444-51264-0/hbk},
	pages = {887--903},
	year = {2003},
	publisher = {Amsterdam: Elsevier},
	language = {English},
	msc2010 = {17B01},
	zbl = {1071.17003}
}

@incollection{Stefani1986,
	author = {Stefani, Gianna},
	title = {On the local controllability of a scalar-input control system},
	booktitle = {Theory and applications of nonlinear control systems ({S}tockholm, 1985)},
	pages = {167--179},
	publisher = {North-Holland, Amsterdam},
	year = {1986},
	mrclass = {49E15 (93B05)},
	mrnumber = {935375}
}

@article{Sussmann1983,
	author = {Sussmann, H\'ector},
	title = {Lie brackets and local controllability: a sufficient condition for scalar-input systems},
	journal = {SIAM J. Control Optim.},
	fjournal = {SIAM Journal on Control and Optimization},
	volume = {21},
	year = {1983},
	number = {5},
	pages = {686--713},
	issn = {0363-0129},
	mrclass = {49E15 (58F40 93B15)},
	mrnumber = {710995},
	mrreviewer = {Henry Hermes},
	doi = {10.1137/0321042},
	url = {http://dx.doi.org/10.1137/0321042}
}

@incollection{Sussmann1986,
	author = {Sussmann, H\'ector},
	title = {A product expansion for the {C}hen series},
	booktitle = {Theory and applications of nonlinear control systems ({S}tockholm, 1985)},
	pages = {323--335},
	publisher = {North-Holland, Amsterdam},
	year = {1986},
	mrclass = {93B27},
	mrnumber = {935387},
	doi = {10.1177/000992288602500608},
	url = {https://doi.org/10.1177/000992288602500608}
}

@article{Sussmann1987,
	author = {Sussmann, H\'ector},
	title = {A general theorem on local controllability},
	journal = {SIAM J. Control Optim.},
	fjournal = {SIAM Journal on Control and Optimization},
	volume = {25},
	year = {1987},
	number = {1},
	pages = {158--194},
	issn = {0363-0129},
	mrclass = {93B05 (93B27)},
	mrnumber = {872457},
	mrreviewer = {P. Brunovsk\'{y}},
	doi = {10.1137/0325011},
	url = {https://doi.org/10.1137/0325011}
}

@book{Viennot1978,
	author = {Viennot, G\'{e}rard},
	title = {Alg\`ebres de {L}ie libres et mono\"{\i}des libres},
	series = {Lecture Notes in Mathematics},
	volume = {691},
	note = {Bases des alg\`ebres de Lie libres et factorisations des mono\"{\i}des libres},
	publisher = {Springer, Berlin},
	year = {1978},
	pages = {ii+124},
	isbn = {3-540-09090-8},
	mrclass = {17B65},
	mrnumber = {516004},
	mrreviewer = {Juri A. Bahturin}
}

@article{Witt1956,
	author = {Witt, Ernst},
	title = {Die {U}nterringe der freien {L}ieschen {R}inge},
	journal = {Math. Z.},
	fjournal = {Mathematische Zeitschrift},
	volume = {64},
	year = {1956},
	pages = {195--216},
	issn = {0025-5874,1432-1823},
	mrclass = {09.3X},
	mrnumber = {77525},
	mrreviewer = {C.\ W.\ Curtis},
	doi = {10.1007/BF01166568},
	url = {https://doi.org/10.1007/BF01166568}
}

@article{BeauchardLeBorgneMarbach2026,
  title={{Convergent realizations of Lie subalgebras}},
  author={Beauchard, Karine and Borgne, J{\'e}r{\'e}my Le and Marbach, Fr{\'e}d{\'e}ric},
  journal={arXiv:2607.06490},
  year={2026}
}

@misc{Marbach2026,
 author = {Fr{\'e}d{\'e}ric Marbach},
 title = {Time-iteration methods for controllability},
 year = {2026},
 howpublished = {Preprint, {arXiv}:2602.19272 [math.{OC}] (2026)},
 url = {https://arxiv.org/abs/2602.19272},
 arXiv = {arXiv:2602.19272}
}

@inproceedings{Beauchard2026,
    title={{Quadratic Terms, Lie Brackets, and Local Controllability}},
    author={Beauchard, Karine},
    booktitle={International Congress of Mathematicians 2026},
    pages={153--172},
    year={2026},
    organization={SIAM},
    doi = {10.1137/25M1804960},
    URL = {https://epubs.siam.org/doi/abs/10.1137/25M1804960},
    eprint = {https://epubs.siam.org/doi/pdf/10.1137/25M1804960},
}

\end{document}